%% file: dissertation.tex
\documentclass[dissertation,actual]{uhthesis}
\input{packages}

\input{COMMANDS}

\begin{document}
\title{On New Notions of Algorithmic Dimension, Immunity, and Medvedev Degree}
\author{David J. Webb}
\degreemonth{August}
\degreeyear{2022}
\degree{Doctor of Philosophy}
\chair{Bjorn Kjos-Hanssen}
\othermembers{Monique Chyba\\
Ruth Haas\\
Michelle Manes\\
David Ross\\
Michelle Seidel}
\numberofmembers{6}
\field{Mathematics}
\versionnum{0.0.1}

\maketitle

\begin{frontmatter}
\signaturepage 
\copyrightpage
\dedication{\emph{For my mother, and hers.}}


\begin{acknowledgments}
It is my distinct pleasure to acknowledge and thank: \begin{itemize}
\item[-] My advisor, Bj\o rn Kjos-Hanssen, for his endless patience and careful tutelage. It was a tremendous stroke of good fortune that I arrived at the University when I did, to have been able to take my first courses in set theory and logic with you. They set me on the path through the Garden of Logic whose culmination - for the moment - is this document.

\item[-] Rachael Alvir, for helpful conversations about the Ershov Hierarchy (and for giving me my favorite story).

\item[-] My friends and peers Sam Birns and Carl Eadler, conversations with whom clarified my thinking about randomness and truth-table reductions, respectively.

\item[-] My mentors Jamie Pommersheim (\emph{``1. Gil is a Pavel\dots"}) and Jerry Shurman (\emph{``To do mathematics effectively we should characterize our objects
rather than construct them."}), for the fine examples they set when I was a young mathematician. I am indebted to you both, for those quotes that live rent-free in my head, and for so much more. 

\item[-] The computability theory community, who have been nothing but welcoming. Special thanks to Mushfeq Khan, for teaching an excellent first course in the subject, and to Dan Turetsky, for attending my first public research talk in 2019 and solving the conjecture I posed there.

\item[-] The choral community on O`ahu, in particular the directors I've had the pleasure to sing under.

\item[-] Patchen Mortimer, whose radio show has been the soundtrack of this endeavor for years.

\item[-] Andrew and Amber Baker, Alex Char, Arc Chase, Aaron Do, Shane and Nolan Endicott, Candice Ferge, Aaron King, Scott Peiterson, Lion Pirsig, Nora Sender, and Chase Yap, whose friendship has kept me sane over the last decade (or more, in some cases). The best ears I've ever been lent have been yours: few folks would be keen to listen to my ramblings, and fewer still would reply with kindness, advice, and other immeasurably useful conversations. Here's to more D\&D and curry in days to come.

\item[-] And last but certainly not least, Alan Turing and Kurt G\"odel. Though I certainly couldn't hope to fill your shoes, following in your footsteps has been one of the principal joys of my life.
\end{itemize}
\end{acknowledgments}

\begin{abstract}
We prove various results connected together by the common thread of computability theory.

First, we investigate a new notion of algorithmic dimension, the inescapable dimension, which lies between the effective Hausdorff and packing dimensions. We also study its generalizations, obtaining an embedding of the Turing degrees into notions of dimension.

We then investigate a new notion of computability theoretic immunity that arose in the course of the previous study, that of a set of natural numbers with no co-enumerable subsets. We demonstrate how this notion of $\Pi^0_1$-immunity is connected to other immunity notions, and construct $\Pi^0_1$-immune reals throughout the high/low and Ershov hierarchies. We also study those degrees that cannot compute or cannot co-enumerate a $\Pi^0_1$-immune set.

Finally, we discuss a recently discovered truth-table reduction for transforming a Kolmogorov--Loveland random input into a Martin-L\"of random output by exploiting the fact that at least one half of such a KL-random is itself ML-random. We show that there is no better algorithm relying on this fact, i.e.\ there is no positive, linear, or bounded truth-table reduction which does this. We also generalize these results to the problem of outputting randomness from infinitely many inputs, only some of which are random.

\end{abstract}

\tableofcontents
\listoffigures
\end{frontmatter}
\renewcommand{\null}{\emptyset}
\chapter{Introduction}\label{ch:intro}
\input{1intro}
\input{1notation}
\chapter{Arithmetical Oracle Dimensions}\label{ch:dim}
\input{2oracledimensions}
\input{2properties}
\section{The Recursively Bounded $\Pi^0_1$ Case}
\input{2rbpi}
\chapter{$\Pi^0_1$-Immunity}\label{ch:pi}
\input{3intro}
\input{3piimmune}
\input{3HighLowErshov}
\input{3BN}
\input{3lowness}
\chapter{When You Have Two Hammers and One of Them Works}\label{ch:either}
\input{4Either}


\nocite{*}
\bibliographystyle{plain}
\bibliography{dissertation}


\end{document}

%% file: packages.tex
\usepackage{amsmath, amssymb,amsfonts,amsthm, stmaryrd, tikz, tikz-cd}
\usepackage[pagebackref=true,colorlinks]{hyperref}
\hypersetup{
    colorlinks=true,
    linkcolor=blue,
    citecolor=blue,
    }
\usepackage[nottoc]{tocbibind}
\usepackage{backref}

\renewcommand*{\backrefalt}[4]{%
    \ifcase #1 ()%
    \or        (Page~#2)%
    \else      (Pages~#2)%
    \fi}

\newif\ifbackrefshowonlyfirst
\backrefshowonlyfirstfalse
\makeatletter
\let\BR@direct@old@hyper@natlinkstart\hyper@natlinkstart
\renewcommand*{\hyper@natlinkstart}{\phantomsection\BR@direct@old@hyper@natlinkstart}
\let\BR@direct@oldBR@citex\BR@citex
\renewcommand*{\BR@citex}{\phantomsection\BR@direct@oldBR@citex}%

\long\def\hyper@page@BR@direct@ref#1#2#3{\hyperlink{#3}{#1}}

\ifx\backrefxxx\hyper@page@backref
    \let\backrefxxx\hyper@page@BR@direct@ref
    \ifbackrefshowonlyfirst
    \fi
\else
    \ifbackrefshowonlyfirst
    \fi
\fi

\RequirePackage{etoolbox}
\patchcmd{\Hy@backout}{Doc-Start}{\@currentHref}{}{\errmessage{I can't seem to patch backref}}
\makeatother

\usepackage{thmtools}

\usepackage{thm-restate}

\usepackage{cleveref}

\usepackage[all,curve]{xy}

%% file: COMMANDS.tex
\theoremstyle{definition}
\newtheorem{theorem}{Theorem}[section]

\newtheorem{definition}[theorem]{Definition}

\newtheorem{lemma}[theorem]{Lemma}
\newtheorem{conjecture}[theorem]{Conjecture}
\newtheorem{corollary}[theorem]{Corollary}

 \newtheorem*{claim*}{Claim}
\newtheorem*{remark}{Remark}

\newtheoremstyle{OldTheorem}
    {\topsep}{\topsep}  
    {}          
    {}          
    {\bfseries}         
    {.}                 
    { }                 
    {\thmname{#1}\thmnote{ \bfseries #3}}
\theoremstyle{OldTheorem}

\newtheorem{oldlemma}{Lemma}
\newtheorem{olddef}{Definition}

\newcommand{\ra}{\rightarrow}
\newcommand{\Ra}{\Rightarrow}
\newcommand{\La}{\Leftarrow}

\newcommand{\Lra}{\Leftrightarrow}

\newcommand{\eps}{\varepsilon}
\newcommand{\N}{\mathbb{N}}

\newcommand{\Q}{\mathbb{Q}}
\newcommand{\B}{\mathfrak{B}}
\newcommand{\C}{\mathfrak{C}}

\newcommand{\Cant}{2^\omega}

\newcommand{\upto} {{\upharpoonright}}

\newcommand{\MLR}{Martin-L\"of random}

\newcommand{\rbpi}{\widehat{\Pi}^0_1}

\newcommand{\abs}[1]{\lvert#1\rvert}

\newcommand{\halts}{\!\!\downarrow}
\newcommand{\divs}{\!\!\uparrow}

\DeclareMathOperator{\Either}{Either}
\DeclareMathOperator{\KLR}{KLR}
\DeclareMathOperator{\DIM}{DIM}

\DeclareMathOperator{\Some}{Some}

\DeclareMathOperator{\Many}{Many}

%% file: 1intro.tex
Computability theory is concerned with the computational strength of mathematical objects, usually viewed as infinite sequences of $0$s and $1$s (also called \emph{reals}). For instance, given a listing of all computer programs, consider the sequence $K\in 2^\omega$ such that $K(e) = 1$ iff the $e$th program will halt, and $0$ otherwise. Alan Turing famously showed that $K$ is not computable --- one must either prove that the program in question will (not) halt, or run it and hope that it does. But as an object unto itself, we can ask many questions about $K$: what else can it compute? Is any regularity to which of its entries are $1$? In computability theory, we seek to answer such questions, and to develop the necessary tools to do so.

One such tool that has been very effective in this study is Kolmogorov complexity. For a finite string $\sigma\in2^{<\omega}$, $K(\sigma)$\footnote{It is an unfortunate notational collision that $K$ is both the halting problem, an infinite binary sequence, and Kolmogorov complexity, a function from finite strings to naturals.} is (essentially) the length of the shortest program whose output is $\sigma$. This allows for elegant characterizations of randomness --- for instance, a random sequence should be as difficult to describe as possible, and so all of its initial segments should have high Kolmogorov complexity. 

This leads naturally to `effective' versions of fractal dimensions from geometry. For instance the effective packing dimension of $X\in2^\omega$ is $$\dim_p(X) =\limsup_{n\in\N} \dfrac{K(X\upto n)}{n}$$
where $X\upto n$ is the first $n$ bits of $X$. In \Cref{ch:dim} we investigate a modification of this, the \emph{inescapable} dimension, where one takes infimums of supremums over computable ($\Delta^0_1$) sets of natural numbers:
$$\dim_i(X) = \inf_{N\in\Delta^0_1}\sup_{n\in N} \dfrac{K(X\upto n)}{n}.$$

We then generalize further by considering oracles, i.e.\ non-computable reals. With an oracle $A$ in hand, we can consider the $\Delta^0_1(A)$ sets, i.e.\ those computed by some program with access to the non-computable information contained in $A$. For instance, the halting problem $K$ can compute a random sequence, which is necessarily not $\Delta^0_1$. Each oracle thus corresponds to a notion of dimension, and we ultimately obtain an embedding theorem between the Turing degrees and the $\Delta^0_1(A)$ dimensions. We also prove corresponding results for generalizations of the complex packing dimension, which was defined in \cite{CompPack}.

In addition to what can be computed (possibly by an oracle), computability theory is also concerned with weaker notions of computation. A set $W$ is enumerable (or $\Sigma^0_1$) if there is an algorithm which lists its members in some order --- if $n\in W$, we will eventually know it, but until the program enumerates $n$, we can never be sure. Classical computability theory has much to say about \emph{immune} reals, those with no enumerable subsets (again, random sets provide an easy example).

In \Cref{ch:pi}, we explore a related notion that arose in the course of studying the $\Pi^0_1(A)$ dimensions: that of a $\Pi^0_1$-immune set (the $\Pi^0_1$ sets are \emph{co}-enumerable, their complements are $\Sigma^0_1$). This notion appears (though is not studied unto itself) in \cite{MinInd1, MinInd2, Teutsch}, in connection with sets of minimal indices. We explore connections between this notion and previously studied immunity notions in classical computability theory, and construct $\Pi^0_1$-immune sets that fit into various computability theoretic hierarchies. Finally, we study the classes of sets that compute or co-enumerate no $\Pi^0_1$-immune sets, and make connections with notions of computational weakness. In the course of doing so we obtain a result of independent interest, that the class of hyperimmune-free sets coincides with those that compute no truth-table CEA set.

Finally in \Cref{ch:either}, we shine a small light on one of the biggest open questions in algorithmic randomness: the Kolmogorov--Loveland randomness problem. While it is known that Martin-L\"of random (MLR) sequences are Kolmogorov--Loveland random (KLR), the reverse implication remains open. Several partial results are known; for instance, if some $X\in\mathrm{KLR}$ is decomposed into its even and odd entries $X_0$ and $X_1$, at least one $X_i$ is Martin-L\"of random \cite{MR2183813}. This gives a weak equivalence between the notions, as the reduction from KLR to MLR is non-uniform. Miyabe asked if this could be strengthened to a uniform reduction \cite{Miyabe}, and this was answered in the affirmative in \cite{review}. Here we prove that this reduction is in a sense optimal for the following problem: what kind of algorithm suffices to always output randomness given two inputs, an unknown one of which is random? We also generalize this result to the setting of infinitely many inputs, an unknown one of which is known to be random.

Material in Chapters \ref{ch:dim} and \ref{ch:either} previously appeared in proceedings in \emph{Computability in Europe} in 2021 and 2022, respectively \cite{KHW,CiE2022}.

%% file: 1notation.tex
\section{Notation and Preliminaries}

Our notation follows the standard texts in the fields of computability theory \cite{Soare, SoareTuring}  and algorithmic randomness \cite{DaH, Nies}.

The natural numbers are denoted $\omega$, and contain $0$. We will often make use of $\langle\cdot, \cdot \rangle$, a fixed computable bijective pairing function from $\omega^2$ to $\omega$. Any such function should suffice, but to be explicit, we use the Cantor pairing function: $\langle x, y\rangle = \frac{1}{2}(x+y)(x+y+1)+y$.

Strings are functions $\sigma:\{0, 1, \dots, n-1\}\ra\{0, 1\}$, while reals are functions $A:\omega\ra\{0, 1\}$ (in analogy to binary expansions of real numbers in $[0,1]$). They are generally denoted by lowercase Greek and capital Latin letters, respectively. We say $\sigma\preceq\tau$ if, as sets of ordered pairs, $\sigma\subseteq\tau$, and similarly for $\sigma\prec A$. The set of all strings of length $n$ is $2^n$, and the set of all strings of any length is $2^{<\omega}$. The set of all reals is $2^\omega$.

It is often convenient to write strings and reals as binary sequences, e.g.\ $\sigma = 010$. In this view, each element of the sequence is a \emph{bit}. We will write $\sigma i$ or $\sigma^\frown i$ to mean the sequence $\sigma$ with the bit $i$ appended. Other strings or a real may be appended this way as well. The sequence of the first $n$ bits of a real $A$ is written $A\upto n$.

We also use $A$ to denote $\{n\mid A(n) = 1\}$. We denote complements with overlines, with the ambient set as $\omega$ or $2^\omega$ taken to be clear from context. The complement of a real $A$ is $\overline{A}=\{n\mid A(n) = 0\}$.

Partial computable functions are indexed by $e\in\omega$ as $\varphi_e$, and their domains as $W_e$. These functions are represented by Turing machines, which compute in steps $s\in\omega$. We can similarly define partial functions $\varphi_{e, s}$ and $W_{e, s}$ by only running $\varphi_e$ for $s$ steps on inputs $n\leq s$. We say $\varphi_{e,s}(n)\halts$ (the computation halts) if there is a stage when the computation $\varphi_{e, s}(n)$ has halted, and $\varphi_{e}(n)\divs$ if there is no such stage (the computation diverges). 

We often view $\varphi_e$ as enumerating a list of elements --- we imagine all computations $\varphi_e(n)$ being run in parallel, with $n$ being added to $W_{e, s}$ when $\varphi_{e,s}(n)$ is defined.

A function $f$ is computable (or recursive) iff there is an $e$ such that $f = \varphi_e$, and $\varphi_e$ is total. In the language of Turing machines, $f$ is computable iff there is a Turing machine $M$ that is guaranteed to halt when run on any natural number input $n$.

A set $X$ is computably enumerable (c.e.) if there is some $e$ for which $X = W_e$. We also say $X$ is $\Sigma^0_1$ if it can be written $X = \{x\in\omega\mid (\exists y\in\omega) R(x, y)\}$, where $R$ is a computable binary predicate. These notions are equivalent.

A set is co-c.e.\ iff its complement is c.e., or equivalently if it can be written using a computable predicate $R$ as $\{x\in\omega\mid (\forall y\in\omega) R(x, y)\}$. We thus write $\Sigma^0_1$ and $\Pi^0_1$ for the sets of c.e.\ and co-c.e.\ reals, respectively, and $\Delta^0_1=\Sigma^0_1\cap\Pi^0_1$ for the computable sets. In general, a $\Sigma^0_n$ set is one that can be written $\{x\in\omega\mid (\exists y_1\in\omega)(\forall y_2\in\omega)\cdots (Q y_n\in\omega)R(x, y_1, y_2, \dots, y_n)\}$ for a computable $n+1$-ary predicate $R$, where $Q$ is $\forall$ if $n$ is even, and $\exists$ otherwise. Similarly a $\Pi^0_n$ set is one that can be written $\{x\in\omega\mid (\forall y_1\in\omega)(\exists y_2\in\omega)\cdots (Q y_n\in\omega)R(x, y_1, y_2, \dots, y_n)\}$. In general $\Delta^0_n = \Sigma^0_n\cap\Pi^0_n$.

We very frequently give our functions access to an oracle $X$, and write $\Phi^X_{e}$ and $W^X_e$ in analogy to $\varphi_e$ and $W_e$. If $\Phi^X_{e}(n)\halts$, its \emph{use} $\varphi^X_{e}(n)$ is the largest bit of $X$ that is queried during the computation.

If for sets $A$ and $B$, there is an $e$ such that $A = \Phi^B_e$, we say that $A$ is Turing reducible to $B$, written $A\leq_T B$. We also say ``$B$ is above $A$" or ``$B$ bounds $A$''. If $B\leq_T A$ as well, we say $A\equiv_T B$. This is an equivalence relation, whose equivalence classes are called Turing degrees.

We describe a Turing degree as having a certain property of reals iff it contains a real with that property, i.e. ``a c.e.\ degree" is one that contains a c.e.\ real.

For $A\in2^\omega$, $\Delta^0_1(A) = \{B\in 2^\omega\mid B\leq_T A\}$ is the set of reals computed by $A$. Similarly $\Sigma^0_1(A) = \{B\in 2^\omega\mid (\exists e\in\omega)\ B = W^A_e\}$, the set of reals enumerated by $A$ (also called ``$A$-c.e." sets). $\Pi^0_1(A)=\{B\in 2^{\omega}\mid \overline{B}\in\Sigma^0_1(A)\}$ is the sets of reals co-enumerated by $A$ (also called ``$A$-co-c.e."). $\Sigma^0_n(A)$ and $\Pi^0_n(A)$ are defined analogously. We may also refer to sets of strings as having an arithmetic complexity by computably encoding strings as natural numbers.

The halting problem is $\null' = \{\langle e, x\rangle \mid \varphi_e(x)\halts\}$. In constructions it will be useful that as an oracle, the halting problem can settle any $\Sigma^0_1$ or $\Pi^0_1$ question.

The jump of $A$ is $A' = \{\langle e, x\rangle \mid \Phi^A_e(x)\halts\}$. Subsequent jumps can be abbreviated $A^{(n)}$.

Post's theorem will often be used without mention: $\Delta^0_n(A) = \{B\in 2^\omega\mid B\leq_T A^{(n-1)}\}$.

When defining an algorithm $\varphi_e$ or $\Phi^X_e$, rather than writing out the precise Turing machine corresponding to our algorithm, we implicitly appeal to the Church--Turing thesis, that any ``effectively calculable" function is describable by a Turing machine. 

%% file: 2oracledimensions.tex
\noindent The first four sections of this chapter previously appeared in print in \cite{KHW}.
\section{The Complex Packing and Inescapable Dimensions}
\label{sec:intro}
Let $K(\sigma)$ denote the prefix-free Kolmogorov complexity of a string $\sigma\in 2^{<\omega}$. We will not consider other variants (such as plain complexity) in the sequel, so for notation we may drop `prefix-free' and/or `Kolmogorov'. While prefix-free Kolmogorov complexity is not computable, it is at least approximable from above in stages $s$, so let $K_s(\sigma)\geq K(\sigma)$ be such an approximation. For more on Kolmogorov complexity, see \cite{KolmApp}.

Viewed this way \cite{Athrey1,Mayor1}, the Hausdorff and packing dimensions are dual to one another:
\begin{definition}\label{Hausdorff}\label{packing}
	The effective \emph{Hausdorff} and \emph{packing} dimensions of $A\in 2^{\omega}$ are, respectively
\[\displaystyle\dim_{H}(A) = \sup_{m\in\N} \inf_{n\ge m} \dfrac{K(A\upto n)}{n}\]
\[
\displaystyle\dim_{p}(A) = \inf_{m\in\N} \sup_{n\ge m} \dfrac{K(A\upto n)}{n}.\]
	\end{definition}
Another notion of dimension was defined in previous work by Kjos-Hanssen and Freer \cite{CompPack}.
Let $\mathfrak{D}$ denote the collection of all infinite $\Delta^0_1$ elements of $\Cant$.
\begin{definition}
The \emph{complex packing dimension} of $A\in 2^{\omega}$ is
\label{cp}
	$\displaystyle \dim_{cp}(A) = \sup_{N\in \mathfrak{D}} \inf_{n\in N} \dfrac{K(A\upto n)}{n}$.
\end{definition}
This leads naturally to a dual notion, obtained by switching the order of $\inf$ and $\sup$
\begin{definition}\label{ines}The \emph{inescapable dimension} of $A\in 2^\omega$ is
	$\displaystyle \dim_i(A) = \inf_{N\in\mathfrak{D}} \sup_{n\in N} \dfrac{K(A\upto n)}{n}.$
\end{definition}
This is so named because if $\dim_i(A) = \alpha$, every infinite computable collection of prefixes of $A$ must contain prefixes with $K(A\upto n)/n$ arbitrarily close to $\alpha$. For such a real, there is no (computable) escape from high complexity prefixes.
As Freer and Kjos-Hanssen show in \cite{CompPack}, 
\begin{theorem}\label{cpineq}
    For any $A\in \Cant$, $0\le \dim_H(A)\le \dim_{cp}(A)\le \dim_p(A)\le 1.$
\end{theorem}
	The expected analogous result also holds:
	\begin{theorem}\label{iineq}
		For any $A\in \Cant$, $0\le \dim_H(A)\le \dim_{i}(A)\le \dim_p(A)\le 1$.
	\end{theorem}
	\begin{proof} As the sets $[n, \infty)$ are computable subsets of $\N$, $ \dim_{i}(A)\le \dim_p(A)$. For the second inequality, notice that for all $m\in\N$ and all $N\in\Delta^0_1$,
		$$
			\inf_{n\in [m, \infty)} \dfrac{K(A\upto n)}{n}
			\le \inf_{n\in N\cap[m, \infty)}\dfrac{K(A\upto n)}{n}\le \sup_{n\in N\cap[m, \infty)}\dfrac{K(A\upto n)}{n} \le \sup_{n\in N}\dfrac{K(A\upto n)}{n}.\qedhere
		$$
	\end{proof}
\section{Incomparability}
Unexpectedly, Theorems \ref{cpineq} and \ref{iineq} are the best one can do --- while the packing dimension of a string is always lower than its Hausdorff dimension, any permutation is possible for the complex packing and inescapable dimensions of a real:
\begin{theorem}\label{incomp}
	There exist $A$ and $B$ such that $\dim_{cp}(A)<\dim_{cp}(B)$, but $\dim_{i}(B)<\dim_i(A)$. \end{theorem}
We first set up some definitions and notation.

For a real $A$, let us write $A[m,n]$ to denote the string $A(m) A(m+1) \dots A(n-1)$. For two functions $f(n), g(n)$ we write $f(n)\le^+ g(n)$ to denote $\exists c\forall n\, f(n)\le g(n)+c$. We write $f(n) = \mathcal{O}(g(n))$ to denote $\exists M\exists n_0\forall n>n_0\,f(n)\leq Mg(n)$.
\begin{definition}\label{sch}
	$A$ is Martin-L\"of random\footnote{We will not consider another notion of randomness until \Cref{ch:either}, so we may write `random' to mean Martin-L\"of random when it is clear in context.} iff $n\le^+ K(A\upto n)$.
\end{definition}
While this is not Martin-L\"of's original definition, it is an equivalent characterization due to Schnorr \cite{DaH}. 
\begin{definition} Let $\mathcal{S}\subseteq 2^{<\omega}$. A real $A$ \emph{meets} $\mathcal{S}$ iff some prefix of $A$ is in $\mathcal{S}$. $A$ \emph{avoids} $\mathcal{S}$ iff it has a prefix $\sigma$ such that no extension $\tau\succ\sigma$ is in $\mathcal{S}$. $\mathcal{S}$ is \emph{dense} iff every $\sigma\in2^{<\omega}$ has an extension $\tau\in \mathcal{S}$.
\end{definition}
\begin{definition}
A real $A$ is \emph{$n$-generic} iff for every $\Sigma^0_n$ set $\mathcal{S}\subseteq 2^{<\omega}$, $A$ meets or avoids $\mathcal{S}$.
\end{definition}
\begin{definition}
A real $A$ is \emph{weakly $n$-generic} iff it meets every dense $\Sigma^0_n$ set $\mathcal{S}$.
\end{definition}

Finally, for a real $A$ and $n\in\omega$ we use the indicator function $1_A$ defined by $$1_A(n)=\begin{cases} 1& \text{if } n\in A,\\ 0&\text{otherwise.}\end{cases}$$
\begin{proof}[Proof of \Cref{incomp}]
	Let $A$ be a weakly 2-generic real, and let $R$ be a Martin-L\"of random real.
	Let $s_k=2^{k^2}$, $k_n = \max\{k\mid s_k\le n\}$, and $C = (01)^\omega$. Define
	$B(n) = R\left(n-s_{k_n}\right)\cdot 1_{C}(k_n).$
	
	Unpacking this slightly, this is
	$$B(n) =
	    \begin{cases}
		R\left(n - s_k\right), & \text{if }\ s_k\le n < s_{k+1} \text{ for some odd }k,\\
		0, & \text{otherwise.}
		\end{cases}$$
     In this proof, let us say that an $R$-segment is a string of the form $B[s_{2m},s_{2m+1}]$ for some $m$, and say that a 0-segment is a string of the form $B\upto [s_{2m+1},s_{2m+2})$ for some $m$. These are named so that an $R$-segment consists of random bits, and a 0-segment consists of zeros.

	Notice that by construction, each such segment is much longer than the combined length of all previous segments. This guarantees certain complexity bounds at the segments' right endpoints. For instance, $B$ has high complexity at the end of $R$-segments: for any even $k\in\N$,
		\begin{align*}
			s_{k+1} - s_k\le^+K\left(B\left[s_k, s_{k+1}\right]\right)
			\le^+& K(B\upto s_k) + K(B\upto s_{k+1})
			\le^+ 2 s_k + K(B\upto s_{k+1}).
		\end{align*}
	The first inequality holds by \Cref{sch} because $B\left[s_k, s_{k+1}\right] = R\upto(s_{k+1} - s_k)$.
	The second (rather weak) inequality holds because prefix-free complexity is subadditive: from descriptions of $B\upto s_k$ and $B\upto s_{k+1}$ we can recover $B[s_k,s_{k+1}]$. Finally, $K(\sigma)\le^+ 2\abs{\sigma}$ is a property of prefix-free complexity.
	Combining and dividing by $s_{k+1}$ gives
	\begin{align}\label{using}
		s_{k+1} - 3s_k&\le^+K(B\upto s_{k+1})\nonumber \\
		1-3\cdot 2^{-(2k+1)}&\le\dfrac{K(B\upto s_{k+1})}{s_{k+1}}+ \mathcal{O}\left(2^{-(k+1)^2}\right) \quad\text{as $k\to\infty$.}
	\end{align}
	Dually, the right endpoints of $0$-segments have low complexity: for any odd $k\in\N$,
	\begin{align*}
		K(B\upto s_{k+1})&\le^+ K(B\upto s_k) + K(B[s_k, s_{k+1}]) \le^+ 2s_k + 2\log(s_{k+1} - s_k).
	\end{align*}
	The first inequality is again the weak bound that $B\upto s_{k+1}$ can be recovered from descriptions of $B\upto s_k$ and $B[s_k, s_{k+1}]$. For the second, we apply the $2\abs{\sigma}$ complexity bound to $B\upto s_k$, but also notice that since $B[s_k, s_{k+1}] = 0^{s_{k+1}-s_k}$, it can be recovered effectively from a code for its length. Combining and dividing by $s_{k+1}$, we have
	\begin{align}\label{using2}
		K(B\upto s_{k+1})&\le^+ 2s_k + 2(k+1)^2\text{, and hence} \nonumber\\
		\dfrac{K(B\upto s_{k+1})}{s_{k+1}}&\le 2^{-(2k+1)} + \mathcal{O}\left(2^{-(k+1)^2}\right) \quad\text{as $k\to\infty$.}
	\end{align}
		Now we can examine the dimensions of $A$ and $B$.
		
		\noindent\textbf{Claim 1:} $\dim_{cp}(B) = 1$.
		
		\noindent \emph{Proof:} Let $R_n$ be the set of right endpoints of $R$-segments of $B$, except for the first $n$ of them, i.e.~$R_n = \{s_{2k+1}\}_{k=n}^\infty$.
		Then the collection of these $R_n$ is a subfamily of $\mathfrak{D}$, so that a supremum over $\mathfrak{D}$ will be at least the supremum over this family. Using (\ref{using}), we find that
            $$\sup_{N\in\mathfrak{D}} \inf_{n\in N} \dfrac{K(B\upto n)}{n}
            \geq \sup_{n\in\N} \inf_{s\in R_n} \dfrac{K(B\upto s)}{s}
			\geq \sup_{n\in\N} \inf_{s\in R_n} 1-3\cdot 2^{-(2s+1)}
			= \sup_{m\in\N} 1-3\cdot 2^{-(2m+1)}
			=1.$$

		\noindent \textbf{Claim 2:} $\dim_{i}(B) = 0$.
		
		\noindent \emph{Proof:} Let $Z_n$ be the set of right endpoints of $0$-segments of $B$, except for the first $n$ of\linebreak them, i.e.~$Z_n = \{s_{2k}\}_{k=n}^\infty$. Similarly to Claim 1, we use (\ref{using2}) to obtain
	    $$\inf_{N\in\mathfrak{D}} \sup_{n\in N} \dfrac{K(B\upto n)}{n}
	    \leq \inf_{n\in\N} \sup_{s\in Z_n} \dfrac{K(B\upto s)}{s}
		\leq \inf_{n\in\N} \sup_{s\in Z_n} 2^{-(2s+1)}
		= \inf_{m\in\N}2^{-(2m+1)}=0. \qedhere$$

		\noindent \textbf{Claim 3:} $\dim_{cp}(A) = 0$.
		
		\noindent \emph{Proof:} For each $N\in\mathfrak{D}$ and each natural $k$, the following sets are dense $\Sigma^0_1$:
		\[
			\left\{\sigma\in\Cant : \abs{\sigma}\in N \text{ and } (\exists s)\ K_s(\sigma)<\abs{\sigma}/k]\right\}.
		\]
		As $A$ is weakly 2-generic, it meets all of them. Hence
	    $\displaystyle\sup_{N\in\mathfrak{D}}\inf_{m\in N}\dfrac{K(\sigma\upto m)}{m}=0.$
	    
        \noindent \textbf{Claim 4:} $\dim_{i}(A) = 1$.
		
		\noindent \emph{Proof:} For each $N\in\mathfrak{D}$ and each natural $k$,
		\[
			\left\{\sigma\in\Cant : \abs{\sigma}\in N \text{ and } (\forall s)\ K_s(\sigma)>\abs{\sigma}(1-1/k)\right\}
		\]
		is a dense $\Sigma^0_2$ set. As $A$ is weakly 2-generic, it meets all such sets. Hence 
		$\displaystyle\inf_{N\in\mathfrak{D}}\sup_{m\in N}\dfrac{K(A\upto m)}{m}=1.$
	    \end{proof}
\section{Further Dimensions: (Non-)Collapse and Embedding}
\label{sec:separation}
After considering supremums and infimums of $\Delta^0_1$ sets, it is natural to extend these definitions into the arithmetic hierarchy. For full generality, we say that $A$ is finite-to-one reducible to $B$ iff there is a total computable function $f:\omega\to\omega$ such that the preimage of each $n\in\omega$ is finite and for all $n$, $n\in A\iff f(n)\in B$.
\begin{definition}
		Let $\B$ be a class of infinite sets that is downward closed under finite-to-one reducibility. For $A\in \Cant$, define
		\[
			 \dim_{is\B}(A) = \inf_{N\in\B} \sup_{n\in N} \dfrac{K(A\upto n)}{n}\quad\text{and}\quad
			 \dim_{si\B}(A) = \sup_{N\in\B} \inf_{n\in N} \dfrac{K(A\upto n)}{n}.
		\]
	\end{definition}

	Notice that for any oracle $X$, the classes of infinite sets that are $\Delta^0_n(X), \Sigma^0_n(X)$ or $\Pi^0_n(X)$ are downward closed under finite-to-one reducibility, and so give rise to notions of dimension of this form.
	We will label these $\mathfrak{D}_n(X)$, $\mathfrak{S}_n(X)$, and $\mathfrak{P}_n(X)$ respectively, leaving off $X$ when $X$ is computable. Interestingly, for fixed $n$, the first two give the same notion of dimension.
	\begin{theorem} For all $A\in\Cant$ and $n\in\N$, $\dim_{is\Sigma^0_n}(A) = \dim_{is\Delta^0_n}(A)$.
	\end{theorem}
	\begin{proof}
		We prove the unrelativized version of the statement, $n=1$.\\
		\textbf{[$\le$]} As $\Delta^0_1\subseteq \Sigma^0_1$, this direction is trivial.\\
		\textbf{[$\ge$]} As every infinite $\Sigma^0_1$ set $N$ contains an infinite $\Delta^0_1$ set $N_1$, we have
		\begin{align*}
			\dim_{is\Sigma^0_1}(A)  = \inf_{N\in \mathfrak{S}_1} \sup_{n\in N} \dfrac{K(A\upto n)}{n}
			&\ge \inf_{N\in \mathfrak{S}_1} \sup_{n\in N_1} \dfrac{K(A\upto n)}{n}
			&\ge \inf_{N\in \mathfrak{D}_1} \sup_{n\in N} \dfrac{K(A\upto n)}{n}
			= \dim_{is\Delta^0_1}(A).\ \ \ \ \qedhere
		\end{align*}
	\end{proof}
	By a similar analysis, the analogous result for $si$ dimensions is also true.
	\begin{theorem} For all $A\in\Cant$ and $n\in\N$, $\dim_{si\Sigma^0_n}(A) = \dim_{si\Delta^0_n}(A)$.\label{sigeqdelsi}\end{theorem}

	What about the $\Pi^0_n$ dimensions? Unlike the $\Sigma^0_n$ case, these do not collapse down to their $\Delta^0_n$ counterparts, nor up to the $\Delta^0_{n+1}$ dimensions. Two lemmas will be useful in proving this.
	The first (which was implicit in Claims 1 and 2 of \Cref{iineq}) will allow us to show that an $si$-dimension of a real is high by demonstrating a sequence that witnesses this.\footnote{The converse is not true --- this is the content of \Cref{witness}.}
	The second is a generalization of the segment technique, forcing a dimension to be $0$ by alternating $0$- and $R$-segments in a more intricate way, according to the prescriptions of a certain real.
	The constructions below proceed by selecting a real that will guarantee that one dimension is 0 while leaving room to find a witnessing sequence for another.
	\begin{lemma}[Sequence Lemma]\label{seqlem}
	Let $\B$ be a class of infinite sets downward closed under finite-to-one reducibility, and let $N= \{n_k\mid k\in\omega\}\in\B$.
	\begin{enumerate}
		\item \label{one} If $\displaystyle  \lim_{k\ra\infty}\dfrac{K(X\upto n_k)}{n_k}=1$, then $\dim_{si\mathfrak{B}}(X) = 1$.
		\item\label{two} If $\displaystyle\lim_{k\ra\infty}\dfrac{K(X\upto n_k)}{n_k}=0$, then $\dim_{is\mathfrak{B}}(X)=0$.
	\end{enumerate}
	\end{lemma}
	\begin{proof}
		We prove (i); (ii) is similar.

		Form the infinite family of sets $\{N_m\}$ defined by $N_m = \{n_k\mid k\geq m\}$. From the definition of the limit, for any $\eps>0$ there  is an $l$ such that $$\inf_{N_l}\dfrac{K(X\upto n_k)}{n_k}> 1-\eps.$$ As $\eps$ was arbitrary, $$\sup_m \inf_{N_m}\dfrac{K(X\upto n_m)}{n_m}= 1.$$
		Thus as $\B$ is closed under finite-to-one reduction, the $N_m$ form a subfamily of $\B$, so that $\displaystyle \sup_{N\in \mathfrak B} \inf_{n\in N} K(X\upto n)/n= 1$.
	\end{proof}

\begin{definition}\label{def:im}
A real $A$ is \emph{immune} to a class $\B$ if there is no infinite member $B\in\B$ such that $B\subseteq A$ as sets. $A$ is \emph{co-immune} to a class $\B$ if its complement is immune to $\B$. $A$ is \emph{bi-immune} to $\B$ iff it is immune and co-immune to $\B$.
\end{definition}
We will often refer to these properties as $\B$-immunity, co-$\B$-immunity, and bi-$\B$-immunity, respectively. In the case that $\B = \Delta^0_1$, we drop the $\B$ and simply say $A$ is immune.
\begin{definition}\label{directsum}
For reals $A$ and $B$, $A\oplus B = \{2k\mid k\in A\}\cup\{2k+1\mid k\in B\}$.
\end{definition}
	\begin{lemma}[Double Segment Lemma]\label{dslem} Let $X_0\in \Cant$ be such that $X_0$ is co-immune to reals of a class $\mathfrak{B}$, and set $X = X_0\oplus X_0$.
		For all natural $n$, define $k_n = \max\{\text{odd }
		k\mid 2^{k^2}\le n\}$. Let $A$ be an arbitrary real and let $R$ be Martin-L\"of random. \begin{enumerate}
		\item\label{dslem-one} If $B = A\left(n-2^{k_n^2}\right)\cdot 1_{\overline{X}}(k_n)$, then $\dim_{si\mathfrak{B}}(B) = 0$.
		\item\label{dslem-two} If $B = R\left(n-2^{k_n^2}\right)\cdot 1_{X}(k_n)$, then $\dim_{is\mathfrak{B}}(B) = 1$.
	\end{enumerate}
	\end{lemma}
		Again, we will give a detailed proof of only the $\dim_{si\B}$ result (though the necessary changes for $\dim_{is\B}$ are detailed below). Unpacking the definition of $B$,
		\[
			B(n) = \begin{cases}
					A\left(n - s_k\right) & \text{if $k_n\in X$}\\
					0 & \text{otherwise.}
					\end{cases}
		\] $B$ is once again built out of segments of the form $B\left[s_{k_n}, s_{k_n+2}\right]$ for odd $k$. Here a segment is a $0$-segment if $k_n\not\in X$, or an $A$-segment if $k_n\in X$, which by definition is a prefix of $A$. These segments are now placed in a more intricate order according to $X$, with a value $n$ being contained in a $0$-segment if $X(k_n) = 0$, and in an $A$-segment if $X(k_n) = 1$. With some care, this will allow us to leverage the $\B$-immunity of $X_0$ to perform the desired complexity calculations.
			
        Specifically, we want to show that for any $N\in\B$, $\inf_NK(B\upto n)/n=0$. It is tempting to place the segments according to $X_0$ and invoke its $\B$-immunity to show that for any $N\in\B$, there are infinitely many $n\in N$ such that $n$ is in a $0$-segment, then argue that complexity will be low there. The problem is that we have no control over \emph{where} in the $0$-segment $n$ falls. Consider in this case the start of any segment following an $A$-segment: $n=s_{k_n}$ for $k_{n}-1\in X_0$ and $k_n\in X_0$. We can break $A$ and $B$ into sections to compute \begin{align*}K(A\upto n) &\leq^+ K(A\upto(n-s_{k_n-1}))+K(A[n-s_{k_n-1}, n]) \\
        &=K(B[s_{k_n-1},n])+K(A[n-s_{k_n-1}, n])\tag{$k_n-1\in X_0$}\\
        &\leq^+K(B\upto n) + K(B\upto s_{k_n-1})+K(A[n-s_{k_n-1}, n])\\
         K(A\upto n)&\leq^+K(B\upto n)+4s_{k_n-1}\tag{$K(\sigma)\leq^+2|\sigma|$}
        \end{align*} Even if $n$ is the start of a $0$-segment, if $K(A\upto n)$ is high, $K(B\upto n)$ may not be as low as needed for the proof. Our definition of $X$ avoids this problem:
	\begin{proof}[Proof of Theorem \ref{dslem}]
	Suppose for the sake of contradiction that for some $N\in\B$, there are only finitely many $n\in N$ with $k_n, k_n-1\in \overline{X}$, i.e., that are in a $0$-segment immediately following another $0$-segment.
			Removing these finitely many counterexamples we are left with a set $N_1\in\B$ such that for all $n\in N_1$, $\lnot[(k_n\not\in X) \land (k_n-1\not\in X)].$
			As $k_n$ is odd, the definition of $X$ gives that $\lfloor k_n/2\rfloor\in X_0$.
			By a finite-to-one reduction from $N_1$, the infinite set $\{\lfloor k_n/2\rfloor\}_{n\in N_1}$ is a member of $\B$ and is contained in $X_0$, but $\overline{X_0}$ is immune to such sets.
			
			Instead it must be the case that there are infinitely many $n\in N$ in a $0$-segment following a $0$-segment, where the complexity is \begin{align*}
			K(B\upto n) &\leq^+ K\left(B\upto s_{n_{k-1}}\right)+ K\left(B\left[s_{n_{k-1}},n\right]\right)\\
			&\le^+ 2s_{n_{k-1}} + 2\log\left(n-s_{n_{k-1}}\right).
			\end{align*}
			Here the second inequality follows from the usual $2\abs{\sigma}$ bound and the fact that $B \left[s_{n_{k-1}},n\right]$ contains only $0$s. As $2^{k_n^2}\le n$, we can divide by $n$ to get
			\begin{align*}
			\dfrac{K(B\upto n)}{n}\le^+ \dfrac{2^{k_n^2-2k_n}}{2^{k_n^2}} + \dfrac{2\log(n)}{n}= 2^{-2k_n}+\dfrac{2\log(n)}{n}.
			\end{align*}
			As there are infinitely many of these $n$, it must be that $\inf_{n\in N} K(B\upto n)/n = 0$.
			This holds for every set $N$ in the class $\B$, so taking a supremum gives the result.
			
			The $\dim_{is\B}$ version concerns reals $B$ constructed in a slightly different way. Here, the same argument now shows there are infinitely many $n\in N$ in an $R$-segment following an $R$-segment. At these locations, the complexity $K(B\upto n)$ can be shown to be high enough that $\sup_N K(B\upto n)/n = 1$, as desired. 
	\end{proof}
	
	With these lemmata in hand, we are ready to prove
	\begin{theorem}\label{pidelsepsi}
		For all natural $n$ there is a set $A$ with $\dim_{si\Delta^0_n}(A)=0$ and $\dim_{si\Pi^0_n}(A) = 1$.
	\end{theorem}
	\begin{proof}
		We prove the $n=1$ case, as the proofs for higher $n$ are analogous.
		
		Let $S_0$ be c.e.~and co-immune set\footnote{These are also called \emph{simple} sets, and were shown to exist by Post \cite{Post}.}, and let $R$ be \MLR. Let $S = S_0\oplus S_0$, and define $k_n = \max\{k\mid 2^{k^2}\le n\}$. Define $A(n) = R\left(n-2^{k_n^2}\right)\cdot 1_{\overline S}(k_n)$, 
		so that $A$ is made of $0$-segments and $R$-segments.
		
		As $S$ is $\Sigma^0_1$, the set of right endpoints of $R$-segments, $M = \left\{2^{k^2}\mid k-1\in \overline{S}\right\}$ is $\Pi^0_1$.
		By construction $\lim_{m\in M}K(A\upto m)/m = 1$ and thus the Sequence Lemma \ref{seqlem} gives that $\dim_{si\Pi^0_1}(A) = 1$.

		As $\overline{S}$ is immune, the Double Segment Lemma \ref{dslem} shows that $\dim_{si\Delta^0_1}(A) = 0$.
	\end{proof}

	The proof of analogous result for the $is$-dimensions is similar, using the same $S_0$ and $S$, and the real defined by $B(n) = R\left(n-2^{k_n^2}\right)\cdot 1_{\overline{S}}(k_n)$.
	\begin{theorem} For all natural $n$ there is a set $B$ with 
		$\dim_{is\Delta^0_{n}}(B) = 1$ and $\dim_{is\Pi^0_n}(B) = 0$. \end{theorem}
	
	It remains to show that the $\Delta^0_{n+1}$ and $\Pi^0_n$ dimensions are all distinct.
	We can use the above lemmata for this, so the only difficulty is finding sets of the appropriate arithmetic complexity with the relevant immunity properties.

\remark{In \Cref{ch:pi}, we give a fuller account of $\Pi^0_1$-immune sets and their properties. But for the sake of keeping this chapter self-contained, we include the following definition and lemma now:}
\begin{definition}\label{pi:coh}
A real $C$ is \emph{cohesive} iff it cannot be split into two infinite halves by a c.e.\ set, i.e.\ for all $e$ either $W_e\cap C$ or $\overline{W_e}\cap C$ is finite.
\end{definition}
\begin{lemma}\label{piim} For all $n\geq 1$, there is an infinite $\Delta^0_{n+1}$ set $S$ that is $\Pi^0_n$-immune.
\end{lemma}
\begin{proof} We prove the unrelativized version, $n=1$. Let $C$ be a $\Delta^0_2$ cohesive set that is not co-c.e.~(such a set exists by \cite{JockHighCoh})\footnote{We will also construct more explicit $\Delta^0_2\setminus\Pi^0_1$ cohesive sets in \Cref{hcompPi01im}.}. As $\overline{C}$ is not c.e.\ it cannot finitely differ from any $W_e$, so for all $e$, $W_e\setminus \overline{C} = W_e\cap C$ is infinite. Hence if $\overline{W_e}\subseteq C$, then by cohesiveness, $\overline{W_e}\cap C = \overline{W_e}$ is finite.
\end{proof}

\begin{theorem} For all $n\ge 1$ there exists a set $A$ with $\dim_{si\Pi^0_n}(A)=0$ and $\dim_{si\Delta^0_{n+1}}(A) = 1$. \label{pidelsepsin+1}\end{theorem}

\begin{proof}
		This is exactly like the proof of Theorem \ref{pidelsepsi}, but $S_0$ is now the $\Pi^0_1$-immune set guaranteed by Lemma \ref{piim}.
	\end{proof}

	Again, the analogous result for $is$-dimensions is similar:
	\begin{theorem} For all $n\geq 1$ there exists a set $B$ with $\dim_{is\Pi^0_n}(B)=1$ and $\dim_{is\Delta^0_{n+1}}(B) = 0$. \label{pidelsepisn+1}\end{theorem}
After asking questions about the arithmetic hierarchy, it is natural to turn our attention to the Turing degrees.
	We shall embed the Turing degrees into the $si\Delta^0_1(A)$ (and dually, $is\Delta^0_1(A)$) dimensions. First, a helpful lemma:
	\begin{lemma}[Immunity Lemma]\label{immlem} If $A\nleq_T B$, there is an $S\le_T A$ such that $S$ is $B$-immune.
	\end{lemma}
	\begin{proof}
		Let $S$ be the set of finite prefixes of $A$. If $S$ contains an infinite $B$-computable subset $C$, then we can recover $A$ from $C$, but then $A\leq_T C\leq_T B$.
	\end{proof}
	\begin{theorem}[{$si$-$\Delta^0_1$ Embedding Theorem}]\label{delembedding} Let $A,B\in \Cant$. Then  $A\le_T B$ iff
	for all $X\in \Cant, \dim_{si\Delta^0_1 (A)} (X) \le \dim_{si\Delta^0_1 (B)} (X)$.
	\end{theorem}
	\begin{proof}
		\textbf{[$\Ra$]} Immediate, as $\Delta^0_1(A)\subseteq \Delta^0_1(B)$.\\
		\textbf{[$\La$]} This is again exactly like the proof of \Cref{pidelsepsi}, now using the set guaranteed by the Immunity Lemma \ref{immlem} as $S_0$.
	\end{proof}

The result for $is$-dimensions is again similar:

\begin{theorem}[{$is$-$\Delta^0_1$ Embedding Theorem}]\label{delembedding2} Let $A,B\in \Cant$. Then  $A\le_T B$ iff
	for all $X\in \Cant, \dim_{is\Delta^0_1 (A)} (X) \ge \dim_{is\Delta^0_1 (B)} (X)$.
	\end{theorem}
\section{Weak Truth Table Reduction}
We can push this a little further by considering weak truth table reductions.
\begin{definition} $A$ is \emph{weak truth table reducible to} $B$ ($A\leq_{wtt} B$) if there exists a computable function $f$ and an oracle machine $\Phi$ such that $\Phi^B = A$, and the use of $\Phi^X(n)$ is bounded by $f(n)$ for all $n$ ($\Phi^X(n)$ is not guaranteed to halt).
\end{definition}
	\begin{theorem}\label{tt}
		If $A\not\le_T B$, then for all wtt-reductions $\Phi$ there exists an $X$ such that $\dim_{si\Delta^0_1(A)}(X)=1$ and, if $\Phi^X$ is total, $\dim_{si\Delta^0_1(B)}(\Phi^X)=0$.
	\end{theorem}
That is, Turing irreducibility of degrees implies wtt-irreducibility of $si$-dimensions. It will be illuminating to consider a proof sketch first, to illustrate the ideas at play.

\emph{Proof Sketch:} Fix a $wtt$-reduction $\Phi$ with use bounded by $g(n)$. We wish to construct segments $[\Lambda_k, \Lambda_{k+1}]$ of length $\lambda_k$ in $\Phi^X$ based on use-segments $[L_k, L_{k+1}]$ of length $\ell_k$ in $X$. That is, the $L_k$ are chosen so that each segment is much longer than those that have come before, such that $g(n)\leq L_k$ for $n\leq \Lambda_{k}$, and that $\lambda_{k+1}$ is much longer than $L_k$. These requirements are all computable, as $g$ is.

For $X$, we fill use-segments in alternating fashion just as in the previous proof --- even segments are filled with $0$, and odd segments are filled with $0$s or \MLR\ bits according to what $S$ prescribes. We imagine $\Phi$ as an antagonist, trying to fill $\Phi^X$ with as much complexity as possible in the hopes of attaining a non-zero infimum on \emph{some} $B$-recursive infinite set.

For the first segment, $\Phi$ only has access to $0$s, so despite its best efforts it cannot push up complexity at all. However, as soon as some use-segment is filled with random bits, $\Phi$ takes full advantage of this, pushing complexity up as high as it likes (as the length of the segment provides at least enough random bits to choose from). Once some randomness has appeared above, $\Phi$ can try to access it when it is otherwise stuck with zeroes in the latest use segment --- it still has access to the same random bits it has already used. But here the requirement that $\lambda_{k+1}$ is much longer than $L_k$ comes in: $\Phi$ tries to fill a tremendous number of entries with randomness, but only has access to a small number of random bits. Despite $\Phi$'s best efforts, the final complexity cannot be that high, as $\Phi$'s use is computable and we can hard-code these random bits for a small cost relative to the number of bits in the segment.

Even in this worst-case scenario where $\Phi$ is playing against us, in a sense it can at best match the pattern of the segments in $\Phi^X$ to the pattern in $X$. Defining $X$ via an $A$-computable, $B$-immune set thus ensures that $\dim_{si\Delta^0_1(B)}(\Phi^X) = 0$.

For the actual proof, we assume a general $\Phi$ with unknown (rather than antagonistic) motives, and formally carry out the proof by contradiction:

	\begin{proof}[Proof of \Cref{tt}]
		Let $A\not\le_T B$, and let $\Phi$ be a wtt-reduction. Let $f$ be a computable bound on the use of $\Phi$, and define $g(n) = \max\{f(i)\mid i\le n\}$, so that $K(\Phi^X\upto n)\le^+ K(X\upto g(n)) + 2\log(n)$. For notational clarity, for the rest of this proof we will denote inequalities that hold up to logarithmic (in $n$) terms as $\leq^{\log}$.
		
		Next, we define two sequences $\ell_k$ and $\lambda_k$ which play the role $2^{k^2}$ played in previous constructions:
		\begin{align*}
			\ell_0 = \lambda_0 = 1, &&\lambda_k = \lambda_{k-1} +  \ell_{k-1}, && \ell_k = \min\left\{2^{n^2}\mid g(\lambda_{k})<2^{n^2}\right\}.
		\end{align*}
	    These definitions have the useful consequence that $\lim_k \ell_{k-1}/\ell_{k} = 0$. To see this, suppose $\ell_{k-1} = 2^{(n-1)^2}$. As $g$ is an increasing function, the definitions give
		$$\ell_k>g(\lambda_{k}) \ge \lambda_{k}=\lambda_{k-1} + \ell_{k-1}\ge \ell_{k-1}=2^{(n-1)^2}.$$
		Hence $\ell_{k}\ge 2^{n^2}$, so that $\ell_{k-1}/\ell_k \le 2^{-2n+1}$. As $\ell_k>\ell_{k-1}$ for all $k$, this ratio can be made arbitrarily small, giving the limit.

        A triple recursive join operation is defined by
        \[
        \bigoplus_{i=0}^2 A_i = \{3k+j \mid k\in A_j,\quad 0\le j\le 2\},\quad A_0,A_1,A_2\subseteq\omega.
        \]
		
		Let $S_0\leq_T A$ be as guaranteed by \Cref{immlem}, and define $S=\bigoplus_{i=0}^2 S_0$.
		Let $R$ be \MLR, and define $X(n) = R\left(n - \ell_{k_n}\right) \cdot 1_{S}(k_n)$, where $k_n = \max\{k = 2 \pmod 3\mid \ell_k\le n\}$. This definition takes an unusual form compared to the previous ones we have seen in order to handle the interplay between $\lambda_k$ and $\ell_k$ --- specifically the growth rate of $g(n)$.
	
		\noindent \textbf{Claim 1:} $\dim_{si\Delta^0_1(A)}(X) = 1$.
		
		\noindent \emph{Proof:} As $N= \left\{\ell_k\right\}_{k\in S}$ is an $A$-computable set, by the Sequence Lemma \ref{seqlem} it suffices to show that
		\(
			\lim_{k\in S} K(X\upto \ell_k)/\ell_k = 1.
		\)
		For $\ell_k\in N$,
		\begin{align*}
			K(X\upto \ell_k)&\ge^+ K(X[\ell_{k-1}, \ell_{k}]) - K(X\upto \ell_{k-1})\tag{subadditivity}\\
			&\ge^+ K(R\upto(\ell_k-\ell_{k-1})) - 2\ell_{k-1} \tag{$k\in S$}\\
			&\ge^+ \ell_k - \ell_{k-1} - 2\ell_{k-1} \tag{$R$ is \MLR}\\
			\dfrac{K(X\upto \ell_k)}{\ell_k}&\ge^+ \dfrac{\ell_k - 3\ell_{k-1}}{\ell_k}= 1 - 3\dfrac{\ell_{k-1}}{\ell_k}.
		\end{align*}
		which gives the desired limit by the above.
		
		\noindent\textbf{Claim 2:} If $\Phi^X$ is total, $\dim_{si\Delta^0_1(B)}(\Phi^X) = 0$.
		
		\noindent \emph{Proof:} Suppose $N\leq_T B$. By mimicking the proof of Lemma \ref{dslem}, we can use the $B$-immunity of $S$ to show that there are infinitely many $n\in N$ such that $g(n)$ is in a $0$-segment following two $0$-segments.
		For such an $n$, define $a = k_{g(n)}$, so that $a-2, a-1, a\not\in S$.
		As $g(n)<\ell_{a+1}$, to compute $\Phi^X\upto n$, it suffices to know $X\upto \ell_{a+1}$. By assumption, $X[\ell_{a-2}, \ell_{a+1}]$ contains only $0$s, so a program that outputs $X\upto\ell_{a-2}$ followed by $0$s until the output is of length $n$ will compute $X\upto\ell_{a+1}$. Thus
		$$K(\Phi^X\upto n)\le^+ K(X\upto \ell_{a-2})+2\log(n)\leq^{\log} 2\ell_{a-2}.$$
        As $g(n)>\ell_{a}$, by the definition of $\ell_a$, $n>\lambda_a$. Dividing by $n$, we find that
		\begin{align*}
		\dfrac{K(\Phi^X\upto n)}{n}\leq^{\log}\dfrac{2\ell_{a-2}}{\lambda_a}=
		 \dfrac{2\ell_{a-2}}{\lambda_{a-1}+\ell_{a-1}}
		 <\dfrac{2\ell_{a-2}}{\ell_{a-1}}.
		 \end{align*}
		As there are infinitely many of these $n$, it must be that $\inf_{n\in N} K(\Phi^X\upto n)/n = 0$.
This holds for every $N\leq_T B$, so taking a supremum gives the result.\end{proof}
\begin{remark} We only consider $si$-dimensions for this theorem, as it is not clear what an appropriate analogue for $is$-dimensions would be. The natural dual statement for $is$-dimensions would be that for all reductions $\Phi$ there is an $X$ such that $\dim_{is\Delta^0_1(A)}(X)=0$, and either $\Phi^X$ is not total or $\dim_{is\Delta^0_1(B)}(\Phi^X)=1$.
But many reductions use only computably much of their oracle, so that $\Phi^X$ is a computable set.
This degenerate case is not a problem for the $si$ theorem, as its conclusion requires $\dim_{\Delta^0_1(B)}(\Phi^X)=0$. But for an $is$ version, it is not even enough to require that $\Phi^X$ is not computable: consider the reduction that repeats the $n$th bit of $X$ $2n-1$ times, so that $n$ bits of $X$ suffice to compute $n^2$ bits of $\Phi^X$.
Certainly $\Phi^X\equiv_{wtt} X$, so that $\Phi^X$ is non-computable iff $X$ is.
But 
\begin{align*} \dfrac{K(\Phi^X\upto n)}{n}\leq^+\dfrac{K(X\upto \sqrt{n})}{n}\leq^+\dfrac{2\sqrt{n}}{n}
\end{align*}
for all $n$, so that $\dim_p(\Phi^X) = 0$, and hence all other dimensions are 0 as well.
\end{remark}

%% file: 2properties.tex
\section{Failure of the Converse of the Sequence Lemma}
\label{sec:properties}

Recall the Sequence Lemma for the inescapable dimension:
\begin{oldlemma}[\ref{seqlem}] If there is an $N\in\Delta^0_1$ such that $\displaystyle \lim_{n\in N} \dfrac{K(X\upto n)}{n} $, then $\dim_i(X) = 0$.
\end{oldlemma}

It is important to note that this is not a characterization of the inescapable dimension. It is possible that no single computable set witnesses complexity going all the way to zero (even as an infimum), while complexity $<\eps$ can always be computably witnessed.
\theorem\label{witness} There is a real with $\dim_i(Y) = 0$ such that for any $N\in\Delta^0_1$, $\displaystyle\lim_{n\in N} \dfrac{K(Y\upto n)}{n}\neq 0$.
\begin{proof}
For strings $\sigma$, say that $\sigma$ is the $\ell(\sigma)$th element of the lexicographic order of $2^{<\omega}$.

Define $s_k=2^{k^2}$, $k_n =\max\left\{k\in\omega : 2^{k^2}\leq n\right\}$, and let $R$ be Martin-L\"of random. For each $k$, let $A_k = \{kn\mid n\in\omega\}$, and define $R_k$ by replacing the $n$th $1$ in $A_k$ with the $n$th bit of $R$.\footnote{In the notation of \Cref{oplus}, $R_k=R\oplus_{A_k} \emptyset$.}
For $n\in\omega$, let $\sigma_n$ be the string such that $k_n = \langle m, \ell(\sigma_n)\rangle$ for some $m$. Finally define
$$Y(n) = \begin{cases}
R_{|\sigma_n|}(n - s_{k_n})& \sigma_n\prec R
\\ R(n)&\sigma_n\not\prec R
\end{cases}$$
That is, start with a random $R$, and build a ``semirandom" string but replace bits $n$ such that $\sigma_n\prec R$ with bits from $R_{|\sigma_n|}$. 
For notation, call the bits $Y[s_{k_n}, s_{k_n+1}]$ a $\sigma_n$-segment, where $k_n = \langle m, \ell(\sigma_n)\rangle$. 

\noindent\textbf{Claim 1:} $\dim_i(Y) = 0$.

\noindent\emph{Proof:} To the nearest integer, $A_k\upto n$ contains $n/k$ 1s. So as $A_k$ is computable, to describe $R_k\upto n$ it suffices to know the first $n/k$ bits of $R$.  
Hence 
$K(R_k\upto n)\leq^+ K(R\upto n/k)\leq^+2n/k$.

Fix a $\sigma\prec R$. Following the derivation of \Cref{using2} in the proof of \Cref{incomp}, the right enpoints $n$ of sufficiently large $\sigma$-segments have\footnote{As in \Cref{using2}, this is technically up to a vanishing error term, which we leave off here for notational clarity.} that $K(Y\upto n)/n\leq |\sigma|$.  As $\sigma$ is computable, these large enough right endpoints form a computable set $N_\sigma$. Thus $$\dim_i(Y) = \inf_{N\in\Delta^0_1}\sup_{n\in N}\dfrac{K(Y\upto n)}{n} \leq \inf_{\sigma\prec R}\sup_{n\in N_\sigma}\dfrac{K(Y\upto n)}{n}\leq \inf_{\sigma\prec R}\sup_{n\in N_\sigma}\dfrac{n}{n|\sigma|}\leq \inf_{\sigma\prec R}\dfrac{1}{|\sigma|}=\inf_{n>0}\dfrac{1}{n}=0.$$

For notation, let $\tau_n$ be the lexicographic predecessor of $\sigma_n$.

\noindent \textbf{Claim 2:} For any $\eps>0$, for large enough $n$, if $K(Y\upto n)/n<1-\eps$, then $\sigma_n\prec R$ or $\tau_n\prec R$.

\noindent \emph{Proof:}
For contraposition, let $n$ be in a $\sigma$ segment following a $\tau$ segment such that $\sigma, \tau\not\prec R$ (so that both segments are filled with random bits). Let $s_\ell$ be the right endpoint of the longest semirandom segment of $Y\upto n$. By the definition of $n$, $k_n\geq\ell+1$, so $n\geq s_{k_n}\geq s_{\ell +1}>s_\ell$.
\begin{align*}
K(Y\upto n)&\geq^+K(Y[s_\ell, n]) - K(Y\upto s_\ell) &\text{property of Kolmogorov complexity}\\
&=K(R[s_\ell,n]) - K(Y\upto s_\ell)&\text{definition of $Y$} \\
&\geq^+ K(R\upto n)-K(R\upto s_\ell)-K(Y\upto s_\ell) &\text{property of Kolmogorov complexity}
\end{align*}
As $R$ is \MLR, $K(R\upto n)\geq^+ n$. For any string $\sigma$, $K(\sigma)\leq^+ 2|\sigma|$. Therefore
\begin{align*}
K(Y\upto n)&\geq^+n-4s_\ell\\
\dfrac{K(Y\upto n)}{n}&\geq 1 - 4\dfrac{s_\ell}{s_{\ell+1}} + \mathcal{O}(1/n)\\
&\geq 1 - 2^{-2\ell + 1}+\mathcal{O}(1/n)
\end{align*}
For any $\eps$, this can be made to be greater than $1-\eps$ for sufficiently large $n$.

\noindent \textbf{Claim 3:} There is no computable set $N$ such that $\displaystyle \lim_{n\in N}\dfrac{K(Y\upto n)}{n}=0$.

\noindent \emph{Proof:} Suppose towards a contradiction that such an $N$ exists. Let $\eps>0$. By Claim 2 and the convergence of $K(Y\upto n)/n$, let $M$ be large enough that for $n>M$, $K(Y\upto n)<\eps n$ and one of $\sigma_n$ or $\tau_n$ is a prefix of $R$. Note that as  $\tau_n<_{lex}\sigma$, by looking at longer $\sigma_m$ and $\tau_m$ we can decide which of $\sigma_n$ or $\tau_n$ is a prefix of $R$.\pagebreak

Suppose $\sigma_n\prec R$ infinitely often. Write $k$ for $k_n$ and $\sigma$ for $\sigma_n$ for ease of notation.
We have that \begin{align*}
|\sigma|^{-1}(n-s_{k})\leq^+K(R\upto |\sigma|^{-1}(n-s_{k}))\leq^+K(R_{|\sigma|}\upto (n-s_{k}))=K(Y[s_{k}, n])\leq^+K(Y\upto n)< \eps n.
\end{align*}
The first inequality follows from our definition of Martin-L\"of randomness. For the second, $m$ bits of $R_k$ can be used to recover $m/k$ bits of $R$ by looking at every $k$th bit. The equality is the definition of $Y$, and for the penultimate inequality, $s_{k_n}$ can be obtained computably from $n$. The final strict inequality is by hypothesis. Rearranging slightly, $|\sigma|^{-1}n\leq^+ \eps n + |\sigma|^{-1}s_k$.

We can also compute
\begin{align*}
s_k - 2s_{k-1}\leq^+ K(R\upto s_k) - K(R\upto s_{k-1})\leq^+ K(R[s_{k-1}, s_k])= K(Y[s_{k-1}, s_k])\leq^+ K(Y\upto n)<\eps n.
\end{align*}
Here the first inequality uses the definition of Martin-L\"of randomness, and the $K(\sigma)\leq^+2|\sigma|$ complexity upper bound. The second inequality is a property of prefix-free Kolmogorov complexity, and the equality is the definition of $Y$. Finally $s_{k-1}$ and $s_k$ can be computed from $n$, for the penultimate inequality. Rearranging, $s_k \leq^+ \eps n + 2s_{k-1}$.

Combining the rearranged inequalities, we have that  \begin{align*}
    |\sigma|^{-1}n
    \leq^+ \eps n + |\sigma|^{-1}s_k
    &\leq^+ \eps n + |\sigma|^{-1}(\eps n + 2s_{k-1}),
    \end{align*}
    so that a bit of algebra gives\begin{align*}
    |\sigma|^{-1}n(1 - \eps-2s_{k-1}/n)&\leq^+ \eps n.
\end{align*}
As $n>s_k$, $2s_{k-1}/n$ is less than $2s_{k-1}/s_k = 2^{-2k+2}$. So
\begin{align*}
|\sigma|^{-1}n(1 - \eps-2^{-2k+2})&\leq^+ \eps n.
\end{align*}
As $n$ increases, so does $k_n$, so by shrinking $\eps$, $1 - \eps-2^{-2k+2}$ can be made as close to 1 as needed. This forces $|\sigma_n|^{-1}<\eps$, so that $N$ computes arbitrarily long prefixes of $R$. As $N$ is computable, given $M$ we can recover arbitrarily long prefixes of $R$, and hence $R$.

If instead $\sigma_n\prec R$ only finitely often, then for large enough $n$, $\tau_n\prec R$. So ``shift'' the $\tau_n$ somewhat: define $\hat{N}= \{s_{k_n}-1\mid n\in N\}$. By definition, coinfinitely many $\sigma_n\prec R$, so $R\leq_T \hat{N}\leq_T N$.

In either case, $R$ is now computable, a contradiction.
\end{proof}

%% file: 2rbpi.tex
\label{sec:rbpi}
Recall \Cref{delembedding}: $A\leq_T B$ iff for all $X\in 2^\omega, \dim_{si\Delta^0_1 (A)} (X) \leq \dim_{si\Delta^0_1 (B)} (X)$. We would like to establish a similar ``if and only if" theorem for $\Pi^0_1$ dimensions:

\begin{conjecture} $A\leq_T B$ iff for all $X\in 2^\omega, \dim_{si\Pi^0_1 (A)} (X) \leq \dim_{si\Pi^0_1 (B)} (X)$.\end{conjecture}

However, there is a central difficulty in adapting the proof: the notion of reals immune to $\Pi^0_1(B)$ sets. Before considering a different setting to avoid this problem, note that we at least get a weak result fairly easily:
\begin{theorem}If for all $X\in 2^\omega, \dim_{\Pi^0_1(A)}(X)\leq \dim_{\Pi^0_1(B)}(X)$, then $A\leq_T B'$.\end{theorem}
\begin{proof}\ \\[-13ex] \begin{align*}
\dim_{si\Delta^0_1(A)}(X) &\leq \dim_{si\Pi^0_1(A)}(X) &\Delta^0_1(A)\subseteq \Pi^0_1(A)\\
&\leq\dim_{si\Pi^0_1(B)}(X) &\text{hypothesis}\\
&\leq\dim_{si\Delta^0_2(B)}(X) &\Pi^0_1(B)\subseteq \Delta^0_2(B)\\
\dim_{si\Delta^0_1(A)}(X)&\leq\dim_{si\Delta^0_1(B')}(X) &\text{relativized Post's Theorem}
\end{align*}
Hence by \Cref{delembedding}, $A\leq_T B'$.\end{proof}

\begin{definition}\label{prin}
The \emph{principal function} of an infinite set $A = \{a_0<a_1<a_2<\cdots\}$ is defined by $p_A(n)= a_n$. For a string $\sigma$, $p_\sigma(n)$ is the position of the $n$th 1 in $\sigma$, and undefined otherwise.
\end{definition}

\begin{definition}
A string $X\in2^\omega$ is \emph{($A$-)computably bounded} if its principal function $p_X$ is bounded above by some ($A$-)computable function $f$ (for all $n$, $p_X(n)\leq f(n)$).
\end{definition}
Write $\widehat{\Sigma}^0_1(A)$ for the $A$-computably bounded $A$-c.e. sets, and similarly $\rbpi$ for the $A$-computably bounded $A$-co-c.e.\ sets. We are motivated to consider these sets by the following observation:
\theorem For all $A\in2^\omega$, $\widehat{\Sigma}^0_1(A) = \Sigma^0_1(A)$.
\begin{proof} The $\subseteq$ inclusion is by definition. For $\supseteq$, we prove the unrelativized version.

Define $X = W_e$, so $X_s = W_{e, s}$. If $X$ is computable, its principal function is computable. If $X$ is not computable, it is infinite, so for each $n$, let $s(n)$ be the least stage when $|X_{s(n)}|\geq n$. For all $s$, elements are never removed from $X_s$, only added, so that $p_{X_{s}}(n)\leq p_{X_{s+1}}(n)$. Thus $p_X(n)\leq \max\{X_{s(n)}\}$, a computable function. \end{proof}

Thus the dual notion to $\Sigma^0_1$ could equally well be taken to be $\Pi^0_1$ or $\rbpi$, depending on the setting. In fact, the two yield distinct notions for dimension. We prove this for the non-relativized, $si$- case:

\theorem There exists an $X$ such that $\dim_{si\Pi^0_1}(X) = 1$ and $\dim_{si\rbpi}(X) = 0$.
\begin{proof} The template of \Cref{pidelsepsi} works here, now using a hypersimple set\footnote{A c.e.~set with hyperimmune (\Cref{hyperimmune}) complement. Every non-computable c.e.~degree contains one \cite{Dekker}.
}.\end{proof}
In this setting, obtaining separation results is as easy as it was for $\Pi^0_1$ dimensions:
\theorem There exists $X$ with $\dim_{si\rbpi}(X) = 1$ and $\dim_{si\Delta^0_1}(X) = 0$.
\begin{proof} Follow \Cref{pidelsepsi} using a simple but not hypersimple set\footnote{Such sets can also be found in every non-computable c.e.\ degree \cite{Yates}} $S_0$.\end{proof}
\theorem There exists $X$ with $\dim_{si\rbpi}(X) = 0$ and $\dim_{si\Delta^0_2}(X) = 1$.
\begin{proof} Since $\rbpi\subseteq \Pi^0_1$, this is a corollary of the $n=1$ case of \Cref{pidelsepsin+1}.\end{proof}
\lemma[\textbf{$\rbpi$ Immunity Lemma}]\label{pihatembedding} If $A\nleq_T B$, there is a $\rbpi(B)$-immune $S\leq_T A$.
\begin{proof}  Let $S$ be the set of prefixes of $A$. Suppose $S$ contains a $B$-co-c.e.~set $C$ that is $B$-computably bounded by $f$. To compute $A(n)$ from $B$, compute $f(n)$ and co-enumerate $C$. The computably bounded condition guarantees that there will be at least $n$ distinct $\sigma_i\in C$ less than $f(n)$ which are never enumerated out, so we can run the co-enumeration until the strings of size $|\sigma|<f(n)$ form a linear order under $\subseteq$. As $C\subseteq S$, these $\sigma_i$ are distinct prefixes of $A$, so they have different lengths. Hence the longest is at least $n$ bits long, giving $A(n)$. Now $A\leq_T B$. Contrapose.\end{proof}
This lemma allows us to establish the following, the desired analogue \Cref{delembedding}:

\begin{theorem}[\textbf{$\rbpi$ Embedding Theorem}]Let $A,B\in \Cant$. Then  $A\le_T B$ iff for all $X\in \Cant, \dim_{si\rbpi (A)} (X) \leq \dim_{si\rbpi (B)} (X)$.
\end{theorem}
\begin{proof}\textbf{[$\Ra$]} Immediate, as $\rbpi(A)\subseteq\rbpi(B)$.
\\
\textbf{[$\La$]} Just as in \Cref{delembedding}, using \Cref{pihatembedding} to provide the appropriate immune set.
\end{proof}

%% file: 3intro.tex
\section{The Motivating Conjecture}
\label{sec:pi0nfail}
While the new $\rbpi$ setting seems to be the correct dual to $\Sigma^0_1$ for our dimension results, it is instructive to examine the difficulty in proving results for the $\Pi^0_1$ case. The desired theorem is

\begin{conjecture}[$\Pi^0_1$ Embedding Theorem]\label{PiEmbedConj} $A\leq_T B$ iff for all $X\in 2^\omega$, $\dim_{si\Pi^0_1(A)}(X)\leq\dim_{si\Pi^0_1(B)}(X)$.
\end{conjecture}
To prove this in a manner analogous to \Cref{delembedding}, we would need to have

\begin{conjecture}[$\Pi^0_1$ Immunity Lemma]\label{PiImLem} If $A\nleq_T B$, there is an $S\in\Pi^0_1(A)$ such that $S$ is $\Pi^0_1(B)$-immune.
\end{conjecture}

We could cast doubt on the theorem by disproving the lemma, but a priori this would only show that this particular proof technique is flawed: the theorem could be true while the lemma is false. Fortunately, this is not the case:
\begin{theorem} \Cref{PiEmbedConj} and \Cref{PiImLem} are logically equivalent.
\end{theorem}
\begin{proof} Define the statements\\[-6.5ex] \begin{align*}
X&: A\leq_T B\\
Y&: \left(\forall X\in 2^\omega\right)\ \dim_{si\Pi^0_1(A)}(X)\leq\dim_{si\Pi^0_1(B)}(X)\text{, and}\\
Z&: \left(\forall S\in\Pi^0_1 (A)\right)\left(\exists C\in\Pi^0_1(B)\right)\ C\subseteq S,\\[-6.5ex]
\end{align*}
so that the theorem is $X\Lra Y$, and the lemma is $\lnot X\Ra \lnot Z$. We wish to show
\begin{center}
$(\lnot X\Ra \lnot Z)\Lra (X\Lra Y).$
\end{center}
It's clear that $X\Ra Y$. In the presence of the Immunity \Cref{immlem}, we can prove the embedding theorem by the usual construction, so $(\lnot X\Ra\lnot Z)\Ra(Y\Ra X)$. Thus to prove their equivalence, it suffices to show $(Y\Ra X)\Ra (\lnot X\Ra\lnot Z)$. Tautologically, this is $\lnot X\Ra (Z\Ra Y)$. In fact, $Z\Ra Y$:
\begin{align*}
Z&\Lra\left(\forall S\in \Pi^0_1(A)\right)\left(\exists C\in\Pi^0_1(B)\right)\ C\subseteq S\\
&\Ra \left(\forall X\in 2^\omega\right)\left(\forall S\in\Pi^0_1(A)\right)\left(\exists C\in\Pi^0_1(B)\right) \inf_{n\in S} \dfrac{K(X\upto n)}{n}\leq \inf_{m\in C}\dfrac{K(X\upto m)}{m}\\
&\Ra  \left(\forall X\in 2^\omega\right)\left(\forall S\in\Pi^0_1(A)\right) \inf_{n\in S} \dfrac{K(X\upto n)}{n}\leq \sup_{M\in \Pi^0_1(B)}\inf_{m\in M}\dfrac{K(X\upto m)}{m}\\
&\Lra \left(\forall X\in 2^\omega\right) \sup_{N\in \Pi^0_1(A)} \inf_{n\in N} \dfrac{K(X\upto n)}{n}\leq \sup_{M\in \Pi^0_1(B)}\inf_{m\in M}\dfrac{K(X\upto m)}{m}\\
&\Lra \left(\forall X\in 2^\omega\right)\  \dim_{si\Pi^0_1(A)}(X)\leq\dim_{si\Pi^0_1(B)}(X)\\
Z&\Ra Y\qedhere
\end{align*}
\end{proof}
\vspace{-.5cm}
The full lemma can be viewed as a relativization of the following statement:
\begin{center} If $A$ is not computable, $A$ co-enumerates a $\Pi^0_1$-immune real.
\end{center}
This motivates our study of $\Pi^0_1$-immunity.

%% file: 3piimmune.tex
\section{$\Pi^0_1$-Immunity and Cohesiveness}

As mentioned in \Cref{sec:separation}, $\Pi^0_1$-immunity (see \Cref{def:im}) is closely related to cohesiveness (see \Cref{pi:coh}). Here we will expand on exactly how.

\definition A coinfinite c.e.\ set $M$ is \emph{maximal} iff for all indices $e$, if $M\subseteq W_e$, then $\overline{W_e}$ is finite or $W_e\setminus M$ is.

This definition comes from considering c.e.\ sets as a lattice under set inclusion, modulo finite differences: a maximal set in the sense above is a maximal element of this lattice.

The following characterization is also commonly used as a definition:
\begin{theorem}\label{cocoh} An infinite c.e. set $M$ is maximal iff its complement is cohesive.\end{theorem}

\begin{proof}
\textbf{[$\Ra$]} As $M$ is c.e., all $W_e\cup M$ are c.e.\ as well. As $M$ is maximal and a subset of $W_e\cup M$, either $\overline{W_e\cup M} = \overline{W_e}\cap\overline{M}$ is finite or $(W_e\cup M)\setminus M = W_e\cap \overline{M}$ is.

\noindent \textbf{[$\La$]} Suppose $M\subseteq W_e$ for some $e$. By cohesiveness, either $W_e\cap\overline{M}$ is finite or $\overline{W_e}\cap \overline{M} = \overline{W_e}$ is.
\end{proof}

Notice that as cohesive sets are not required to be co-c.e., the reverse direction of this theorem connects cohesiveness to only part of our definition of maximality. Indeed cohesive sets are either co-maximal or $\Pi^0_1$-immune:

\theorem\label{neximal} Let $A$ be an infinite set such that for all $e$, if $A\subseteq W_e$, then $\overline{W_e}$ is finite or $W_e\setminus A$ is.\linebreak Then $A$ is c.e.\ iff $A$ is maximal, and $A$ is not c.e.\footnote{If this notion of ``a non-c.e.\ that has has no c.e.\ supersets" does not have a name, we suggest the neologism \emph{neximal}, so called because in the setting of enumerable sets, it is the characterizing property of maximal sets, but here emphasizing sets that are \underline{\textbf{n}}ot \underline{\textbf{e}}numerable.}\ iff $\overline{A}$ is $\Pi^0_1$-immune.
\begin{proof} The first biconditional is the definition of maximality. We prove the second in cases:

\noindent \textbf{[$\Ra$]} If $\overline{W_e}\subseteq \overline{A}$, $A\subseteq W_e$. As $A$ is not c.e., $|W_e\setminus A|=\infty$. Instead, $\overline{W_e}$ is finite.

\noindent \textbf{[$\La$]} As $\Delta^0_1=\Sigma^0_1\cap\Pi^0_1$, $\Pi^0_1$-immune sets are immune. By definition, immune sets are not c.e.
\end{proof}

We will use the following corollary so often that it deserves to be called a lemma:

\begin{lemma}\label{cohpi} Cohesive sets are not co-c.e.\ iff they are $\Pi^0_1$-immune.\end{lemma}

Nevertheless, every cohesive degree is $\Pi^0_1$-immune. To prove this, we need a new piece of notation. Recall \Cref{directsum}:

\begin{olddef}[\ref{directsum}]$A\oplus B = \{2k\mid k\in A\}\cup \{2k+1\mid k\in B\}$.
\end{olddef}
This replaces the $n$th even bit with the $n$th bit of $A$, and similarly the $n$th odd bit with the $n$th bit of $B$. There is nothing special about the even and odd numbers here, we could generalize to an arbitrary set $X$:

\begin{definition}\label{oplus} $A\oplus_X B = \{p_X(n)\mid n\in A\}\cup\{p_{\overline{X}}(n)\mid n\in B\}.$\end{definition}
Now the $n$th $1$ in $X$ is replaced with the $n$th bit of $A$ and the $n$th $0$ in $X$ is replaced with the $n$th bit of $B$. It will be useful to notice that $\overline{A\oplus_X B} = \overline{A}\oplus_{X}\overline{B}$ and $X = \omega\oplus_X \null$.
\theorem\label{hcompPi01im} Every cohesive set $C$ has a $\Pi^0_1$-immune subset $D\equiv_T C$.
\begin{proof} If $C$ is not $\Pi^0_1$, then it is itself $\Pi^0_1$-immune by \Cref{cohpi}. So assume $C$ is $\Pi^0_1.$

Define $D = C\oplus_C\null$. By definition $D = \{p_C(n)\mid n\in C\}$, so that $D=\{p_C(p_C(n))\mid n\in\omega\}$.

Note that as $C$ is cohesive, it is infinite and coinfinite.

As $C$ is infinite, $D$ is infinite by definition. Notice $D\subseteq C$, so that $D\cap C = D$ is also infinite.

As $C$ is coinfinite, $\overline{D}\cap C = (\overline{C}\oplus_C\omega)\cap (\omega\oplus_C \null) = \overline{C}\oplus_C\null$ is infinite by definition.

As $C$ is cohesive, the above shows that $D$ and $\overline{D}$ cannot be c.e., so in particular $D$ is not co-c.e.\ As $D$ is an infinite subset of a cohesive set, it is itself cohesive. By \Cref{cohpi}, $D$ is $\Pi^0_1$-immune.

To see that $D\geq_T C$, fix an index $e$ with $C=\overline{W_e}$, and write $C_s = \overline{W_{e, s}}$. Then $D_s = C_s\oplus_{C_s}\null$ is a $\Delta^0_2$ approximation of $D$. As above, $D_s = \{p_{C_s}(p_{C_s}(n))\mid n\in\omega\}$. To keep track of the $k$th element of $D_s$, define movable markers $m_k$ (for the direct sum over $C$), and let $m_{k,s}$ be the location of $m_k$ at stage $s$. The placements of the markers can be tracked in stages: \begin{itemize}
    \item[-] At stage $0$, the markers are set to $m_{k,0}=k$ for all $k$.
    \item[-] At stage $s+1$, if $n$ is removed from $C_{s+1}$, then for all $k\geq n$, put $m_{k,s+1} = m_{k+1, s}$.
\end{itemize}

In this framework, $\{m_{k,s}\}_{s\in\omega} = C_s$, so that $D_s = \{m_{k,s}\mid k\in C_s\}$.

Let $n_s=p_{D_s}(n)$, the $n$th element of $D_s$. By definition, this is always some element with a marker $m_k$ on it, so it can only increase as stages run, whether because the markers move right, or because for some $\ell<k$, $\ell\in C_s$ but $\ell\not\in C_{s+1}$ (so that at stage $s$, the first $k$ markers contain the first $n$ elements of $D_s$, but at stage $s+1$ one of them leaves $D_{s+1}$). So the $n_s$ are non-decreasing: once the first $n$ bits of $D_s$ agree with $D$, at no later stage do they disagree.

To decide if $x\in C$, let $s$ be the first stage when $x_s = p_{D(x)}$. By construction some marker $m_k=x_s$, so necessarily $k\geq x$. At this stage, markers $m_\ell$ with $\ell\leq k$ never move again, meaning elements $\ell\leq k$ are never again enumerated out of $C$. Thus $C_s$ correctly approximates $C$ up to $k$, so as $k\geq x$, $C_s(x) = C(x)$. As $D$ can find this stage $s$, $D\geq_T C$.
\end{proof}

\begin{corollary}\label{cohdeg}Every cohesive degree contains a $\Pi^0_1$-immune real.
\end{corollary}

\section{Other Immunity Notions}
Cohesiveness is a very strong property, implying a tower of immunity notions. It is natural to wonder where $\Pi^0_1$-immunity falls in this tower, both in terms of degrees and individual reals. For instance, every $\Pi^0_1$-immune real is immune (as $\Delta^0_1\subseteq\Pi^0_1$), but the existence of co-c.e.~immune sets means the converse does not hold.

One slight strengthening of immunity is \emph{hyper}immunity:
\begin{definition}
If $f, g:\omega\ra\omega$, then $f$ \emph{dominates} $g$ if for all but finitely many $n$, $g(n) < f(n)$. If $f$ does not dominate $g$, then $g$ \emph{escapes} $f$, i.e.\ there are infinitely many $n$ such that $g(n) > f(n)$.
\end{definition}
\begin{definition}\label{hyperimmune}
A set $A$ is hyperimmune iff its principal function $p_A$ escapes every computable function. 
\end{definition}
The question of which degrees are (not) $\Pi^0_1$-immune and (not) hyperimmune is ultimately uninteresting, as it merely hinges on whether a degree is below $\null'$. Every non-computable $\Delta^0_2$ degree is hyperimmune \cite{MillerMartin},
but every degree outside $\Delta^0_2$ is $\Pi^0_1$-immune:
\theorem \label{BNDelta} $S(X) = \{\sigma\in2^{<\omega}|\sigma\prec X\}$ is not $\Pi^0_1$-immune iff $X\in\Delta^0_2$.
\begin{proof} Notice that any infinite subset of $S(X)$ suffices to compute $X$.

If there is an infinite $\overline{W_e}\subseteq S(X)$, then $X\leq_T \overline{W_e}\leq_T \null'$. Similarly if $X\not<_T\null'$, then no $\Pi^0_1$ set computes $X$, so every infinite subset of $S(X)$ is not $\Delta^0_2$, let alone $\Pi^0_1$.
\end{proof}
The observation that $S(X)\equiv_T X$ immediately gives two useful corollaries:

\corollary\label{nonDelta2PiIm} Every non-$\Delta^0_2$ degree contains a $\Pi^0_1$-immune real.

\corollary\label{BNPiImDelta2} Every real that computes no $\Pi^0_1$-immune set is $\Delta^0_2$.

The question of which \emph{sets} are (not) $\Pi^0_1$-immune is much more interesting; despite its close relationship to cohesiveness, $\Pi^0_1$-immunity does not imply most immunity notions. In fact this can be witnessed at a level even weaker than hyperimmunity, recently studied by Astor \cite{astor1,astor2}:
\definition The \emph{upper density} of $A$ is $\overline{\rho}(A) = \displaystyle\limsup_{n\ra\infty}|A\upto n|/n$. If for any computable permutation $\pi:\omega\ra\omega$, the value of $\overline\rho(\pi(A))$ is the same, this is the \emph{intrinsic upper density} of $A$. Similarly define \emph{lower density} $\underline\rho$ and \emph{intrinsic lower density} using $\liminf$. If the intrinsic upper and lower densities of $A$ are equal, they are its \emph{intrinsic density}.

These definitions give rise to two new classes of reals: ID0, reals with intrinsic density 0, and ILD0, for intrinsic lower density 0. Their place in the hierarchy of immunity notions is shown in \Cref{fig:densityImmunityHierarchy}.
To see that intrinsic density and $\Pi^0_1$-immunity are incomparable notions, we will make use of the following definitions:
\begin{definition} A function $f$ is \emph{dominant} iff it dominates every computable function.
\end{definition}
\begin{definition}
A real $A$ is \emph{dense immune} iff its principal function $p_A$ is dominant.
\end{definition}

\begin{theorem}\label{ID1}
There is a $\Delta^0_2$, dense immune set whose complement is $\Pi^0_1$-immune.\end{theorem}
\begin{proof}
Let $f(n)$ be a $\null'$-computable dominant function. Without loss of generality, assume $f(n+1)>f(n)>n$ for all $n$. Define $A = \left\{p_{\overline{W_e}}(f(e)): \left|\overline{W_e}\right|\geq f(e)\right\}$.

The $f(e)$th $1$ of any real is necessarily at least $f(e)$. So to determine if $A(n) = 1$, $\null'$ can first find all $e$ such that $f(e)<n$. As $f(n)>n$, there will only be finitely many such $e$. For these indices, $\null'$ can then compute $\varphi_e(k)$ for all $k\leq n$, to determine if $p_{\overline{W_e}}(f(e))= n$. Altogether, $A\leq_T\null'$.

If $p_{\overline{W_e}}(f(e)) = n$, then $A(n) = 1$. By definition, $p_A$ lists these values $p_{\overline{W_e}}(f(e))$ in increasing order.
Then for all $n$, $p_A(n)=p_{\overline{W_k}}(f(k))$ for some $k\geq n$ (the $n$th input such that this function is defined), and $p_A(n) = p_{\overline{W_k}}(f(k))\geq f(k)>f(n)$. As $f$ is dominant, so is $p_A$.

Finally if $\overline{W_e}$ is infinite, then $\overline{W_e}$ has an $f(e)$th element $n$, so that by construction $\overline{A}(n) = 0$. Thus $\overline{A}$ is $\Pi^0_1$-immune.
\end{proof}
\begin{corollary}\label{PiDense1}There is a $\Delta^0_2$, $\Pi^0_1$-immune set with intrinsic density 1.
\end{corollary}
\begin{proof}
As Astor shows in \cite{astor1}, dense immune reals have intrinsic density 0. So taking $A$ to be as in \Cref{ID1}, $\overline{A}$ has intrinsic density 1 and is $\Pi^0_1$-immune.
\end{proof}
\begin{theorem}\label{NotILD0} In the non-computable $\Delta^0_2$ degrees, there are reals of every combination of being (not) $\Pi^0_1$-immune and having intrinsic lower density (greater than) 0.
\end{theorem}

\begin{proof} \renewcommand\qedsymbol{} Cohesive reals have intrinsic density 0, so \Cref{cohpi} gives two cases. \Cref{PiDense1} gives a third, leaving only the case of a non-$\Pi^0_1$-immune real with intrinsic lower density greater than 0, for which any non-immune real suffices.
\end{proof}
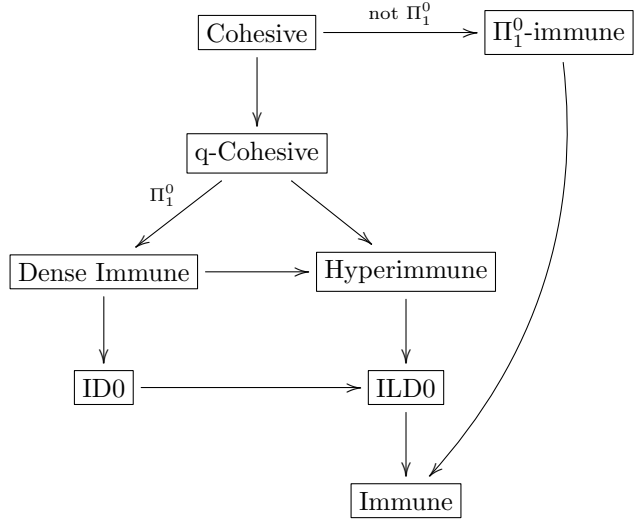
\begin{figure}
\centering
{\ }{
\xymatrix@C=-10pt{
	& *+[F]+\hbox{Cohesive} \ar[d] \ar[rr]^-{\text{not }\Pi^0_1} &&*+[F]+\hbox{$\Pi^0_1$-immune} \ar@/^2pc/[ddddl] \\
	&*+[F]+\hbox{q-Cohesive}\ar[dr]\ar[dl]_-{\Pi^0_1} & &\\
	*+[F]+\hbox{Dense Immune} \ar[d] \ar[rr] & & *+[F]+\hbox{Hyperimmune} \ar[d]\\
	 *+[F]+\hbox{ID0} \ar[rr] && *+[F]+\hbox{ILD0} \ar[d]\\
	& & *+[F]+\hbox{Immune}
	}
}
\caption[Implications between immunity notions]{The graph of implications between immunity notions considered in this section. Note that certain implications only hold when the reals under consideration are (not) $\Pi^0_1$. Implications not proven here are taken from \cite{astor2}.}
\label{fig:densityImmunityHierarchy}
\end{figure}
We can use the same reals to prove
\begin{corollary}\label{NonCoh}
In the non-computable $\Delta^0_2$ degrees, there are reals of every combination of being (not) $\Pi^0_1$-immune and (not) cohesive.
\end{corollary}
Altogether, there is no general relationship between $\Pi^0_1$-immunity and any of the notions considered above. But in the case of non-$\Pi^0_1$ reals, we have that cohesiveness implies $\Pi^0_1$-immunity. A reasonable question arises: does this implication hold for any weaker cohesiveness property? One natural candidate is $q$-cohesiveness:

\definition A set is \emph{quasicohesive} (\emph{$q$-cohesive}) iff it is the union of finitely many cohesive sets.

Classically, being $q$-cohesive implies many other commomly considered immunity properties, such as  (strong) hyperhyperimmunity, or ((finite) strong) hyperimmunity (see Figure 1 of \cite{astor2}).
In fact in the co-c.e.\ case, $q$-cohesiveness even implies dense immunity, so the following theorem rules out a host of possibilities:

\begin{theorem} There is a $q$-cohesive set that is neither $\Pi^0_1$-immune nor $\Pi^0_1$.
\end{theorem}
\begin{proof}
Let $C$ be any $\Pi^0_1$ cohesive set. Every infinite set has a cohesive subset (see Exercise III.4.17 in \cite{OdiI}), so there is a cohesive $C_2\subset\overline{C}$. Every subset of $C_2$ is also cohesive, and there are uncountably many such subsets $D$. They are in bijection with the collection of all sets $C\cup D$, so there are uncountably many of these as well. There are only countably many $\Pi^0_1$ sets, so there must be some $D$ such that $Q=C\cup D$ is not $\Pi^0_1$. As $C$ is not $\Pi^0_1$-immune, neither is $Q$.
\end{proof}
Of course, this result gives no hint whatsoever as to where such a $Q$ might live --- perhaps with an additional restriction, not being $\Pi^0_1$ is enough to guarantee $\Pi^0_1$-immunity. Motivated by \Cref{NotILD0}, we show that being $\Delta^0_2$ is not a sufficient restriction.
\begin{lemma}[Lachlan \cite{LachlanMaximal}]\label{MaxSup}
If $A$ is a coinfinite c.e.\ real with no maximal superset, then $A''>_T\null''$.
\end{lemma}

\begin{theorem}There is a $\Delta^0_2$ $q$-cohesive set that is neither $\Pi^0_1$-immune nor $\Pi^0_1$.
\end{theorem}
\begin{proof}
Let $C$ be a $\Pi^0_1$ cohesive set. As $\overline{C}$ is c.e., it has an infinite computable subset $D$, so that $\overline{D}$ is infinite and coinfinite.
As $D'' \equiv_T \null'$, by \Cref{MaxSup} $\overline{D}$ has a maximal superset $M$. Thus $\overline{M}$ is a cohesive subset of $D$, and hence of $\overline{C}$.
Let $C_2 = \{p_{\overline{M}}(n)\mid n \text{ is even or } p_{\overline{M}}(n)\in W_n\}\subseteq\overline{M}$.
As $\null'$ can compute $p_{\overline{M}}$ and every $W_e$, $C_2$ is $\Delta^0_2$.

If $n$ is even, then $p_{\overline{M}}(n)\in C_2$, so $C_2$ is infinite. If $n$ is odd, $p_{\overline{M}}(n)\in C_2$ iff $p_{\overline{M}}(n)\not\in\overline{W_n}$, so $C_2$ disagrees with every $\Pi^0_1$ set. Infinite subsets of cohesive sets are cohesive, so as $C_2\subset\overline{M}\subset D\subset\overline{C}$, it is cohesive and disjoint from $C$. Thus $C\cup C_2$ is $q$-cohesive and disagrees with every $\Pi^0_1$ set, and so is also not $\Pi^0_1$. But it has the $\Pi^0_1$ subset $C$, so it is not $\Pi^0_1$-immune.
\end{proof}

%% file: 3HighLowErshov.tex
\section{$\Pi^0_1$-Immunity Below $\null'$: Highness and Lowness}\label{highlow}
It is not hard to construct a $\Pi^0_1$-immune set below $\null'$ --- as we will see below, there are low$_n$ and high$_n$ examples for every $n$, even properly so for $n>1$.

\subsection{Lowness}

\begin{definition} $A$ is low$_n$ iff $A^{(n)}\leq_T\null^{(n)}$. For $n=1$ we omit the subscript.
\end{definition}
\begin{definition}
A family of sets $\mathcal{D} = \{D_e\}_{e\in\omega}$ is \emph{uniformly} $\Delta^0_2$ iff $\{\langle x, e\rangle \mid x\in D_e\}\leq_T\null'$.
\end{definition}
\begin{theorem}\label{lowD}
For any uniformly $\Delta^0_2$ family $\mathcal{D}$, there is a low $\mathcal{D}$-immune real.
\end{theorem}
\begin{proof} We will build such a real $A$ in segments $a_0\preceq a_1\preceq\cdots\prec A$ by forcing the jump under certain constraints. This follows a proof originally by Spector \cite{Spector}, as presented in \cite{OdiI} (Proposition V.2.21).

Index the elements of $\mathcal{D}$ as $D_0, D_1, D_2, \dots$ Begin with $a_0=\null$, and a list containing only $0$. At stage $s-1$, we have built a finite string $a_{s-1}$, and have a finite list of indices $e$ providing constraints. At stage $s$, add $s$ to the list and force the jump: ask $\null'$ if there is an extension $\tau\succ a_{s-1}$ such that $\Phi^{\tau}_{s}(s)\!\!\downarrow$. There are two cases to consider:\begin{itemize}
\item[1.] If a suitable extension $\tau$ exists, then for each $e$ on the list, run a $\null'$-computable check to find the least $n_e\in D_e\cap \left[|a_{s-1}|, |\tau|\right)$, if such values exist. Where they do, set $\tau(n_e) = 0$ to diagonalize against $D_e$, then remove $e$ from the list. If this causes $\Phi^{\tau}_s(s)\divs$, make the following query to $\null'$: \begin{center}
Does there exist a $\tau\succ a_s$ satisfying all $\tau(n_e)=0$ such that $\Phi^{\tau}_s(s)\halts$? \end{center}
As there are at most $s$ restrictions at this stage, this is a $\Sigma^0_1$ question. If such a $\tau$ exists, check again with the remaining $e$ on the list, repeating this process until either:\begin{itemize}
\item[a.] Some extension $\tau$ is found which meets all the restrictions imposed by indices on the list. Set $a_{s+1}=\tau$.
\item[b.] The restrictions cause every possible extension $\tau$ to have $\Phi^\tau_s(s)\divs$. In particular $a_s$ as currently defined has this property, so leave it be.
\end{itemize}
\item[2.] If no suitable extension exists, consult the list: if $|a_{s-1}|$ is in any $D_e$ for $e$ on the list, append a 0 to $a_{s-1}$ and remove those $e$ from consideration. Repeat this search-and-append process until the list is empty or the next bit is in none of the remaining $D_e$, at which point append a $1$ and call this string $a_{s+1}$.
\end{itemize}
Define $A = \bigcup_{s\in\omega}a_s$. To see that $A$ is infinite, note that infinitely many indices $e$ code for machines that never halt, so the second case occurs infinitely many times, each time adding a $1$ to $A$.

For $\mathcal{D}$-immunity, let $D_e$ be infinite. It is added to the list at stage $e$, so all $n\in D_e\cap[|a_{e-1}|,\infty)$ could be used to diagonalize. As $\left|D_e\right|=\infty$, this set is non-empty, so the diagonalization works and $D_e\not\subseteq A$.

Finally for lowness, as $\mathcal{D}$ is uniformly $\Delta^0_2$, each stage of the construction requires finitely many queries to $\null'$, so $A\leq_T\null'$. In addition, $e\in A'$ iff $\Phi^{a_{e}}_e(e)\!\!\downarrow$, so $\null'$ computes $A'$.\end{proof}
This theorem actually proves slightly more:
\begin{corollary}\label{PiIm1Gen} For any uniformly $\Delta^0_2$ family $\mathcal{D}$, there is a low $1$-generic $\mathcal{D}$-immune set.
\end{corollary}
\begin{proof} For any index $e$, the above $A$ forces the jump at stage $e$, so $A$ is $1$-generic \cite{ForcetheJump}. \end{proof}
Finally the desired theorem is a corollary.
\begin{corollary}\label{lowPi01set}
There is a low $1$-generic $\Pi^0_1$-immune set.
\end{corollary}
\begin{proof} The $\Pi^0_1$ sets for a uniformly $\Delta^0_2$ family: $\left\{\langle x, e\rangle\mid x\in\overline{W_e}\right\}\leq_T\null'$.
\end{proof}
There are other, stronger lowness notions; for instance this technique can be improved to produce a superlow $\Pi^0_1$-immune set. We will do so in \Cref{superlow1G}.

Finally, \Cref{lowPi01set} gives another way in which $\Pi^0_1$-immunity differs from cohesiveness:
\begin{corollary} The real constructed in \Cref{lowPi01set} is $\Pi^0_1$-immune, but not cohesive.
\end{corollary}
\begin{proof}
Cohesive sets are not low \cite{CooperHhi}.
\end{proof}
\subsection{Highness}
Just as we adapted a proof of Spector in \Cref{lowD}, we could apply the same modification to a proof of Sacks. While we will do this this to construct a high bi-$\Pi^0_1$-immune in \Cref{biPi}, earlier results give several less involved proofs.
\begin{definition} $A$ is high$_n$ iff $A^{(n)}\geq_T\null^{(n+1)}$. For $n=1$ we omit the subscript.
\end{definition}
\begin{theorem}[High Domination Theorem \cite{Martin}]\label{HDT} $A$ is high iff $A$ computes a dominant function.
\end{theorem}
\begin{theorem}\label{PiImhigh}
There is a high $\Pi^0_1$-immune set.
\end{theorem}
\begin{proof} The set in \Cref{ID1} is $\Pi^0_1$-immune and computes a dominant function (namely its own principal function). By \Cref{HDT}, it is high.
\end{proof}
It is a well-known theorem of Martin \cite{Martin} that the high c.e.\ degrees are exactly those containing a maximal (and hence co-cohesive) real. In fact slightly more is true:
\begin{theorem}[Jockusch \cite{JockHighCoh}]\label{highcoh} Every high degree contains a cohesive real.
\end{theorem}
\begin{corollary}\label{highB}
Every high degree contains a $\Pi^0_1$-immune set.
\end{corollary}
\begin{proof}
Combine \Cref{highcoh} and \Cref{cohdeg}.
\end{proof}

\subsection{low$_n$, high$_n$, and Intermediate Sets}
As a warmup, we present another theorem along the lines of \Cref{PiImhigh}. We will make use of the $B=\null$ case of \Cref{oplus}, so note that $A\oplus_X \null = \{p_X(n)\mid n\in A\}$. That is, the $n$th $1$ in $X$ is replaced with the $n$th bit of $A$.
\theorem\label{inchighPi01} There is a high, incomplete, $\Pi^0_1$-immune set.
\begin{proof} Let $L$ be the set constructed in \Cref{lowPi01set}. By relativizing a construction of Sacks \cite{Sacks} to $L$, we can obtain an $H$ with $L<_T H<_TL'\equiv_T\null'$ and $H'\equiv_T L''\equiv_T\null''$, so that $H$ is high and incomplete\footnote{By the upward closure property in \cite{jockusch1973}, this is actually enough.}. As $L$ is $\Pi^0_1$-immune and $H\oplus_L\null\subseteq L$, $H\oplus_L\null$ and hence $L\oplus(H\oplus_L\null)$ are $\Pi^0_1$-immune. Now we can compute \vspace{-.25cm}
\begin{center}$H\leq_T L\oplus(H\oplus_L\null)\leq_T L\oplus H\leq_T H <_T\null'$\end{center}
to see that $L\oplus(H\oplus_L \null)\equiv_T H$ is incomplete and high.\end{proof}
This techinique of combining a set with a highness/lowness property with a $\Pi^0_1$-immune set can be used to show the existence of $\Pi^0_1$-immune degrees into every level of the high/low hierarchy. To that end, we work with psuedojumps $J_e(A)$:

\begin{definition}
$J_e(A) = A\oplus W_e^A$.
\end{definition}

The usual way to populate the high/low hierarchy uses a finite extension argument under $\null'$, so it may be possible to adapt it to produce $\Pi^0_1$-immune sets by adapting that proof (via the same modifications we made to Spector's proof that there is a $\Delta^0_2$ 1-generic in \Cref{lowD}). But it is easier to demonstrate $\Pi^0_1$-immune sets Turing equivalent to $J_e(A)$:
\definition $P_e(A) = A\oplus\left(W_e^A\oplus_A \null\right)$.
\lemma For any index $e$, if a real $A$ is $\Pi^0_1$-immune, then so is $P_e(A)$.
\begin{proof}
It suffices to show that the second summand is $\Pi^0_1$-immune, so notice $W_e^A\oplus_A\null\subseteq A$.
\end{proof}
\begin{lemma}\label{P=J} For any index $e$, $P_e(A)\equiv_T J_e(A)$.\end{lemma}
\begin{proof}
We must show that $A\oplus\left(W_e^A\oplus_A \null\right)\leq_T A\oplus W_e^A$ and $W_e^A\leq_T A\oplus\left(W_e^A\oplus_A\null\right)$.\\
The first inequality follows by definition. For the second, $W_e^A(n) = (W_e^A\oplus_A\null)(p_A(n))$.
\end{proof}
\begin{lemma}\label{2hops} There is a computable $f$ such that for all indices $e$ and sets $B$, $P_{f(e)}(B)>_TB$ and $P_e(P_{f(e)}(B))\equiv_T B'$.\end{lemma}
\begin{proof}
In \cite{Psuedo}, this is shown for $J_e$ in place of $P_e$. So apply \Cref{P=J}.
\end{proof}
\begin{theorem}\label{lownhighn} There are $\Pi^0_1$-immune degrees in every proper level $H_{n+1}\setminus H_n$ and $L_{n+1}\setminus L_n$ of the high/low hierarchy.
\end{theorem}
\begin{proof}
We perform two dovetailing induction steps, following \cite{AMiller}. For the first, suppose for some index $e$ and all reals $X$, $[P_e(X)^{(n)}\equiv_T X^{(n)}$ and $P_e(X)^{n-1}\not\equiv_T X^{(n-1)}]$.
Then for an arbitrary $X$, we can use \Cref{2hops} with $B=P_{f(e)}(X)$ and apply this induction hypothesis to obtain
\begin{center}
    $X^{(n+1)}\equiv_T (X')^{(n)}\equiv_T P_e(P_{f(e)}(X))^{(n)}\equiv_T P_{f(e)}(X)^{(n)}$, and\\
$X^{(n)}\equiv_T (X')^{(n-1)}\equiv_T P_e(P_{f(e)}(X))^{(n-1)}\not\equiv_T P_{f(e)}(X)^{(n-1)}.$
\end{center}
Thus for all $X$, $P_{f(e)}(X)^{(n)}\equiv_T X^{(n+1)}$ and $P_{f(e)}(X)^{(n-1)}\not\equiv X^{(n)}$.

For the second induction, suppose the conclusion of the first: that for an index $e$ and all reals $X$,
    $[P_e(X)^{(n)}\equiv_T X^{(n+1)}$ and $P_e(X)^{n-1}\not\equiv_T X^{(n)}]$.
Then a similar computation shows that for all $X$, $P_{f(e)}(X)^{(n+1)}\equiv_T X^{(n+1)}$ and $P_{f(e)}(X)^{(n)}\not\equiv X^{(n)}$. Note that this is the first induction hypothesis, with $f(e)$ as the index.

Now to populate the high/low hierachy, it suffices to start with an index $i$ and a low, $\Pi^0_1$-immune $L$ that satisfies one of the induction hypotheses. Then $L^{(n)} = \null^{(n)}$, so that in this case the induction steps become\vspace{-.25cm} \begin{center}``if $X$ is properly low$_n$, then $P_{f(e)}(X)$ is properly high$_n$", and\\
``if $X$ is properly high$_n$, then $P_{f(e)}(X)$ is properly low$_{n+1}$".
\end{center}

Let $i$ be a uniform index for the jump, i.e.~for all $X, W_i^X = X'$. Using $i$ in \Cref{2hops}, \vspace{-.5cm} $$X'\equiv_TP_i(P_{f(i)}(X)) \equiv_T J_i(P_{f(i)}(X)) = P_{f(i)}(X)\oplus W^{P_{f(i)}(X)}_i = P_{f(i)}(X)\oplus P_{f(i)}(X)'\equiv_T P_{f(i)}(X)'\vspace{-.5cm}$$
for all $X$. Again by \Cref{2hops}, $P_{f(i)}(X)>_T X$, so that $i$ satisfies the first induction hypothesis. We constructed a low, $\Pi^0_1$-immune $L$ in \Cref{lowPi01set}, so using $e=i$ and $X=L$ populates the high/low hierarchy: $P_{f^{2n+1}(i)}(L)$ is properly low$_n$ while $P_{f^{2n+2}(i)}(L)$ is properly high$_n$.
\end{proof}

Finally we consider the \emph{intermediate} sets, i.e. those $A$ such that for all $n$, $\null^{(n)}<_T A^{(n)}<_T\null^{(n+1)}$.

\begin{theorem} There is a $\Pi^0_1$-immune set of intermediate degree.
\end{theorem}
\begin{proof} As the function $f$ in \Cref{2hops} is computable, it has a fixed point $e$ for which $W_e = W_{f(e)}$, and hence $P_e(A) = P_{f(e)}$ as operators. The lemmata give that the operator $P_e$ commutes with the jump, as $$P_e(A')\equiv_T P_e(P_e(P_{f(e)}(A)))\equiv_T P_e(P_{f(e)}(P_e(A)))\equiv_T P_e(A)',$$
so that $P_{f(e)}(A^{(n)}) =  P_{f(e)}(A)^{(n)}$. Thus for our low $\Pi^0_1$-immune set $L$, $$\null^{(n)}\equiv_T L^n<_T P_{f(e)}(A^{(n)})<_T  P_e(P_{f(e)}(L^{(n)}))\equiv_T L^{(n+1)}\equiv_T \null^{(n+1)}.$$

Replacing $P_{f(e)}(L^{(n)})$ with $P_{f(e)}(L)^{(n)}$ between the inequalities gives the result.
\end{proof}

\section{$\Pi^0_1$-Immunity below $\null'$: The Ershov Hierarchy}\label{Ershov}
\subsection{Definitions and Lemmata}
Having obtained results in the `vertical' stratification of $\Delta^0_2$ sets, we turn our attention `horizontally' to examples in the Ershov hierarchy. This statifies $\Delta^0_2$ reals by how many `mind changes' it takes to build them. For instance, a c.e.\ set $W_e$ can only change its mind about an element $x$ once, from $x\not\in W_{e,s}$ to $x\in W_{e,s+1}$. More formally, these mind changes are tracked by $\Delta^0_2$-approximations, as in the Shoenfield Limit Lemma:

\begin{lemma}[Shoenfield \cite{Shoenfield}]\label{limitlemma} A real $A$ is $\Delta^0_2$ iff there is a computable function $f:\omega^2\ra\omega$ (called a \emph{$\Delta^0_2$-approximation}) such that for all $x$, $\displaystyle\lim_{s\ra\infty}f(x,s) = A(s)$.
\end{lemma}
\begin{definition} A $\Delta^0_2$ real $A$ is \emph{$n$-c.e.}\ iff there is a $\Delta^0_2$-approximation $f$ for $A$ such that for all $x$, $f(x,0) = 0$ and $|\{s\in\omega\mid f(x,s)\neq f(x,s+1)\}|\leq n$.
\end{definition}
Notice that in this framework, the $1$-c.e.~sets are exactly the c.e.~sets, via $f(x,s) = W_{e,s}(x)$. Beyond the finite $n$, we also have the $\omega$-c.e.~sets:
\begin{definition}
A $\Delta^0_2$ real $A$ is \emph{$\omega$-c.e.}\ iff it there is a $\Delta^0_2$-approximation $f$ for $A$ and a computable function $g:\omega\ra\omega$ such that for all $x\in\omega$, $|\{s\in\omega\mid f(x,s)\neq f(x,s+1)\}|\leq g(x)$.
\end{definition}
\begin{definition}
A set is \emph{properly} $\alpha$-c.e. if is $\alpha$-c.e.\ but not $\lambda$-c.e.\ for any $\lambda<\alpha$.
\end{definition}
The following well-known results will be useful. They can be strengthened to iff \cite{Ershov}, but we will only need (and hence prove) one direction.
\begin{lemma}\label{EvenCeDecomp} If $C$ is $2k$-c.e., then there exist c.e.~sets $A_0\supseteq B_0\supseteq\cdots\supseteq A_{2k-1}\supseteq B_{2k-1}$ such that $C = \displaystyle \bigcup_{i\leq k}\left(A_i-B_i\right)$.
\end{lemma}
\begin{proof}
Let $C$ be a $2k$-c.e.~set with $\Delta^0_2$-approximation $f$. For notation, given a natural $n$ let $M_n = \{s\in\omega\mid f(n, s)\neq f(n, s+1)\}$. Define c.e.~sets $A_i = \{n\in\omega: |M_n|\geq 2i+1\}$, and $B_i=\{x\in\omega: |M_n|\geq 2i+2\}$. Clearly these $A_i$ and $B_i$ are nested as desired. Then\\[-6.5ex] \begin{align*}
x\in C&\Lra \lim_{s\ra\infty} f(x,s) = 1\\
&\Lra \exists t\ \forall s\geq t\ f(x, s) = 1\\
&\Lra \exists i\ x\in A_i\land x\not\in B_i\\
x\in C&\Lra x\in \bigcup_{i\leq k}(A_i-B_i)\qedhere
\end{align*}
\end{proof}
\begin{corollary}\label{OddCeDecomp} If $C$ is $(2k+1)$-c.e., it is the union of a $2k$-c.e.\ set and a c.e.~set.
\end{corollary}
\begin{proof}
If $C$ is $(2k+1)$-c.e.\ via the $\Delta^0_2$ approximation $f$, let $A = \{x\in\omega: |M_x|\leq 2k\}$. This is almost all of $C$, but will not include those $x$ such that $|M_x| = 2k+1$. The set $B$ of these elements is enumerable, so $C=A\cup B$.
\end{proof}

We can also change $2k$ to $\omega$ in the statement and proof of \Cref{EvenCeDecomp} to obtain

\begin{lemma}\label{OmegaCeDecomp} If $C$ is $\omega$-c.e., there are c.e.\ sets $A_0\supseteq B_0\supseteq\cdots\supseteq A_{2k-1}\supseteq B_{2k-1}$ such that $C = \displaystyle \bigcup_{i\in\omega}\left(A_i-B_i\right)$.
\end{lemma}

With these lemmata in hand, we will now assume $n$-c.e.\ and $\omega$-c.e.\ sets are of the above forms.
\subsection{$2n$-c.e.\ sets}
\Cref{OddCeDecomp} shows that for odd $n$, $n$-c.e.\ sets cannot be $\Pi^0_1$-immune, as they have a c.e.\ subset, and so fail to be immune. We might hope to make use of \Cref{cohpi}, but a result of \cite{PostErshov} nixes this:

\begin{theorem}\label{ncecoh} For all $n$, if an $n$-c.e.\ set is cohesive, it is $\Pi^0_1$.
\end{theorem}
Instead, we can use major subsets:
\begin{definition} $A\subseteq^* B$ iff $A\cap\overline{B}$ is finite.
\end{definition}
\begin{definition}
Let $B\subseteq A$ be c.e.\ sets. $B$ is a \emph{major subset} of $A$, written $B\subset_m A$, iff $|A-B| = \infty$ and for all $e$, $\overline{A}\subseteq^* W_e\Ra \overline{B}\subseteq^*W_e$.
\end{definition}
\begin{lemma}[Lachlan \cite{LachlanMajor}]\label{cemajor} Every non-computable c.e.\ set has a major subset.
\end{lemma}
\begin{theorem}\label{PiImMaj} If $B\subseteq A$ are c.e.\ sets and $A\not\in\Delta^0_1$, then $A-B$ is $\Pi^0_1$-immune iff $B\subset_m A$.
\end{theorem}
\begin{proof}\textbf{[$\Ra$]} $\Pi^0_1$-immune sets are infinite, so $|A-B| =\infty$. If $\overline{A}\subseteq W_e$, then $\overline{W_e}\subseteq A$, so that $\overline{W_e}\cap\overline{B}$ is a $\Pi^0_1$ subset of $A-B$, and thus finite. Immediately $\overline{B}\subseteq^*W_e$.\\
\textbf{[$\La$]} Let $\overline{W_e}\subseteq A-B$. Then $\overline{W_e}\subseteq A$, so $\overline{A}\subseteq W_e$. Thus $\overline{B}\subseteq^* W_e$, so as $\overline{W_e}\subseteq\overline{B}$, it must be that $\overline{W_e}$ is finite.
\end{proof}
\begin{theorem}\label{2kmajor} A $2n$-c.e.\ set $X = \bigcup_{i=1}^n \left(A_i- B_i\right)$
is $\Pi^0_1$-immune iff $i\leq n$, $B_i\subset_m A_i$ and $A_i\not\in\Delta^0_1$.
\end{theorem}
\begin{proof}\textbf{[$\Ra$]} For all $i$, $(A_i-B_i)\subseteq X$, so as $X$ is $\Pi^0_1$-immune, so is $A_i-B_i$. So by \Cref{PiImMaj}, $B_i\subset_m A_i$ and $A_i\not\in\Delta^0_1$.
\\
\textbf{[$\La$]} Let $\overline{W_e}\subseteq X$. As $X\subseteq A_1$ and $B_1\subset_m A_1$, $\overline{W_e}\subseteq^*B_1$. So $\overline{W_e}\cap \overline{B_1}$ is finite, and hence so is $\overline{W_e}\cap(A_1-B_1)$. Thus $\overline{W_e}\subseteq^*\bigcup_{i=2}^n\left(A_i-B_i\right)$. This is a subset of $A_2$, so we can repeat this reasoning to get that only finitely many elements of $\overline{W_e}$ are in $A_2-B_2$, and indeed in any $A_i-B_i$. As $\overline{W_e} = \overline{W_e}\cap X =  \bigcup_{i=1}^n\left(\overline{W_e}\cap (A_i-B_i)\right)$ and each of these terms is finite, $\overline{W_e}$ is finite.
\end{proof}
\subsection{$\omega$-c.e.\ Sets and Superlowness}\label{sec:omegace}
For $\omega$-c.e.\ sets, while the proof of the reverse direction can no longer rely on there being finitely many terms in the union, the forward direction works just fine:
\begin{theorem}
If an $\omega$-c.e.\ set $X=\bigcup_{i=1}(A_i-B_i)$ is $\Pi^0_1$-immune, then for all $i$, $B_i\subset_m A_i$.
\end{theorem}
\begin{proof} For all $i$, $A_i-B_i\subseteq X$ is $\Pi^0_1$-immune, so $B_i\subset_m A_i$ and $A_i\not\in\Delta^0_1$ by \Cref{PiImMaj}.
\end{proof}
But do any properly $\omega$-c.e.\ $\Pi^0_1$-immune reals exist? In fact, we found one earlier:
\begin{theorem}There is a properly $\omega$-c.e., $\Pi^0_1$-immune real $D$.
\end{theorem}
\begin{proof}
Let $M$ be a maximal set, so that $C=\overline{M}$ is $\Pi^0_1$ and cohesive. Then let $D = C\oplus_C \null$. \Cref{hcompPi01im} shows that $D$ is $\Pi^0_1$-immune, and clearly $D_s(x) = f(x,s) = (C_s\oplus_{C_s}\null)(x)$ is computable. As discussed in the proof of \Cref{hcompPi01im}, the $n$th element of $D_s$ only changes when some $m\leq n$ leaves $C$. So $D$ is $\omega$-c.e.\ via the computable bound $g(x) = x$.

As $D$ is $\Pi^0_1$-immune, by \Cref{cohpi} it is not $\Pi^0_1$. As $D\subseteq C$, $D$ is cohesive, so by \Cref{ncecoh}, $D$ is not $n$-c.e.\ for any $n$.
\end{proof}

The $\omega$-c.e.\ sets are closely tied to another notion of computational weakness, \emph{superlowness}, which is essentially lowness for a stronger notion of oracle reduction.

\begin{definition}\label{tttotal}
$A$ is \emph{truth-table reducible} to $B$ (written $A\leq_{tt}B$) iff $B$ computes $A$ via a total Turing reduction, i.e.\ $A = \Phi^B$ and for all $X\in 2^\omega$ and $n\in\omega$, $\Phi^X(n)$ is defined.
\end{definition}

The following minor lemma goes back to Post (a stronger, unrelativized result is proven in \cite{Post}).

\begin{lemma}\label{ttjump} For all $A\in 2^{\omega}$, $A\leq_{tt}A'$.
\end{lemma}
\begin{proof} Let $e$ be an index for a oracle program that, on input $n$, halts iff its oracle contains $n$. Define $\Phi^X_i(n) = 1$ iff $\langle e, n\rangle \in X$. This is clearly a total reduction, and by definition $\Phi^{A'}_i = A$. 
\end{proof}

The name ``truth-table reduction" derives from an equivalent definition (that we will use later, see \Cref{def:tt}). For more on this (and a proof of the following lemma), see section 3.8.3 in \cite{SoareTuring}.

\begin{definition}
$A$ is \emph{superlow} iff $A'\leq_{tt}\null'$.
\end{definition}
\begin{lemma}\label{ttomega}
$A$ is $\omega$-c.e. iff $A\leq_{tt}\null'$.
\end{lemma}
\begin{corollary}\label{superlowomega}
Superlow sets are $\omega$-c.e.
\end{corollary}
\begin{proof}
If $A$ is superlow, $A\leq_{tt}A'\leq_{tt}\null'$ by \Cref{ttjump}. By definition, $\leq_{tt}$ is transitive, so $A\leq_{tt}\null'$. Apply \Cref{ttomega}.
\end{proof}
As superlowness implies lowness, we adapt the technique of \Cref{lowD} to build a superlow real $G$ that is immune to uniformly $\Delta^0_2$ families. By \Cref{superlowomega}, with a little care we can get another example of a properly $\omega$-c.e.~$\Pi^0_1$-immune real.

\begin{theorem}\label{superlow1G}
For every uniformly $\Delta^0_2$ family $\mathcal{D}$, there is a superlow, $\mathcal{D}$-immune, $1$-generic real.
\end{theorem}
\begin{proof} Let $\mathcal{D} = \{D_e\}_{e\in\omega}$ be a uniformly $\Delta^0_2$ family.
Let $f(\langle x, e\rangle, s)$ be a $\Delta^0_2$-approximation for $\{\langle x, e\rangle\mid x\in D_e\}$, and define $f_{e, s}(x) = f(\langle x, e\rangle, s)$.

Build such a real $A$ in segments $a_0\preceq a_1\preceq\cdots\prec A$ as in \Cref{lowD}. There are two types of requirements to be met: \vspace{-.75cm}
\begin{align*}
J_e&: \Phi_e^A(e)\halts\text{ or }(\exists \sigma\prec G)(\forall\tau\succ\sigma) \Phi_e^\tau(e)\divs\\
N_e&: \text{If $D_e$ is infinite, then }D_e\not\subseteq A.
\vspace{-.75cm}
\end{align*}
\vspace{-1cm}

We will write $X_e$ to mean an arbitrary requirement of either type. Order the requirements in decreasing order of priority as $J_0<N_0<J_1<N_1<\cdots$.

Begin at stage $0$ with $a_0 = \null$. At each stage of the construction, give some requirement attention as defined below. Do this in such a way that over the course of the construction, each requirement receives attention infinitely often.

At stage $s+1$, we have defined a finite string $a_s$. When giving a requirement $X_e$ attention at this stage, simulate all $\Phi^{a_s}_{i, s}(i)$ for $i\leq e$. Define $u_{s+1}=\max\{\varphi^{a_s}_{i}(i)\mid i\leq e\text{ and }\Phi^{a_s}_{i,s}(i)\halts\}$. Let $\sigma_{s+1}$ be the shortest prefix of $a_s$ longer than $e$, $u_{s+1}$, and any $n$ mentioned by $N_i$ for $i<e$ (see below). Then do as follows, depending on the type of requirement:

\noindent \textbf{$J_e$ requirements}
\begin{enumerate}
    \item Search reverse lexicographically for a $\tau\succ\sigma_{s+1}$ of length $|a_s|+s$ such that $\Phi^{\tau}_{e,s}(e)\halts$.
    \item If such a $\tau$ is found, set $a_{s+1}=\tau$.
\end{enumerate}
\noindent\textbf{$N_e$ requirements}
\begin{enumerate}
    \item If there is a least $n\in [|\sigma_{s+1}|, s+1]$ such that $f_e(n, s+1)=1$, set $a_{s+1}(n) = 0$.
\end{enumerate}

\noindent\textbf{Verification:} Inductively assume that there is a stage $s$ when $X_e$ is given attention when all lower priority $X_i<X_e$ requirements have been satisfied, so that they never again make changes to $A$. As they never act again, $\sigma_t=\sigma_s$ for all later stages $t$ when $X_e$ is given attention. As higher priority requirements never act again and lower priority requirements cannot interfere with $X_e$, it suffices to show $X_e$ eventually meets its requirement at some stage $t\geq s$.

\noindent\textbf{Claim:} Every $J_e$ meets its requirement.

\noindent \emph{Proof:}
Suppose there exists an extension $\tau\succ\sigma_s$ such that $\Phi_e^\tau(e)\halts$. At each stage when $J_e$ is given attention, it searches for such an extension, so as $J_e$ is given attention infinitely often, there is a large enough stage $t$ such that some $\tau\succ\sigma_s$ is found and $a_t=\tau$.

If no such extension exists, then whenever $J_e$ is given attention, it makes no change to $\sigma_s$.

In either case, $A\succ \sigma_s$ and $J_e$ meets its requirement.

\noindent\textbf{Claim:} Every $N_e$ meets its requirement.

\noindent \emph{Proof:}
Suppose $D_e$ is infinite, so that there is a least $n$ greater than $|\sigma_{s+1}|$ such that $n\in D_e$. At some stage $t\geq s$ when $N_e$ is given attention, $N_e$ sets $a_t(n)=0$. At all subsequent stages, $N_e$ never changes another bit of $A$, so it meets its requirement.



\noindent\textbf{Claim:} $A$ is $1$-generic and $\mathcal{D}$-immune.

\noindent \emph{Proof:}
Every $J_e$ and $N_e$ requirement is met, respectively.

\noindent \textbf{Claim:} $A$ is superlow.

\noindent \emph{Proof:} Define $g(e, s) = 1$ iff $\Phi_{e, s}^{a_s}(e)\halts$, so that $e\in A'$ iff $\lim_{s\ra\infty} g(e, s) = 1$. The only times $g(e, s)$ changes are when $J_e$ is injured, whereupon it changes at most twice: first to $0$, if $\Phi^{a_s}_{e, s}(e)\divs$, then to $1$ if $J_e$ finds an extension which causes the computation to converge. As there are $2e$ requirements that can change the value of $e$, at most $2^{2e}-1$ injuries can occur. Thus $A'(e)$ changes at most $2^{2e}$ times, so $A'$ is $\omega$-c.e. By \Cref{ttomega}, $A'\leq_{tt}\null'$.
\end{proof}
\begin{lemma}\label{nceuniform}
The collection $\mathcal{E}$ of all $n$-c.e.\ reals is uniformly $\Delta^0_2$.
\end{lemma}
\begin{proof}
For all $e\in\omega$, decompose $e$ as $\langle n, \vec{x}\rangle$, where $\vec{x} = \langle x_0, \dots, x_{n-1}\rangle$. Let $f_e(y, s) = 1$ iff $$y\in (\cdots((W_{x_0,s}\cap\overline{W_{x_1,s}})\cup W_{x_3,s})\cdots )\cup W_{x_{n-1},s}$$ if $n$ is odd, and $\cdots\cap W_{x_{n-1},s}$ if $n$ is even. Define $E_e = \{y\mid \lim_{s\ra\infty} f_e(y, s) = 1\}$. The Ershov hierarchy contains exactly the Boolean combinations of c.e.~sets \cite{Ershov}, which this enumeration exhausts. Finally the halting problem computes every c.e.~and every co-c.e.~set, so $\mathcal{E}$ is uniformly $\Delta^0_2$. 
\end{proof}
\begin{corollary}
There is a superlow, properly $\omega$-c.e., $\Pi^0_1$-immune, $1$-generic real.
\end{corollary}
\begin{proof} By \Cref{nceuniform}, we can use $\mathcal{E}$ in \Cref{superlow1G}. Every $\Pi^0_1$ set is 2-c.e.~$\left(\overline{W_e} = \omega-W_e\right)$, so $\Pi^0_1\subseteq\mathcal{E}$ and the resulting real $A$ is $\Pi^0_1$-immune. For all $n$, $A$ is immune to $n$-c.e.\ sets, so $A$ is not $n$-c.e.\ for any $n$. By \Cref{ttjump}, $A\leq_{tt} A'\leq_{tt} \null'$. By \Cref{superlowomega}, $A$ is $\omega$-c.e.
\end{proof}
\section{Bi-$\Pi^0_1$-Immunity Below $\null'$}\label{biPi}
So far, we have not encountered any bi-$\Pi^0_1$-immune reals, and the obvious candidates are $\Delta^0_3$ at best (for instance a Martin-L\"of random relative to $\null'$, which is bi-$\Delta^0_2$-immune). To obtain $\Delta^0_2$ such reals, we can extend the proof technique of \Cref{lowPi01set}:
\begin{theorem}\label{lowbiPi}
There is a $\Delta^0_2$ bi-$\Pi^0_1$-immune. In particular there is a low, properly 1-generic, bi-$\Pi^0_1$-immune real.
\end{theorem}
\begin{proof} In the proof of \Cref{lowD}, whenever adding constraints to the list, add $e$ to the list twice, as $(e,0)$ and $(e,1)$ to handle $\Pi^0_1$- and co-$\Pi^0_1$-immunity, respectively. Rather than diagonalizing against an index $e$ by setting a certain bit to $0$, we diagonalize against $(e,i)$ by instead setting the relevant bit to $i$.

Unlike the original proof, restrictions may now conflict, as they no longer all prescribe the same value. To organize the construction, we give them a lexicographic priority ordering, so that lower priority restrictions cannot interfere with bits assigned by higher priority requirements. The modified procedure still forces the jump, ensuring 1-genericity and lowness, but we need to check that the priority ordering ensures bi-$\Pi^0_1$-immunity.

Suppose $\overline{W_e}$ is infinite, and for induction let $s$ be a stage large enough that\begin{enumerate}
    \item[-] all $(i,1-k)<(e,0)$ representing infinite $\overline{W_i}$ have been diagonalized against, and
    \item[-] all $(i,1-k)<(e,0)$ representing finite $\overline{W_i}$ have $\max\left\{n\in\overline{W_i}\right\}<|a_s|$.
    \end{enumerate}
At this stage, all the diagonalizations that could interfere with the restriction imposed by $(e,0)$ have been performed, so every $n\in \overline{W_e}\cap[|a_{s}|,\infty)$ could be used to satisfy the $(e,0)$ requirement. As $\overline{W_e}$ is infinite, this set is non-empty, so some $x$ will be found and $(e,0)$ will be removed from the list. Similarly for $(e,1)$.

Finally if $\overline{W_e}$ is infinite, as $(e,0)$ is removed from the list, there is an $n$ such that $n\in \overline{A}\cap\overline{W_e}$, so that $\overline{W_e}\not\subseteq A$. Similarly the removal of $(e,1)$ from the list ensures an $n\in A\cap\overline{W_e}$, so that $\overline{W_e}\not\subseteq\overline{A}$, and $A$ is bi-$\Pi^0_1$-immune.
\end{proof}
To obtain a high bi-$\Pi^0_1$-immune set, we modify a proof of Sacks' incomplete high degree \cite{Sacks}, as presented in \cite{OdiII} (Proposition XI.1.11).
\newpage
\begin{theorem}\label{highbiPi}There is an incomplete, high, bi-$\Pi^0_1$-immune real.\end{theorem}
\begin{proof} Let $A<_T\null'$ be a $\Pi^0_1$-immune real, and define an $A$-recursive bijection between $\omega^2$ and $A$ by composing the principal function of $A$ with $\langle\cdot,\cdot\rangle$: $\langle x,t\rangle_A = p_A\left(\langle x, t\rangle\right).$  Fix an index $j$ with $W_j^{\null'} = \null''$. We will use $\null'$ to build a real $B$ as a union of finite $\null'$-computable partial functions $f_s$, and ensure that columns $\langle x, t\rangle_A$ encode $W^{\null'}_{j,t}(x)$ cofinitely often, so that by the Limit \Cref{limitlemma}, $A\oplus B'$ can compute $\null''$. 

First, some notation: say that $\sigma\in2^\omega$ is an \emph{$i$-usable} extension of $f_s$ if $\sigma$ and $f_s$ agree whenever they are both defined, and if for $x<i$, $\sigma$ obeys the restrictions $\langle x, t \rangle_A = W^{\null'}_{j,t}(x)$ wherever $f_s(\langle x, t \rangle_A)$ is not already defined.

Beginning with $f_0=\null$, we can proceed in stages. At stage $s+1$ we have a finite partial function $f_s$ and a finite list of restrictions $(e,k)$ not yet diagonalized against, prioritized lexicographically. Write $b_s$ for the longest initial segment of $f_s$, i.e.\ $|b_s|$ is least such that $f_s(|b_s|)$ is undefined. 
\begin{enumerate}
\item[-] At odd stages $s=2d+1$, add $(d,0)$ and $(d,1)$ to the list. Then pick the least $i\leq d$ that has not yet been \emph{attended to}\footnote{We would ordinarily say the index has been ``diagonalized", but that terminology is already in use in this proof.} (defined below), and that meets the condition
\begin{center}$\exists z\ \exists \text{ $i$-usable extensions } \tau_0, \tau_1\text{ of $f_s$ such that } \Phi^{\tau_0}_i(z)\halts\neq \Phi^{\tau_1}_i(z)\halts$.\end{center}
Let $\tau$ be the extension whose computation disagrees with $\null'(z)$. Now consult the list: for the highest priority $(e, k)$ on the list, search for the least
$n_i\in \overline{W_i}\cap [|b_s|,|\tau|-1]\cap \overline{A}.$
If such an $n_i$ exists, set $\tau(n_i) = k$ to diagonalize against $\overline{W_i}$, then remove $(e,k)$ from the list. If this causes $\Phi^{\tau}(z)\!\!\uparrow$ or $\Phi^{\tau}(z)=\null'(z)$, check again for $\tau_0,\tau_1,$ and $z$ meeting the above condition. Then consult with the remaining $(e,k)$ on the list, and repeat this process until either:\begin{itemize}
\item[a.] $z,\tau_0,$ and $\tau_1$ are found that meet the restrictions and the condition. Again set $\tau$ to have the computation that disagrees with $\null'$, and \emph{attend to} $i$ by setting $f_{s+1} = f_s\cup \tau$.
\item[b.] No $z$ and $i$-usable extensions that meet the restrictions. Proceed to the next stage.
\end{itemize}
\item[-] At even stages $s=2d+2$, for any $x, t<s$ where $f_s(\langle x, t \rangle_A)\!\!\uparrow$, put $f_{s+1}(\langle x, t \rangle_A) = W^{\null'}_{j,t}(x).$
\end{enumerate}
\textbf{Claim:} $\displaystyle\lim_{t\ra\infty} B(\langle x, t \rangle_A) = \null''(x)$.

\noindent \emph{Proof:} Fix $x$. It suffices to show that cofinitely many $t$ have $\langle x, t \rangle_A=W^{\null'}_{j, t}(x)$. As diagonalizations avoid $A$ (and hence all $\langle x, t\rangle_A$), the only times this equality could fail are when $i$-usable extensions $\tau$ code over $\langle x, t \rangle_A$. But this can only happen when indices $i\leq x$ are attended to, as for $i>x$, $i$-usable extensions preserve the coding. As each $i$ is attended to at most once, there are at most $x$ values of $t$ such that $\langle x, t\rangle_A\neq W^{\null'}_{j, t}(x)$, as desired.

\noindent \textbf{Claim:} $B<_T\null'$ and $\null''\leq_T B'$.

\noindent \emph{Proof:} By assumption, $A\leq_T\null'$. So for odd stages of the construction of $B$, $\null'$ suffices to search for $i$-usable extensions, check for $n_i\in\overline{W_i}$, and decide whether $\Phi_i^{\tau}(z)\!\!\downarrow$. For even stages, $\null'$ computes the $A$-recursive $\langle\cdot,\cdot\rangle_A$ and enumerates $\null''$. Altogether, $B\leq_T \null'$.

To see that the inequality is strict, fix $e$ such that $\Phi^{ B}_e$ is total. Let $s$ be an odd stage large enough that that for all $x<e$ and $t\geq s$, $W^{\null'}_{j,t}(x) = \null''(x)$, and such that for all indices $i<e$ that are attended to, this happens before stage $s$.

If we attend to $e$, then immediately $\Phi^{ B}_e\neq\null'$, so suppose not, that for all $x$, the $e$-usable extensions $\tau$ of $f_s$ such that $\Phi_e^{\tau}(x)\!\!\downarrow$ 
all agree. Notice that any extension of $f_s$ that is a prefix of $B$ is $e$-usable, since we have chosen $s$ large enough that attending to any remaining index cannot injure the relevant $\langle x, t\rangle_A$ columns, which for $t>s$ are constant and equal to $\null''(x)$.

Finally to compute $\Phi^{B}_e(x)$ it suffices to search for any $e$-usable extension $\tau$ of $f_s$ with $\Phi^{\tau}_e(x)\!\!\downarrow$. As $A$ can compute all locations $\langle x, t\rangle_A$, and since $s$ is large enough that subsequent $t$ have $\langle x, t\rangle_A = W^{\null'}_{j,t}(x)=\null''(x)$, $A$ and $f_s$ is suffice to decide $e$-usability. As $f_s$ is finite, there is an index $a$ such that $\Phi^{ B}_e=\Phi^A_a\leq_T A<_T\null'$.

For highness, as $A$ computes the encoding $\langle\cdot,\cdot\rangle_A$, by \Cref{limitlemma}  $\null''\leq_T A\oplus B'\leq_T\null'\oplus B'\equiv_T B'$.

\noindent\textbf{Claim:} $B$ is bi-$\Pi^0_1$-immune.

\noindent \emph{Proof:} Suppose $(e,k)$ has that $\overline{W_e}$ is infinite, and for induction let $s$ be a stage large enough that\begin{enumerate}
    \item[-] all $(i,1-k)<(e,k)$ representing infinite $\overline{W_i}$ have been diagonalized against, and
    \item[-] all $(i,1-k)<(e,k)$ representing finite $\overline{W_i}$ have $\max\left\{n\in\overline{W_i}\right\}<|b_s|$.
    \end{enumerate}
At this stage, all the diagonalizations that could interfere with the restriction imposed by $(e,k)$ have been performed, so any $n\in \overline{W_e}\cap [|b_s|,\infty) \cap \overline{A}$ could be used to satisfy said restriction. As $\overline{W_e}\cap [|b_s|,\infty)$ is an infinite $\Pi^0_1$ set and $A$ is $\Pi^0_1$-immune, some $n$ will be found and $(e,k)$ will be removed from the list. As in \Cref{lowbiPi}, the removal of $(e,0)$ and $(e,1)$ from the list guarantees the bi-$\Pi^0_1$-immunity of $B$.\qedhere
\end{proof}

%% file: 3BN.tex
\newpage
\section{Randomness, Genericity, and Typicality}
\label{sec:BN}

Here we connect $\Pi^0_1$-immunity to several commonly studied notions in computability theory.

\begin{definition}
A real $R$ is \emph{weakly $n$-random} iff $R$ is a member of all $\Sigma^0_n$ classes with measure $1$.
\end{definition}

\begin{definition}
A real $G$ is \emph{weakly $n$-generic} iff $G$ meets every dense $\Sigma^0_n$ set of strings.
\end{definition}

\begin{definition}A real $X$ is \emph{weakly n-typical} if $X$ is in every full measure $\Sigma^0_1(\null^{n-1})$ set.\begin{footnote} {Equivalently, $X$ is Kurtz random relative to $\null^{n-1}$ \cite{KurtzRandom}.}
\end{footnote}\end{definition}

Weak $n$-typicality is mentioned in \cite{bn1g}, albeit without a definition.

Notice that both weak $n$-randomness and weak $n$-genericity imply weak $n$-typicality, which matches the intuition that generic strings and random strings are both ``typical" reals. Indeed a number of proofs about weakly 2-randoms are also true for weakly 2-typical reals. For instance, relativizing a result of Jockusch (see Kurtz \cite{KurtzRandom}) shows that weakly 2-randoms are bi-immune to $\Delta^0_2$ sets, so it follows that they are $\Pi^0_1$-immune. In fact being weakly 2-typical suffices:

\begin{theorem} Weak $2$-typicality implies $\Pi^0_1$-immunity.
\end{theorem}
\begin{proof} Let $Y\in2^\omega$ be weakly 2-typical, let $W_e$ be coinfinite, and consider \vspace{-.5cm}\begin{align*} U_e &=\{X\in2^\omega\mid\overline{W_e}\not\subseteq X\}\\
&=\{X\in2^\omega\mid(\exists n\not\in W_e)\ n\not\in X\}\\
&= \bigcup_{n\not\in W_e}\bigcup_{|\sigma|=n}[[\sigma0]].\end{align*}
As $\null'$ can decide $n\not\in W_e$, this is a $\Sigma^0_1(\null')$ class, and as $W_e$ is coinfinite, $\mu(U_e) = 1$. Now $Y\in U_e$, so $\overline{W_e}\not\subseteq Y$. As this holds for every coinfinite $e$, $Y$ is $\Pi^0_1$-immune.\end{proof}

In fact by closely examining a result in \cite{DownMin}, we can say more about weakly 2-typical reals:

\begin{theorem}[\cite{SolovayMartin}] Every weakly 2-typical real $Y$ forms a minimal pair with $\null'$. That is, if $X\leq_T Y$ and $X\leq_T\null'$, then $X$ is computable. \end{theorem} 
\begin{proof} Let $A$ be $\Delta^0_2$, with approximation $A_t$ by the Limit \Cref{limitlemma}. Let $T$ be a weakly 2-typical set and $\Phi_e$ a Turing reduction such that $A = \Phi^T_e$. Notice that $T$ is contained in the $\Pi^0_2$ class\vspace{-.5cm}
$$S = \{X\mid \forall n\forall s \exists t>s\ \Phi^X_{e, t}(n)\halts\ = A_t(n) \}.\vspace{-.5cm}$$
It cannot be that $\overline{S}$ has measure 0, as then $T\in \overline{S}$, so $\mu(S)\neq 0$.

Let $r\in\Q$ be such that $\frac{1}{2}\mu(S) < r < \mu(S)$, and select a finite set $F\subseteq 2^{<\omega}$ with $\mu[[F]]<\mu(S)$ and $\mu([[F]]\cap S)>r$. For any $n$, we can search a finite $G\subseteq 2^{<\omega}$ such that $\mu[[G]]>r$, $[[G]]\subseteq [[F]]$, and all $\sigma, \tau \in G$ have $\Phi^\sigma_e(n)\halts = \Phi^\tau_e(n)\halts$. Such a set will exist, since $\mu([[F]]\cap S)>r$, so this process is computable. It cannot be that $\Phi^\sigma(n)_e\neq A(n)$, as then $\mu([[F]]\cap S)\leq \mu([[F]]\setminus[[G]])<\mu(S)-r<r.$ This search shows that $A$ is computable, so $T$ computes no non-recursive $\Delta^0_2$ sets.\end{proof}
\corollary\label{w2tdelta2} Weakly 2-typical sets are not $\Delta^0_2$.

In a sense we have `upper bounds' in the randomness and genericity notions for being guaranteed to be $\Pi^0_1$-immune. In fact this bound is tight for genericity --- while there are $\Pi^0_1$-immune sets that are 1-generic but not weakly 2-generic (see \Cref{PiIm1Gen}), we'll show now that there are also 1-generics that are not $\Pi^0_1$-immune.

\begin{definition} For a set $A$ in a class $\mathcal{C}$ we say that $I\subseteq\N$ is \emph{$\mathcal{C}$-indifferent for $A$} if no matter how we change bits of $A$ at locations in $I$, the resulting real is still in $\mathcal{C}$.
\end{definition}
Let $\mathrm{MLR}$ and $\mathrm{1G}$ be the classes of Martin-L\"of random and 1-generic reals, respectively.
\begin{theorem}[Figueira, Miller, Nies \cite{Indifferent}]\label{RandomIndiff} Every low $R\in\mathrm{MLR}$ has an infinite $\Pi^0_1$ subset that is $\mathrm{MLR}$-indifferent for $R$.

\end{theorem}
\begin{theorem}[Fitzgerald \cite{IndifferentDay}]\label{GenericIndiff} Every $\Delta^0_2$ 1-generic $G$ has an infinite $\Pi^0_1$ subset that is $1$G-indifferent for $G$.
\end{theorem}
\theorem There are $1$-random and $1$-generic reals that are not $\Pi^0_1$-immune.
\begin{proof} For a low random or $\Delta^0_2$ 1-generic, let $I$ be as given in \Cref{RandomIndiff} and \Cref{GenericIndiff}. Set all the bits of $I$ to 1, so that the resulting real has $I$ as a subset.\end{proof}

%% file: 3lowness.tex
\section{Reals That Can (Not) Co-Enumerate a $\Pi^0_1$-Immune}
Having studied $\Pi^0_1$-immunity, we now return to \Cref{PiImLem}. In its unrelativized form, it says 

\begin{conjecture}\label{PiBNconj}
The only reals that do not co-enumerate a $\Pi^0_1$-immune set are computable.
\end{conjecture}

While we do not settle this conjecture, we make substantial progress towards it. In particular, our work in \Cref{Ershov} allows us to eliminate many degrees from contention:
\newpage
\theorem\label{ceBPiIm} If $A$ computes a c.e.\  $B\not\in\Delta^0_1$, then $A$ co-enumerates a $\Pi^0_1$-immune real.
\begin{proof}
By \Cref{cemajor}, $B$ has a major subset $C$. By \Cref{PiImMaj}, $B-C$ is $\Pi^0_1$-immune. Since $B\in\Delta^0_1(A)$ and $C\in\Sigma^0_1\subseteq\Sigma^0_1(A)$, both $B$ and $\overline{C}$ are $\Pi^0_1(A)$, so that $A$ co-enumerates $B\cap\overline{C} = B-C$.
\end{proof}
\begin{corollary}\label{randomcoPi}
If $A$ computes a Martin-L\"of random $R$, then $A$ co-enumerates a $\Pi^0_1$-immune real.
\end{corollary}
\begin{proof}
If $A\not\in\Delta^0_2$, it computes (and hence co-enumerates) a $\Pi^0_1$-immune set by \Cref{BNPiImDelta2}.

If $R$ is random and $R\leq_T A\in \Delta^0_2$, then $R\in\Delta^0_2$. So $R$ bounds a non-computable c.e.\ set \cite{Kucera}. 
\end{proof}
We can strengthen \Cref{ceBPiIm} by relativizing:
\begin{theorem} If $A$ computes sets $B$ and $X$ such that $B\in\Sigma^0_1(X)-\Delta^0_1(X)$, then $A$ co-enumerates a $\Pi^0_1$-immune real.
\end{theorem}
\begin{proof}
As $B$ is $\Sigma^0_1(X)-\Delta^0_1(X)$, relativizing \Cref{cemajor} gives $B$ a subset $C\in \Sigma^0_1(X)$ that is $X$-major, i.e.\ if $\overline{B}\subseteq W_e^X$, then $\overline{C}\subseteq^* W_e^X$. Relativizing \Cref{PiImMaj}, $B-C$ is $\Pi^0_1(X)$-immune, and so $\Pi^0_1$-immune. Since $B\in\Delta^0_1(A)$ and $C\in\Sigma^0_1(X)\subseteq\Sigma^0_1(A)$, both $B$ and $\overline{C}$ are $\Pi^0_1(A)$, so that $A$ co-enumerates $B\cap\overline{C} = B-C$.
\end{proof}

The following lemmata allow us to restate this result quite cleanly:

\begin{definition}
A real $A$ is \emph{computably enumerable in and above} (CEA) if there is an $X<_T A$ such that $A$ is $X$-c.e. We also say $A$ is CEA($X$) for that $X$.
\end{definition}

\begin{lemma}\label{CEAequiv} For any real $A$ the following are equivalent:\begin{enumerate}
    \item For all $B, X\leq_T A$, if $B\not\leq_TX$, then $B\not\in\Sigma^0_1(X)$.
    \item $A$ computes no CEA $B$.
\end{enumerate}
\end{lemma}
\begin{proof} Certainly (i) implies (ii), as $B>_T X$ implies $B\not\leq_T X$. For the reverse, suppose (ii), and let $B,X\leq_T A$ with $B\not\leq_T X$. As $B\oplus X\leq_T A$, $B\oplus X$ is not CEA($X$). But $X<_T B\oplus X$, so $B\oplus X$ must not be $X$-c.e.\ Trivially, $X$ is $X$-c.e., so $B$ must not be.
\end{proof}
\begin{corollary}\label{PiBNPiBnCEA}
If $A$ is does not co-enumerate a $\Pi^0_1$-immune real, then $A$ bounds no CEA $B$.
\end{corollary}

The following lemma is well-known: for a proof, see Theorems 2.24.9 and 8.21.15 in \cite{DaH}.

\begin{lemma}\label{1GCEA} $A$ computes a $1$-generic $G$ iff $A$ computes a CEA $B$.
\end{lemma}

\begin{corollary}
If $A$ is does not co-enumerate a $\Pi^0_1$-immune real, then $A$ bounds no 1-generic $G$.
\end{corollary}

One possible way to improve this result would be to weaken the notion of genericity in the conclusion to \emph{weak} $1$-genericity. However this merely yields a previous conjecture:
\begin{conjecture}\label{w1gconj}
If $A$ is does not co-enumerate a $\Pi^0_1$-immune real, then $A$ does not compute any weakly 1-generic $G$.
\end{conjecture}
\begin{theorem}\Cref{PiBNconj} and \Cref{w1gconj} are equivalent.
\end{theorem}
\begin{proof}
If the first conjecture holds, then as $A$ is computable, it computes no immune set, and hence no hyperimmune set. Weakly 1-generic sets are hyperimmune \cite{KurtzGen}, so $A$ bounds no weakly 1-generic.

If the second conjecture holds, then as hyperimmune degrees bound weakly 1-generic sets \cite{KurtzGen}, $A$ does not compute any hyperimmune set. By \Cref{BNPiImDelta2}, $A$ is $\Delta^0_2$, and every non-computable $\Delta^0_2$ degree computes a hyperimmune real \cite{MillerMartin}. So $A$ must be computable.
\end{proof}

Now we turn to the case of computing no $\Pi^0_1$-immune. Here we encounter a stark contrast between c.e.\ and non-c.e.\ degrees. In the latter case, $\Pi^0_1$-immune sets exist at every level of the low$_n$ and high$_n$ hierarchies (\Cref{lownhighn}). But in the former case, we have the following:

\begin{theorem}[due to D. Turetsky \cite{Turetsky}]\label{lowceBN} Every low c.e.\ $A$ computes no $\Pi^0_1$-immune set.
\end{theorem}

We will need the iconic Recursion Theorem of Kleene \cite{kleene} to prove this:
\begin{theorem}[The Formal Recursion Theorem]
For any total computable function $f$, there is an index $e$ such that $\varphi_e = \varphi_{f(e)}$.
\end{theorem}
Or, in its more commonly used form:
\begin{theorem}[The Informal Recursion Theorem] When defining a c.e.\ set, without loss of generality we may assume we know the index of that set.
\end{theorem}
\begin{proof}[Proof of \Cref{lowceBN}] Suppose $A$ is a low c.e. set, and that $B = \Phi^{A}$ is infinite. As $A$ is low, the Limit \Cref{limitlemma} gives a computable $g$ such that for all $n, \lim_{s\ra\infty}g(n, s) =A'(n)$. We will build a sequence of $A$-c.e.~sets $\langle V_{n}^{A}\rangle_{n\in\omega}$, and by the recursion theorem we will assume we already know their indices. As determining whether $V_{n}^{A}=\null$ is $\Sigma_{1}^{0}(A)$, there is a total computable function $f$ such that $f(n) \in A' \Leftrightarrow V_{n}^{A}\neq\null.$

We will build a $\Pi_{1}^{0}$ set $C\subseteq B$, meeting the following requirement for every $e$:\vspace{-.25cm}\begin{itemize}
\item[$R_{e}$:] $C$ contains an element greater than $e.$
\end{itemize}

To ensure $C$ is $\Pi^0_1$, we will only remove elements from it.

\noindent Strategy for $R_{e}$:\begin{itemize}
\item[(1)] Wait for a stage $s$ when there is some $x>e$ with $\Phi_{s}^{A_{s}}(x) = 1$ and $x\in C_{s}.$
At such a stage, choose the oldest such computation. Let $\sigma \prec A_{s}$ be the use of this computation. $x$ and $\sigma$ are now our {\it chosen} element and use, respectively.
\item[(2)] Enumerate $0$ into $V_{e}^{\sigma}$, with use $\sigma.$
\item[(3)] Wait until one of the following occurs at some stage $s$:\begin{itemize}
\item[(a)] $\sigma$ is no longer an initial segment of $A_{s}$. In this case, return to step 1.
\item[(b)] $g(f(e), s)=1$. In this case, remove all elements $y<x$ from $C$, except those elements which have been chosen by an $R_{i}$-strategy for some $i\leq e$, then proceed to step 4.\end{itemize}
\item[(4)] Wait until some stage $s$ when $g(f(e), s) = 0$ and $\sigma$ is no longer an initial segment of $A_{s}$. While waiting, begin running the $R_{e+1}$-strategy. If $g(f(e), s) =0$ occurs, discard $x$ (it is no longer {\it chosen}), terminate all $R_{j}$-strategies for $j>e$, and return to step 1.
\end{itemize}
The construction begins by starting the $R_{0}$-strategy at stage $0$, and proceeds from there.

\noindent \textbf{Claim 1:} For each $e$, $R_e$ eventually waits forever at step 4.

\noindent \emph{Proof:}
Induction on $e$. Suppose this holds for all $i<e$. As after some stage each $R_i$ never again returns from step 4 to step 1, eventually $R_e$ is never again terminated. As $\lim_{s}g(f(e), s)$ converges, $R_e$ cannot pass through steps 3b and 4 infinitely often, so fix a stage $s_0$ after which $R_e$ never again returns from step 4 to step 1.

From stage $s_0$ until the strategy returns to step 4, no elements are removed from $C$. By assumption, $B$ is infinite, so $B\cap C_{s_0} \neq\null$. Fix the element of $B\cap C_{s_{0}}$ greater than $e$ with oldest $\Phi^{A}$ computation. Call this element $z$, and let $\tau$ be the use of this true computation. As $R_e$ returns from step 3a to step 1 when $\sigma\not\prec A_s$, this computation will eventually be the oldest, and so $x=z$ will chosen with use $\sigma=\tau\prec A$. Then at step 2, $0$ is enumerated into $V_{e}^{\tau} \subseteq V_{e}^{A}.$ Thus $V_{e}^{A} \neq \null$, and so $\lim_s g(f(e), s) = 1$. Thus the strategy will eventually reach step $3\mathrm{b}$ and so step 4, and so will wait forever at step 4.

\noindent\textbf{Claim 2:} For each $e$, the $R_{e}$-strategy's final chosen element is an element of $C.$

\noindent \emph{Proof:} By construction, the chosen element $x$ is an element of $C_{s}$ at the stage it is chosen. No lower priority strategy can remove $x$ at a later stage, while no higher priority strategy will ever act again.

\noindent\textbf{Claim 3:} For each $e$, if the $R_e$-strategy reaches step 4 with chosen element $x$ and use $\sigma \prec A$, then the strategy waits forever at step 4 with this element.

\noindent \emph{Proof:} By construction, if we reach step 4 at stage $s$, the there is some $t<s$ with $\sigma\prec A_t$. As $A$ is c.e., it can never move away from a true initial segment, so $\sigma\prec A_r$ for all $r>t$, and so we never return to step 1.

\noindent\textbf{Claim 4:} For each $e$ the $R_e$-strategy's final chosen element is in $B$.

\noindent \emph{Proof:} Towards a contradiction, suppose we are waiting forever at step 4 with a chosen use $\sigma\not\prec A$. By the contrapositive of the previous claim, all prior chosen uses were also not initial segments of $A$. So by construction, $V^A_e = \null$, and so $\lim_s g(f(e), s) = 0$. So we will eventually see what we are waiting for at step 4, and will return to step 1, contrary to assumption. Now our final $\sigma \prec A$, and since $B =\Phi^A$ and $\Phi^\sigma(x) = 1$, it follows that $x\in B$.

\noindent\textbf{Claim 5:} If $x\in C$, it is the final chosen element of some strategy.

\noindent \emph{Proof:} Suppose $y$ is not a final chosen element of any strategy. Eventually all $R_i$-strategies with $i < y$ will have settled on their final element, while $R_j$-strategies with $j\geq y$ are not permitted to choose $y$. So eventually there will be a stage after which $y$ is never again chosen. When some large strategy later reaches step 3b, $y$ will be removed from $C$.\end{proof}

Note that by \Cref{PiBNPiBnCEA}, every c.e.~set co-enumerates a $\Pi^0_1$-immune set, so \Cref{lowceBN} cannot be strengthened to settle \Cref{PiBNconj}.

\section{Other Lowness Notions}
In this section we consolidate a number of results about lowness notions, with an eye toward their relation to those those reals which cannot compute/co-enumerate a $\Pi^0_1$-immune real. In doing so, we define several new lowness notions related to highness, maximality, and domination, that arose in the course of trying to prove \Cref{PiBNconj}. Of independent interest is a new characterization of the hyperimmune-free degrees as those that do not compute a truth-table CEA degree.\footnote{This result is claimed without proof in Kjos-Hanssen's computability diagram \cite{bn1g}.}

\newpage
\subsection{Definitions}
Many of these have appeared in earlier sections, but we gather them here for convenience.

\noindent
$A$ is $\Pi^0_1$-immune iff $|A|=\infty$ and $A$ has no infinite $\Pi^0_1$ subset.\\
$A$ bounds no member of a class $\mathcal{C}$ (BN$\mathcal{C})$ iff for all $C\in\mathcal{C}$, $C\not\leq_T A$.\\
$A$ enumerates no member of a class $\mathcal{C}$ ($\mathrm{\Sigma^0_1BN\mathcal{C}}$) iff for all $C\in\mathcal{C}$, $C\not\in\Sigma^0_1(A)$\\
$A$ co-enumerates no member of a class $\mathcal{C}$ ($\mathrm{\Pi^0_1}$BN$\mathcal{C}$) iff for all $C\in\mathcal{C}$, $C\not\in\Pi^0_1(A)$.\\
$A$ is $\mathrm{low}_n$ iff $A^{(n)} \equiv_T \null^{(n)}$.\\
$A$ is $\mathrm{GL}_n$ iff $A^{(n)}\equiv_T(A\oplus\null')^{(n-1)}$.\\
$A$ is $\mathrm{high}$ iff $A'\geq_T\null''$.\\
$A$ is $\mathrm{Low(High)}$ iff any high $B$ has that $(B\oplus A)'\geq_T A''$ ($B$ is \emph{high for} $A$). \\
$A$ is $\mathrm{Low(High\ c.e.)}$ iff any high c.e.\ $B$ has that $(B\oplus A)'\geq_T A''$.\\
$A$ is $\mathrm{Low(Max)}$ iff every maximal (high c.e.) degree contains an $A$-maximal set.\\ 
$A$ is $\mathrm{Low(Dom)}$ iff every function $f$ that dominates all $\Delta^0_1$ $g$ also dominates all $\Delta^0_1(A)$ $h$.\\
$A$ is hyperimmune-free ($\mathrm{HIF}$) iff any $f\leq_T A$ is dominated by a $\Delta^0_1$ function.\\
$A$ is 1-generic (1G) iff it meets or avoids every $\Sigma^0_1$ set of strings.\\
$A$ is weakly 1-generic (W1G) iff it meets every dense $\Sigma^0_1$ set of strings.\\
$A$ is 1-random (1R) iff it is Martin-L\"of random.\\
$A$ is (tt)CEA iff there is a $B<_T A$ (resp. $B<_{tt} A$) such that $A\in\Sigma^0_1(B)$.

\begin{centering}
\begin{figure}
\[
\xymatrix{
&&&\mathrm{\Sigma^0_1BN\Pi^0_1IM}\ar@{<->}[d]^-{3}
\\
&&[A\text{-max}\Ra\text{c.e.}]\ar[ur]^-{2}\ar[d]&\mathrm{BN\Pi^0_1IM}\ar@{.>}[dr]^-4\ar@{.>}[dddr]^>>>>>>>>>>>>>>5
\\
&[A\text{-max}\Lra\text{Max}]\ar[ddd]\ar[ur] &[A\text{-max}\Ra\Delta^0_2]\ar@{.>}[rr]^<<<<<<<<<{6}\ar@{.>}[d]_-7\ar@{<->}[dr]_-8&&\Delta^0_2
\\
\mathrm{\Pi^0_1BN\Pi^0_1IM}\ar@{.>}[dddd]\ar[ur]^-1\ar@{.>}[ddddr]_-{18}&&\mathrm{GL}_1\ar@{<->}[r]_{\Delta^0_2}\ar@{.>}[d]&\text{low}\ar@{.>}[d]\ar[uu]^<<<<<<{9}_<<<<<<{\text{c.e.}}
\\
&&\mathrm{GL}_2\ar@{<->}[r]_-{\Delta^0_3}&\text{low}_2\ar[dr]\ar@{.>}[r]&\text{not high}
\\
&\mathrm{Low(Max)}
\ar[r]_-{10}&\mathrm{Low(High\ c.e.)}\ar[u]^-{11}\ar[urr]&\mathrm{Low(High)}\ar[l]&A''\leq_TA'\oplus\null''\ar[l]^-{12}
\\
&\mathrm{BN1G}\ar@{<->}[d]^-{16}&\mathrm{BNWIG}\ar@{.>}[l]\ar@{<->}[r]^-{15}\ar@{<->}[d]_{17}&\mathrm{HIF}\ar[r]\ar@{.>}_-{13}[ur]\ar@{<->}[d]^-{14}&\Delta^0_1\text{ or not }\Delta^0_2
\\
\mathrm{BN1R}&\mathrm{BNCEA}\ar@{.>}[l]^-{\Delta^0_2}_-{19}&\mathrm{BNttCEA}\ar@{<->}[r]&\mathrm{Low(Dom)}
}
\]
    \caption[Implications between lowness notions]{Implications between lowness notions. Dotted implications are strict. Some implications only hold for certain reals, denoted on the arrows.}
    \label{fig:lowness}
\end{figure}
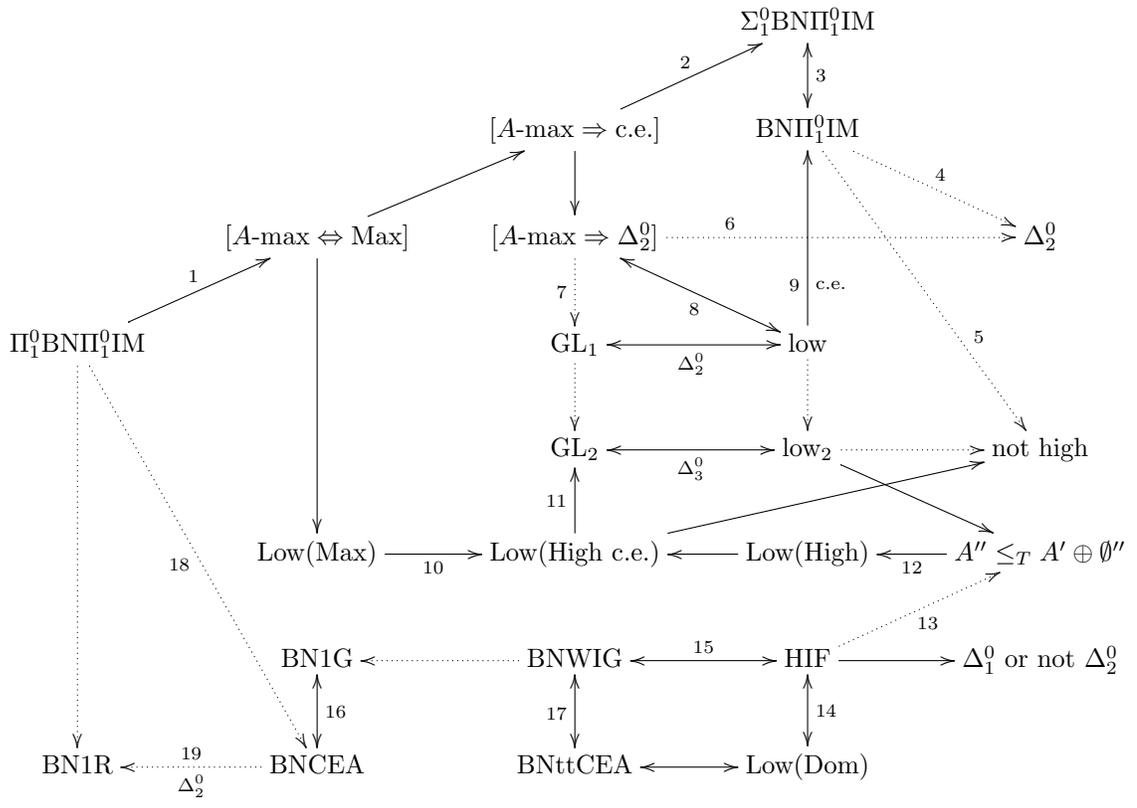
\end{centering}
\newpage
\subsection{Proofs}
\textbf{1.} Let $A\in\mathrm{\Pi^0_1BN\Pi^0_1IM}$, and let $B$ be $A$-maximal. As $\overline{B}$ is $A$-cohesive and $\Pi^0_1(A)$, it is cohesive and not $\Pi^0_1$-immune. By \Cref{cohpi}, $\overline{B}$ is $\Pi^0_1$, so $B$ is $\Sigma^0_1$. As $\Sigma^0_1\subseteq\Sigma^0_1(A)$, $M$ is maximal.

Now let $M$ be maximal, and suppose $M\subseteq W_e^A$. As $\overline{W_e^A}\subseteq\overline{M}$, a cohesive set,  $\overline{W_e^A}$ is cohesive as well. Since $\overline{W_e^A}\in\Pi^0_1(A)$, it is not $\Pi^0_1$-immune, and so must be $\Pi^0_1$ by \Cref{cohpi}. Now as $W_e^A$ is c.e., it is either cofinite or only finitely extends $M$. As $W_e^A$ was arbitrary, $M$ is $A$-maximal.\\
\\
\textbf{2.} Let $B\in\Sigma^0_1(A)$. $B$ has an $A$-computable subset $C$, so $\overline{C}$ has an $A$-maximal superset $M$ by \Cref{MaxSup}. By hypothesis $M$ is c.e., so as $\overline{M}\subseteq C\subseteq B$, $B$ is not $\Pi^0_1$-immune.
\\ \\
\textbf{3.} Certainly if $A$ cannot enumerate a $\Pi^0_1$-immune, it cannot compute such a set. But if $A$ can enumerate a $\Pi^0_1$-immune $B$, it has an $A$-computable subset $C$ which inherits $\Pi^0_1$-immunity.\\
\\
\textbf{4.} See \Cref{BNPiImDelta2}. As there are $\Delta^0_2$ $\Pi^0_1$-immune reals, the reverse implication fails.\\
\\
\textbf{5.} This is the contraposition of \Cref{highB}. We constructed a low $\Pi^0_1$-immune $G$ in \Cref{lowPi01set}, so the reverse implication fails.
\\
\\
\textbf{6.} Suppose $A$-maximal sets are $\Delta^0_2$, and let $B\leq_T A$. Relativizing \Cref{MaxSup}, there is an $A$-maximal $M$ with $\overline{B}\subseteq M$. As $M$ is $\Delta^0_2$, so is $\overline{M}$, so as $\overline{M}\subseteq B$, $A$ computes no $\Delta^0_2$-immune. Now as in \Cref{BNDelta}, as $S(A)$ is not $\Delta^0_2$-immune, $A\in\Delta^0_2$.
\\
\\
\textbf{7.} Suppose $A$-maximal reals are $\Delta^0_2$. Relativizing Corollary 2 of \cite{Yates}, there is an $A$-maximal $M$ with $A\oplus M\equiv_T A'$, so that $A\oplus\null'\equiv_T A'$.\\
\\
\textbf{8. [$\Ra]$} Let $A$-maximal sets be $\Delta^0_2$. By \textbf{6}, $A$ is $\Delta^0_2$. By \textbf{7}, $A$ is $\mathrm{GL}_1$. So $A'\leq_T A\oplus\null'\leq_T\null'$.\\
\textbf{[$\La]$} If $A$ is low then any $A$-maximal set $M$ has $M\leq_T A'\leq_T\null'$.
\\
\\
\textbf{9.} See \Cref{lowceBN}.\\
\\
\textbf{10.} Let $A$ be $\mathrm{Low(Max)}$, and let $H$ be a high c.e.\ set. By Martin \cite{Martin}, there is a maximal $M$ with $H\equiv_T M$. As $M$ is maximal, its degree contains an $A$-maximal real, so it is $A$-high.\\
\\
\textbf{11.} If $A$ is Low(High c.e.), then as $\null'$ is a high c.e.\ set, it is $A$-high, so that $(\null'\oplus A)'\geq_T A''$.\\
\\
\textbf{12.} If $A''\leq_T A'\oplus \null''$ and $B'\geq_T\null''$, then $(B\oplus A)'\geq_T B'\oplus A'\geq_T \null''\oplus A'\geq_T A''$.
\\
\\
\textbf{13.} This is well-known, see for instance Theorem 5.16 of \cite{HIFJump}. This cannot be reversed, as $\Delta^0_2$ low$_2$ degrees are hyperimmune or computable \cite{MillerMartin}.
\\
\\
\textbf{14. [$\Ra$]} If $A$ is hyperimmune-free, then any $f\leq_T A$ is dominated by some computable function. Any function that dominates all computable functions thus dominates $f$.\\
\\
\textbf{[$\La$]} If $A$ is hyperimmune, some $g\leq_T A$ is not dominated by any computable function. Fix an enumeration $\{f_i\}_{i=1}^\infty$ of the computable functions.
Define $F_i(n) = \max\{f_k(n)\mid k\leq i\}$, and notice that if $i<j$ then for all $n$, $F_i(n)\leq F_j(n)$. Each $F_i$ is a computable function, so the hyperimmunity of $g$ guarantees the existence of an increasing sequence $\{n_i\}_{i=1}^\infty$ such that $g(n_i)>F_i(n_i)$.\linebreak For $n\in[n_i,n_{i+1})$, define $h(n) = F_i(n)$. Now for a fixed $i$ and $m\geq n_i$, $h(m)\geq F_i(m)\geq f_i(m)$.\linebreak As $i$ was arbitrary, $h$ is dominant. But for all $i$, $g(n_i)>F_i(n_i) = h(n_i)$, so $g$ escapes $h$ infinitely often and $A$ is not low for domination.\\
\\
\textbf{15.} The weakly 1-generic degrees are exactly the hyperimmune degrees (Corollary 2.10 of \nolinebreak \cite{KurtzGen}).\\
\\
\textbf{16.} See \Cref{1GCEA}.
\\
\\
\textbf{17. [$\Ra$]} Let $A$ bound no weakly 1-generic. By \textbf{15}, $A$ is hyperimmune-free. Certainly $A$ bounds no 1-generic, so by \textbf{16}, $A$ bounds no CEA degree. Let $X<_{tt}B\leq_T A$ with $B\not\leq_{tt}X$. Then $X$ is also hyperimmune-free, so $B\leq_T X\iff B\leq_{tt} X$ (see Theorem 8 of \cite{HIFttT}). Thus $B\not\leq_T{X}$, so as $A$ bounds no CEA degree, $B\not\in\Sigma^0_1(X)$. As $X$ and $B$ were arbitrary, $A$ bounds no ttCEA degree.
\\
\\
\textbf{[$\La$]} We adapt the proof that every 1-generic is CEA given in \cite{Generic} (Theorem 2.24.9).

Define a total functional $\Theta^X = \{\langle i, j\rangle\mid i\in X\land \langle i, j\rangle\not\in X\}$, so that for any $X$, $\Theta^X\leq_{tt}X$. Note also that $X$ is $\Theta^X$-c.e.

Let $\Phi$ be a total reduction, and let $f(n)$ be a computable function bounding its use. \linebreak Let $\rho\in 2^{<\omega}$, and define $i = \langle |\rho|, 0\rangle >|\rho|$ so that for all $j$, $\rho(\langle i, j\rangle)$ is undefined\footnote{Here we are using that $\langle x, y\rangle \geq \max\{x, y\}$, a property of the Cantor pairing function.}. Let $\sigma\succ\rho$ be of length at least $f(i)$ such that $|\rho|, i\not\in\sigma$ and for all $j$, if $n=\langle i, j\rangle\leq|\sigma|$, then $n\in\sigma$. Finally define $\tau=\sigma$ except $i\in\tau$.

As $\sigma$ and $\tau$ only disagree on $i = \langle |\rho|, 0\rangle$, but $|\rho|\not\in \sigma, \tau$, we have that $\Theta^{\sigma}(i) = \Theta^{\tau}(i)=0$, so that $\Theta^{\sigma} = \Theta^{\tau}$. Hence $\Phi^{\Theta^\sigma}=\Phi^{\Theta^\tau}$, so either $\Phi^{\Theta^{\tau}}(i)=0\neq \tau(i)$ or $\Phi^{\Theta^{\sigma}}(i)=1\neq \sigma(i)$.

As $\rho$ was arbitrary, $\{\sigma\mid \exists n<|\sigma|\  \Phi^{\Theta^\sigma}(n)\neq \sigma(n)\}$ is a dense $\Sigma^0_1$ set of strings. Now any weakly 1-generic $A$ meets this set, i.e.~there is an $n$ such that $\Phi^{\Theta^A}(n)\neq A(n)$. So $\Phi$ does not truth-table compute $A$ from $\Theta^A$, so as $\Phi$ was arbitrary, $A\not\leq_{tt}\Theta^A$.\\
\\
\textbf{18.} See \Cref{PiBNPiBnCEA}. To see that the reverse implication does not hold, consider a non-computable, hyperimmune-free $A$. By \textbf{15} and \textbf{16}, $A$ bounds no CEA real. But $A$ is necessarily not $\Delta^0_2$, so by \textbf{4}, $A$ computes (and thus co-enumerates) a $\Pi^0_1$-immune set.\\
\\
\textbf{19.} Every random $R\in\Delta^0_2$ computes a non-computable c.e.\ set \cite{Kucera}. To see that the reverse does not hold, every random $R$ computes a fixed-point free function $f$: for all $e$, $W_{f(e)}\neq W_e$ \cite{Kucera2}. By Arslanov’s Completeness Criterion \cite{arslanov}, any $W_i$ computes a fixed point free function iff $W_i\equiv_T\null'$. So every $W_i<_T\null'$ is CEA, but does not compute a random $R$.

%% file: 4Either.tex
\renewcommand{\MLR}{\mathrm{MLR}}
\renewcommand{\C}{\mathcal C}
\newcommand{\D}{\mathcal D}
\noindent The material in this section (except \Cref{SomeManyOne}) previously appeared in print in \cite{CiE2022}.
\renewcommand{\B}{\mathcal{B}}
\section{The Kolmogorov--Loveland Randomness Problem}

A major open problem of algorithmic randomness asks whether each Kolmogorov--Loveland random (KL-random) real is Martin-L\"of random (ML-random). Recall that a real $A$ is Martin-L\"of random iff there is a positive constant $c$  so that for any $n$, the Kolmogorov complexity of the first $n$ bits of $A$ is at least $n-c$, (that is, $\forall n,\ K(A_i\upto n)\geq n-c$).
	
KL-randomness is most commonly defined using martingales, which we will not have cause to consider here. In brief: $A$ is KL-random iff no computable nonmonotonic martingale succeeds on it. There is also a martingale characterization of ML-randomness --- $A$ is ML-random iff no c.e.\ martingale succeeds on it. For more on this approach to the study of algorithmic randomness, see sections 6.3 and 7.5 of \cite{DaH}.

Instead, we will examine a generalization of KL-randomness, motivated by the following result: one can compute an ML-random real from a KL-random real \cite{MR2183813} and even uniformly so \cite{KHW}. This uniform computation succeeds in an environment of uncertainty, however: one of the two halves of the KL-random real is already ML-random and we can uniformly stitch together a ML-random without knowing which half. Here we pursue this uncertainty and are concerned with uniform reducibility when information has been hidden in such a way.
	Namely, for any class of reals $\C\subseteq 2^\omega$, we write
	\[
		\Either(\C) = \{A\oplus B : A\in \C \text{ or }B\in \C\},
	\]
	where $A\oplus B$ is as in \Cref{directsum}.
	For notation, we often refer to `even' bits of such a real as those coming from $A$, and `odd' bits coming from $B$.
	
	An element of $\Either(\C)$ has an element of $\C$ available within it, although in a hidden way. We are not aware of the $\Either$ operator being studied in the literature, although
	Higuchi and Kihara \cite[Lemma 4]{HIGUCHI20141201} (see also \cite{HIGUCHI20141058}) considered the somewhat more general operation $f(\C,\D)=(2^\omega\oplus\C)\cup (\D\oplus 2^\omega)$, where $\mathcal{A}\oplus\mathcal{B} = \{A\oplus B\mid A\in\mathcal{A}\text{ and } B\in\mathcal{B}\}$.

\begin{definition}
    Let $\B$ and $\C$ be subsets of $2^{\omega}$. $\B$ is \emph{Medvedev} or \emph{strongly reducible} to $\C$, written $\B\leq_s\C$, iff there is a uniform reduction $\Phi$ such that for all $B\in\B$, $\Phi^B\in\C$. $\B$ is \emph{Muchnik} or \emph{weakly reducible} to $\C$ iff for any $B\in\B$, there is a reduction $\Phi$ such that $\Phi^B\in\C$.
\end{definition}
These partial orders induce degree structures on the subsets of $2^\omega$, just as Turing reducibility induces a degree structure on subsets of $\omega$.

As $\MLR\subseteq\KLR$ \cite{MLKL}, KLR is trivially Medvedev reducible to MLR via the identity function. In \cite{KHW}, $\Either$ is implicitly used to show the reverse, that MLR is Medvedev reducible to KLR. In fact it shows something slightly stronger:

\begin{definition}
Let $r$ be a subscript in \Cref{fig:table}, such as $r=tt$. Write $\le_{s,r}$ to denote strong reducibility using $r$-reductions, and $\le_{w,r}$ for the corresponding weak reducibility.
\end{definition}

	\begin{theorem}\label{ref-cie}
	    $\MLR\le_{s,tt}\Either(\MLR)$.
	\end{theorem}
	\begin{proof}
	    \cite[Theorem 2]{KHW} shows that $\MLR\le_{s,tt}\KLR$. The proof demonstrates that $\MLR\le_{s,tt}\Either(\MLR)$ and notes, by citation to \cite{MR2183813}, that $\KLR\subseteq\Either(\MLR)$. 
	\end{proof}
	In fact, the proof shows that the two are truth-table Medvedev equivalent. A natural question is whether they are Medvedev equivalent under any stronger reducibility.

    Letting $\DIM_{1/2}$ be the class of all reals of effective Hausdorff dimension 1/2, 
    \Cref{ref-cie} is a counterpoint to Miller's result $\MLR\not\le_{w}\DIM_{1/2}$ \cite{ExtractInfo}, since $\MLR\not\le_{s,tt}\DIM_{1/2}\supseteq\Either(\MLR)$.

	\begin{definition}\label{def:tt}
		Let $\{\sigma_n\mid n\in\omega\}$ be a uniformly computable list of all the finite propositional formulas in variables $v_1,v_2,\dots$.
		Let the variables in $\sigma_n$ be $v_{n_1},\dots,v_{n_d}$ where $d$ depends on $n$.
		We say that $X\models \sigma_n$ if $\sigma_n$ is true with 
		$X(n_1),\dots,X(n_d)$ substituted for $v_{n_1},\dots,v_{n_d}$.
		A reduction $\Phi$ is a \textbf{truth-table} reduction if there is a computable function $f$ such that for each $n$ and $X$, $n\in\Phi^X$ iff $X\models \sigma_{f(n)}$. 
	\end{definition}

	As shown in \Cref{fig:Degtev}, the next three candidates to strengthen the result (by weakening the notion of reduction under consideration) are the positive, linear, and bounded truth-table reducibilties. Unfortunately, any proof technique using $\Either$ will no longer work, as for these weaker reducibilities, $\MLR$ is not Medvedev reducible to $\Either(\MLR)$.

\section{The Failure of Weaker Reducibilities}
	When discussing the variables in a table $\sigma_{f(n)}$, we say that a variable is of a certain parity if its index is of that parity, e.g.\ $n_2$ is an even variable. As our reductions operate on $2^\omega$, we identify the values $X(n_i)$ with truth values as $1=\top$  and $0=\bot$.

	\subsection{Positive Reducibility}
	\begin{definition}
		A truth-table reduction $\Phi$ is a \textbf{positive} reduction if the only connectives in each $\sigma_{f(n)}$ are $\lor$ and $\land$.
	\end{definition}
	\begin{theorem}\label{positive}
		$\MLR\not\le_{s,p}\Either(\MLR)$.
	\end{theorem}
	\begin{proof}
		Let $\Phi$ be a positive reduction. By definition, for each input $n$, $\sigma_{f(n)}$ can be written in conjunctive normal form:
		$\sigma_{f(n)} = \bigwedge_{k=1}^{t_n} \bigvee_{i=1}^{m_k}v_{f(n),i,k}$. We say that a clause of $\sigma_{f(n)}$ is a disjunct $\bigvee_{i=1}^{m_k}v_{f(n),i,k}$. There are two cases to consider:
		
		\noindent\emph{Case 1:} There is a parity such that there are infinitely many $n$ such that every clause of $\sigma_{f(n)}$ contains a variable.

		Without loss of generality, consider the even case. Let $A = \omega\oplus R$ for $R$ an arbitrary random real.
		Each $\bigvee_{i=1}^{m_k}v_{n,i,k}$ that contains an even variable is true.
		So for the infinitely many $n$ whose disjunctions all query an even variable, $\sigma_{f(n)} = \bigwedge_{k=1}^{t_n} \top = \top$.
		As these infinitely many $n$ can be found computably, $\Phi^A$ is not immune, and so not random.

		\noindent\emph{Case 2:} For either parity, for almost all inputs $n$, there is a clause of $\sigma_{f(n)}$ containing only variables of that parity.

		Set $A = R\oplus \emptyset$ for an arbitrary random real $R$. For almost all inputs, some clause is a disjunction of $\bot$, so that the entire conjunction is false.
		Thus $\Phi^A$ is cofinitely often 0, and hence computable, and so not random.
	\end{proof}
	\begin{table}
		\centering
		\begin{tabular}{c|c|c}
			
			Reducibility	&Subscript	&Connectives\\
			\hline
			truth table		&$tt$		&any\\
			bounded $tt$	&$btt$		&any\\
			$btt(1)$		&$btt(1)$	&$\{\lnot\}$\\
			linear			&$\ell$		&$\{+\}$\\
			positive		&$p$		&$\{\land,\lor\}$\\
			conjunctive		&$c$		&$\{\land\}$\\
			disjunctive		&$d$		&$\{\lor\}$\\
			many-one		&$m$		&none\\
			
		\end{tabular}
		\caption{Correspondences between reducibilities and sets in Post's Lattice. Here $+$ is addition mod 2 (also commonly written XOR).
			Note that while a $btt$ reduction can use any connectives, there is a bound $c$ on how many variables each $\sigma_{f(n)}$ can have, hence if $c=1$ the only connective available is $\lnot$.}
		\label{fig:table}
	\end{table}
	\begin{figure}
		\centerline{
		\xymatrix{
			&&d\ar[r]&p\ar[dr]\\
			1\ar[r]&m\ar[ur]\ar[r]\ar[dr]&c\ar[ur]&\ell\ar[r]&tt\ar[r]&T\\
			&&btt(1)\ar[ur]\ar[r]&btt\ar[ur]
		}}
		\caption[Implications between reducibility notions]{\cite{OdiRed} The relationships between reducibilities in \Cref{fig:table}, which themselves are between $\leq_1$ and $\leq_T$.
			Here $x\rightarrow y$ indicates that if two reals $A$ and $B$ enjoy $A\leq_x B$, then also $A\leq_y B$. 
		}
		\label{fig:Degtev}
	\end{figure}
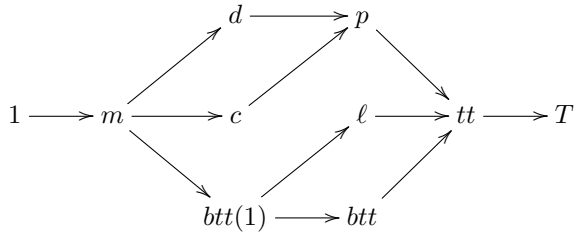
	
	\begin{remark}
		The proof of \Cref{positive} also applies to randomness over $3^{\omega}$ (and beyond).
		To see this, we consider the alphabet $\{0,1,2\}$ and let each $p(j)$ be an identity function and $\vee,\wedge$ be the maximum and minimum under the ordering $0<1<2$.
	\end{remark}
	\subsection{Linear Reducibility}
	\begin{definition}
		A truth-table reduction $\Phi$ is a \textbf{linear} reduction if
		each $\sigma_{f(n)}$ is of the form $\sigma_{f(n)} = \sum_{k=1}^{t_n} v_{f(n),k}$ or $\sigma_{f(n)} = 1+\sum_{k=1}^{t_n} v_{f(n),k} $ where addition is mod 2.
	\end{definition}
	\begin{theorem}\label{linear}
		$\MLR\not\le_{s,\ell}\Either(\MLR)$.
	\end{theorem}
	\begin{proof}
		We may assume that $\Phi$ infinitely often queries a bit that it has not queried before (else $\Phi^A$ is always computable).
		Without loss of generality, suppose $\Phi$ infinitely often queries an even bit it has not queried before.
		We construct $A$ in stages, beginning with $A_0 = \emptyset\oplus R$ for $R$ an arbitrary random real.

		For the infinitely many $n_i$ that query an unqueried even bit, let $v_i$ be the least such bit.
		Then at stage $s+1$, set $v_i=1$ if $\Phi^{A_s}(n_i) = 0$.
		Changing a single bit in a linear $\sigma_{f(n_i)}$ changes the output of $\sigma_{f(n_i)}$, so that $\Phi^{A}(n) = \Phi^{A_{s+1}}(n_i) = 1$.

		As these $n_i$ form a computable set, $\Phi^A$ fails to be immune, and so cannot be random.
	\end{proof}
	\subsection{Bounded Truth-Table Reducibility}

	\begin{definition}
		A truth-table reduction $\Phi$ is a \textbf{bounded truth-table} reduction if
		there is a $c$ such that there are most $c$ variables in each $\sigma_{f(n)}$ (in particular we say it is a \textbf{$btt(c)$} reduction).
	\end{definition}
	\begin{theorem}\label{btt}
		$\MLR\not\le_{s,btt}\Either(\MLR)$.
	\end{theorem}
	\begin{proof}
		Suppose that $\Phi$ is a $btt$-reduction from $\Either(\MLR)$ to $\MLR$ and let $c$ be its bound on the number of oracle bits queried.
		We proceed by induction on $c$, working to show that an $X=X_0\oplus X_1$ exists with $X_0$ or $X_1$ ML-random, for which $\Phi^X$ is not bi-immune.

		\noindent\emph{Base for the induction ($c=1$).} As $btt(1)$ reductions are linear, it is enough to appeal to \Cref{linear}.
		But as a warmup for what follows, we shall prove this case directly.
		Let $\Phi$ be a $btt(1)$ reduction. Here $\Phi^X(n)=f_n(X(q(n))$ where $f_n:\{0,1\}\to\{0,1\}$, $q:\omega\to\omega$ is computable, and $\{f_n\}_{n\in\omega}$ is computable.
		(If no bits are queried on input $n$, let $f_n$ be the appropriate constant function.)

		If for infinitely many $n$, $f_n$ is the constant function $1$ or $0$, and the claim is obvious.

		Instead, suppose $f_n$ is only constant finitely often, i.e. $f_n(x) = x$ or $f_n(x) = 1-x$ cofinitely often.
		Without loss of generality, there are infinitely many $n$ such that $q(n)$ is even. Let $X = \emptyset\oplus R$, where $R$ is an arbitrary ML-random set.

		As $X(q(n)) = 0$ and $f(x)$ is either identity or $1-x$ infinitely often,
		there is an infinite computable subset of either $\Phi^X$ or $\overline{\Phi^{X}}$ so $\Phi^X$ is not bi-immune.

		\noindent\emph{Induction step.} Assume the $c-1$ case, and consider a $btt(c)$ reduction $\Phi$.

		Now there are uniformly computable finite sets $Q(n)=\{q_1(n),\dots,q_{d_n}(n)\}$ and Boolean functions $f_n:\{0,1\}^{d_n}\to \{0,1\}$ such that for all $n$,
		$\Phi^X(n)=f_n(X({q_1(n)}),\dots,X(q_{d_n}(n)))$ and $d_n\le c$.

		Consider the greedy algorithm that tries to find a collection of pairwise disjoint $Q(n_i)$ as follows:
		\begin{itemize}
			\item[-] $n_0=0$.
			\item[-] $n_{i+1}$ is the least $n$ such that $Q(n)\cap\bigcup_{k\leq i}Q(n_k) = \emptyset$.
		\end{itemize}

		If this algorithm cannot find an infinite sequence, let $i$ be least such that $n_{i+1}$ is undefined, and define $H = \bigcup_{k\leq i}Q(n_k)$.
		It must be that for $n>n_i$ no intersection $Q(n)\cap H$ is empty.
		Thus there are finitely many bits that are in infinitely many of these intersections, and so are queried infinitely often.
		We will ``hard code" the bits of $H$ as $0$ in a new function $\hat{\Phi}$.

		To that end, define $\hat{Q}(n) = Q(n)\setminus H$, and let $\hat f$ be the function that outputs the same truth tables as $f$, but for all $n\in H$, $v_{n}$ is replaced with $\bot$.
		List the elements of $\hat{Q}$ in increasing order as $\{\hat{q}_1(n), \dots, \hat{q}_{e_n}(n)\}$.
		Now if $X\cap H = \emptyset$, any $q_i(n)\in H$ have $X(q_i(n)) = 0$, so that $\Phi^X = \hat\Phi^X$, as for every $n$,
		\[
			f(X(q_1(n)),\dots X(q_{d_n}(n)))= \hat{f_n}(X(\hat{q}_1(n)),\dots,X(\hat{q}_{e_n}(n)).
		\]

		As $Q$ and the $f_n$ are uniformly computable and $H$ is finite, $\hat{Q}$ and the $\hat{f}_n$ are also uniformly computable.
		As no intersection $Q(n)\cap H$ was empty, $e_n < d_n \leq c$. So $\hat{Q}$ and the $\hat{f}_n$ define a $btt(c-1)$-reduction.
		By the induction hypothesis, there is a real $A\in \Either(\MLR)$ such that $\hat\Phi^A$ is not random.
		$\Either(\MLR)$ is closed under finite differences (as $\MLR$ is), so the set $B = A\setminus H$ witnesses $\Phi^B = \hat\Phi^A$, and $\Phi^B$ is not random as desired.

		This leaves the case where the algorithm enumerates a sequence of pairwise disjoint $Q(n_i)$.

		Say that a collection of bits $C(n)\subseteq Q(n)$ can \emph{control} the computation $\Phi^X(n)$ if
		there is a way to assign the bits in $C_n$ so that $\Phi^X(n)$ is the same no matter what the other bits in $Q(n)$ are.
		For example, $(a \land b)\lor c$ can be controlled by $\{a, b\}$, by setting $a=b=1$.
		Note that if the bits in $C(n)$ are assigned appropriately, $\Phi^X(n)$ is the same regardless of what the rest of $X$ looks like.

		Suppose now that there are infinitely many $n_i$ such that some $C(n_i)$ containing only even bits controls $\Phi^X(n_i)$.
		Collect these $n_i$ into a set $E$. Let $X_1$ be an arbitrary ML-random set.
		As there are infinitely many $n_i$, and it is computable to determine whether an assignment of bits controls $\Phi^X(n)$, $E$ is an infinite computable set.
		For $n\in E$, we can assign the bits in $Q(n)$ to control $\Phi^X(n)$, as the $Q(n)$ are mutually disjoint. Now one of the sets
		\begin{align*}
			\{n\in E\mid \Phi^X(n) = 0\}&&\text{or}&&\{n\in E\mid \Phi^X(n) = 1\}
		\end{align*} is infinite. Both are computable, so in either case $\Phi^X$ is not bi-immune.

		Now suppose that cofinitely many of the $n_i$ cannot be controlled by their even bits. Here let $X_0$ be an arbitrary ML-random set.
		For sufficiently large $n_i$, no matter the values of the even bits in $Q(n_i)$, there is a way to assign the odd bits so that $\Phi^X(n_i) = 1$.
		By pairwise disjointness, we can assign the odd bits of $\bigcup Q(n_i)$ as needed to ensure this, and assign the rest of the odd bits of $X$ however we wish.
		Now the $n_i$ witness the failure of $\Phi^X$ to be immune. 
	\end{proof}

\section{Infinitely Many Hammers}
	It is worth considering direct sums with more than two summands.
	In this new setting, we first prove the analog of Theorem 2 of \cite{KHW} for more than two columns,
	before sketching the modifications necessary to prove analogues of \Cref{positive,linear,btt}.

	Recall that a real $A$ can be written as an infinite direct sum of columns $A^{[i]}$, $A = \bigoplus_{i=0}^\omega A^{[i]}$, where $A^{[i]}= \{n\mid \langle i, n\rangle \in A\}$ for a fixed computable bijection
	$\langle \cdot, \cdot\rangle:\omega^2\rightarrow\omega$.

	\begin{definition}
		For each $\C\subseteq 2^\omega$ and ordinal $\alpha\leq\omega$, define
		\begin{eqnarray*}
		\Some(\C, \alpha) &=& \left\{\bigoplus_{i=0}^\alpha A_i \in 2^\omega\middle| \exists i\ A_i\in\mathcal{C}\right\},
		\\
		\Many(\C) &=& \left\{\bigoplus_{i=0}^\omega A^{[i]} \in 2^\omega\middle| \exists^\infty i\ A^{[i]}\in\mathcal{C}\right\}.
		\end{eqnarray*}
	\end{definition}
\begin{remark}
As written, technically $\displaystyle\bigoplus^n A_i$ is not the same real as $\displaystyle\bigoplus^n A^{[i]}$, but the two are equivalent via a recursive bijection.
\end{remark}
	These represent different ways to generalize $\Either(\C)$ to the infinite setting:
	we may know that some possibly finite number of columns $A^{[i]}$ are in $\C$, or that infinitely many columns are in $\C$. If $\alpha=\omega$, these notions are $m$-equivalent, so we can restrict our attention to $\Some(\MLR, \alpha)$ without loss of generality:

\begin{theorem}[due to Reviewer 2 of \cite{CiE2022}]\label{thm:reviewer-2}
    $\Some(\C, \omega)\equiv_{s,m}\Many(\C)$.
\end{theorem}
\begin{proof}
    The $\leq_{s, m}$ direction follows from the inclusion $\Many(\C)\subseteq\Some(\C, \omega)$.
    
    For $\geq_{s, m}$,
    let $B\in\Some(\C,\omega)$ and define $A$ by:
    \[\langle \langle i,j\rangle,n\rangle \in A \iff \langle i,n\rangle\in B.\]
    Now $A\le_m B$ by definition. Notice that for all $i$ and $j$, $A^{[\langle i, j\rangle]} = B^{[i]}$.
    As some column $B^{[k]}$ is random, for all $j$, $A^{[\langle k, j\rangle]}\in\MLR$. Thus $A\in\Many(\C)$, so that $\Some(\C, \omega)\geq_{s, m}\Many(\C)$. 
\end{proof}
In the case of $\C = \MLR$, this can be strengthened to a 1-equivalence. 
\begin{lemma}[Corollary 6.9.6 in \cite{DaH}]\label{VanLamCor}
    If $A=\bigoplus_{i=0}^\omega A^{[i]}\in\MLR$, then for all $i$, $A^{[i]}\in\MLR$.
\end{lemma}
\begin{theorem}\label{SomeManyOne}
    $\Some(\MLR, \omega)\equiv_{s,1}\Many(\MLR)$.
\end{theorem}
\begin{proof}
    Again, $\leq_{s, 1}$ follows from subset inclusion.
    
    For $\geq_{s, 1}$, let $B\in\Some(\MLR,\omega)$ and define $A$ by:
    \[\langle \langle i, j\rangle, n \rangle \in A \iff \langle \langle n, j\rangle, i\rangle\in B.\]
    Again, $A\leq_1 B$ by definition. Now for all $i$ and $j$, $A^{[\langle i, j\rangle]} = \left(B^{[i]}\right)^{[j]}$.
    Some column $B^{[k]}$ is random, so by \Cref{VanLamCor}, its columns $\left(B^{[k]}\right)^{[j]}$ are random for all $j$. Thus for that $k$ and every $j$, $A^{[\langle k, j\rangle]}$ is random. Finally $A\in\Many(\C)$ and $\Some(\MLR, \omega)\geq_{s, 1}\Many(\MLR)$.
\end{proof}
\begin{remark}
\Cref{thm:reviewer-2} can be improved to $\equiv_1$ for any $\mathcal{C}\subseteq 2^\omega$ that satisfies the following: for all $D\in\Delta^0_1$ and $A\oplus_D B\in\mathcal{C}$, $A\in\mathcal{C}$. This is one direction of van Lambalgen's theorem \cite{vL} (the so-called `easy' direction --- see \cite{VanLam} for more discussion of this in the context of randomness notions).
\end{remark}

\subsection{Truth-Table Reducibility}
	Recall that a real $A$ is Martin-L\"of random iff there is a positive constant $c$ (the randomness deficiency) so that for any $n$, $K(A_i\upto n)\geq n-c$).
	Let $K_s(\sigma)$ be a computable, non-increasing approximation of $K(\sigma)$ at stages $s\in\omega$.

	\begin{theorem}\label{ttSome}
		For all ordinals $\alpha\leq\omega$, $\MLR\leq_{s,tt}  \Some(\MLR, \alpha)$.
	\end{theorem}
	\begin{proof}
		Given a set $A=\bigoplus_{i=0}^\alpha A_i$, we start by outputting bits from $A_0$,
		switching to the next $A_i$ whenever we notice that the smallest possible randomness deficiency increases.
		This constant $c$ depends on $s$ and changes at stage $s+1$ if
		\begin{equation}\label{Kolcond}
			(\exists n\le s+1)\quad K_{s+1}(A_i\upto n)<n-c_s.
		\end{equation}
		In detail, fix a map $\pi:\omega\rightarrow\alpha$ so that for all $y$, the preimage $\pi^{-1}(\{y\})$ is infinite. Let $n(0) = 0$, and if \Cref{Kolcond} occurs at stage $s$, set $n(s+1) = n(s) + 1$, otherwise $n(s+1) = n(s)$. Finally, define $A(s) = A_{\pi(n(s))}(s)$.
		
		As some $A_i$ is in $\MLR$, switching will only occur finitely often. So there is an stage $s$ such that for all larger $t$, $A(t) = A_i(t)$. Thus our output will have an infinite tail that is ML-random, and hence will itself be ML-random.

		To guarantee that this is a truth-table reduction, we must check that this procedure always halts, so that the reduction is total.\footnote{This is not the definition usually used in this section, but instead \Cref{tttotal}. As mentioned in \Cref{sec:omegace}, it is equivalent to \Cref{def:tt}.}
		But this is immediate, as \Cref{Kolcond} is computable for all $s\in\omega$ and $A_i\in 2^\omega$.
	\end{proof}
	
	\subsection{Positive Reducibility}
	We say that a variable is from a certain column if its index codes a location in that column, i.e. $n_k$ is from $A_i$ if $k = \langle i, n\rangle$ for some $n$.

	\begin{theorem}\label{positiveSome} For all $\alpha\leq\omega$, $\MLR\not\leq_{s, p} \Some(\MLR, \alpha)$.
	\end{theorem}
	\begin{proof}
		Let $\Phi^X$ be a positive reduction. Assume each $\sigma_{f(n)}$ is written in conjunctive normal form. We sketch the necessary changes to the proof of \Cref{positive}:

		\noindent\emph{Case 1:} There is an $i$ such that there are infinitely many $n$ such that every clause of $\sigma_f(n)$ contains a variable from $A_i$.

		Without loss of generality, let that column be $A_0=\omega$. The remaining $A_i$ can be arbitrary, as long as one of them is random.

		\noindent\emph{Case 2:} For all $i$, for almost all $n$, there is a clause in $\sigma_f(n)$ that contains no variables from $A_i$.

		In particular this holds for $i=0$, so let $A_0\in\MLR$ and the remaining $A_i=\emptyset$.
	\end{proof}

\subsection{Linear Reducibility}

	\begin{theorem}\label{linearSome}
	For all $\alpha\leq\omega$, $\MLR\not\leq_{s, \ell}\Some(\MLR, \alpha)$.
	\end{theorem}
	\begin{proof}
		We may assume that $\Phi$ infinitely often queries a bit it has not queried before (else $\Phi^A$ is always computable).
		If there is an $i$ such that $\Phi$ infinitely often queries a bit of $A_i$ it has not queried before,
		the stage construction from \Cref{linear} can be carried out with $A_i$ standing in for $A_0$, and some other $A_j\in\MLR$.

		That case always occurs for $\alpha<\omega$, but may not when $\alpha = \omega$.
		That is, it may the the case that $\Phi$ only queries finitely many bits of each $A_i$.
		Letting each $A_i$ be random, these bits may be set to $0$ without affecting the randomness of any given column, so we could set $A_0\in\MLR$ while other $A_i=\emptyset$.
	\end{proof}

\subsection{Bounded Truth-Table Reducibility}

	As $btt(1)$ reductions are linear, Theorem \ref{linearSome} provides the base case for induction arguments in the vein of \Cref{btt}.
	So we can focus our attention on the induction step:
	\begin{theorem}
		For all $\alpha\leq\omega$, $\MLR\not\leq_{s, btt}\Some(\MLR, \alpha)$.\footnote{This statement of the theorem corrects a typographical error in \cite{CiE2022}.}
	\end{theorem}
	\begin{proof}
		In the induction step, the case where the greedy algorithm fails is unchanged. Instead, consider the case where the algorithm enumerates a sequence of pairwise disjoint $Q(n_i)$.
		If there is a column $A_j$ such that there are infinitely many $n_i$ such that some $C(n_i)$ containing only bits from $A_j$ controls $\Phi^X(n)$, then we proceed as in \Cref{btt}:
		start with some other $A_k\in\MLR$ while the remaining columns are empty.
		We can then set the bits in each $Q(n_i)$ to control $\Phi^X(n_i)$ to guarantee that $\Phi^X$ is not bi-immune.
		This only changes bits in $A_j$, not $A_k$, so the final $A\in\Some(\MLR,\alpha)$.

		This leaves the case where for each $A_j$, cofinitely many of the $n_i$ cannot be controlled by their bits in $A_j$.
		Here put $A_0\in\MLR$ and assign bits to the other columns as in \Cref{btt}.
	\end{proof}